\documentclass[11pt]{amsart}
\usepackage{hyperref,pb-diagram,amssymb,epic,eepic,verbatim,graphicx,graphics,epsfig,psfrag}
\hypersetup{colorlinks}
\usepackage{color}
\definecolor{darkred}{rgb}{0.5,0,0}
\definecolor{darkgreen}{rgb}{0,0.5,0}
\definecolor{darkblue}{rgb}{0,0,0.5}
\hypersetup{ colorlinks,
linkcolor=darkblue,
filecolor=darkgreen,
urlcolor=darkred,
citecolor=darkblue }

\usepackage{pb-diagram,amssymb,epic,eepic,verbatim}
\usepackage{braket, hyperref,graphics,epsfig,psfrag}
\usepackage{tikz-cd}
\newtheorem{theorem}{Theorem}[section]

\newtheorem{corollary}[theorem]{Corollary}

\newtheorem{proposition}[theorem]{Proposition}
\newtheorem{lemma}[theorem]{Lemma}
\theoremstyle{definition}
\newtheorem{definition}[theorem]{Definition}
\newtheorem{notation}[theorem]{Notation}
\newtheorem{lem}[theorem]{}
\newtheorem{remark}[theorem]{Remark}
\newtheorem{example}[theorem]{Example}

\newcommand{\blem}{\begin{lem} \rm}
\newcommand{\elem}{\end{lem}}

\newcommand\M{\mathcal{M}}
\newcommand\D{\mathcal{D}}

\renewcommand\M{\mathcal{M}}
\renewcommand\D{\mathcal{D}}

\newcommand{\T}{\mathcal{T}}
\newcommand{\J}{\mathcal{J}}

\newcommand{\R}{\mathbb{R}}

\newcommand{\C}{\mathbb{C}}
\newcommand{\cC}{\mathcal{C}}
\newcommand{\cT}{\mathcal{T}}

\newcommand{\sS}{\mathcal{S}}

\newcommand{\Z}{\mathbb{Z}}
\newcommand{\Q}{\mathbb{Q}}

\newcommand{\ddt}{\frac{d}{dt}}

\renewcommand{\P}{\mathbb{P}}
\newcommand{\PP}{\mathcal{P}}

\newcommand\lie[1]{\mathfrak{#1}}

\newcommand{\g}{\lie{g}}

\renewcommand{\t}{\lie{t}}

\newcommand{\on}{\operatorname}

\newcommand{\true}{\on{min}}
\newcommand{\fake}{\on{fake}}

\newcommand{\sing}{\on{sing}}

\newcommand{\dual}{\vee}

\newcommand{\Hom}{ \on{Hom}}

\newcommand{\Vol}{  \on{Vol}}

\newcommand{\Spec}{  \on{Spec}}

\newcommand{\codim}{\on{codim}}

\newcommand\dirac{/\kern-1.2ex\partial} 
\newcommand\qu{/\kern-.7ex/} 
\newcommand\lqu{\backslash \kern-.7ex \backslash} 

\newcommand\dr{r_+ \kern-.7ex - \kern-.7ex r_-}
 
\newcommand{\labell}\label

\renewcommand{\d}{{\on{d}}}
\newcommand{\ovl}{\overline}

\newcommand\Phinv{\Phi^{-1}}

\newcommand\Lam{\Lambda}

\newcommand\eps{\epsilon}

\newcommand{\f}{\frac}
\newcommand{\lan}{\langle}
\newcommand{\ran}{\rangle}
\newcommand{\hh}{{\f{1}{2}}}

\newcommand{\ti}{\tilde}

\newcommand\pt{\on{pt}}

\newcommand\Sym{\on{Sym}}

\newcommand\cI{\mathcal{I}}

\renewcommand{\ss}{{\on{ss}}}

\newcommand\Map{\on{Map}}

\newcommand\ev{\on{ev}}
\newcommand\Eul{\on{Eul}}

\newcommand\mO{\mathcal{O}}

\newcommand\bdefn{\begin{definition}}
\newcommand\edefn{\end{definition}}
\newcommand\bea{\begin{eqnarray*}}
\newcommand\eea{\end{eqnarray*}}
\newcommand\bcv{\left[ \begin{array}{r} }
\newcommand\ecv{\end{array} \right] }

\newcommand\bma{\left[ \begin{array}{l} }
\newcommand\ema{\end{array} \right]}
\newcommand\ben{\begin{enumerate}}
\newcommand\een{\end{enumerate}}
\newcommand\beq{\begin{equation}}
\newcommand\eeq{\end{equation}}
\newcommand\bex{\begin{example}}
\newcommand\bsj{\left\{ \begin{array}{rrr} }
\newcommand\esj{\end{array} \right\}}

\newcommand\eex{\end{example}}

\newcommand\sx{*\kern-.5ex_X}

\newcommand{\Jac}{\on{Jac}}
\newcommand{\Crit}{\on{Crit}}

\newcommand{\hJac}{{{\Jac}_+}}

\newcommand{\bA}{\mathbb{A}}

\def\mathunderaccent#1{\let\theaccent#1\mathpalette\putaccentunder}
\def\putaccentunder#1#2{\oalign{$#1#2$\crcr\hidewidth \vbox
to.2ex{\hbox{$#1\theaccent{}$}\vss}\hidewidth}}

\begin{document}

\parskip 0in 

\title[Quantum cohomology and toric minimal model programs]{Quantum cohomology and \\ toric minimal model programs}

\author{Eduardo Gonz\'alez} 


\address{
Department of Mathematics
University of Massachusetts Boston
100 William T. Morrissey Boulevard
Boston, MA 02125}
  \email{eduardo@math.umb.edu}

\author{Chris T. Woodward}

\address{Mathematics-Hill Center,
Rutgers University, 110 Frelinghuysen Road, Piscataway, NJ 08854-8019,
U.S.A.}  \email{woodwardc@gmail.com}

\thanks{Partially supported by grants DMS1104670 and DMS0904358.  We
  thank the Centre de Recerca Matem\`atica, Barcelona, for
  hospitality during preparation of this paper.}

\thanks{Appeared in {\em Adv. Math}. Volume 353, 7 September 2019,
  Pages 591--646.  doi.org/10.1016/j.aim.2019.07.004}

\begin{abstract}
  We give a quantum version of the Danilov-Jurkiewicz presentation of
  the cohomology of a compact toric orbifold with projective coarse
  moduli space.  More precisely, we construct a canonical isomorphism
  from a formal version of the Batyrev ring from \cite{bat:qcr} to the
  quantum orbifold cohomology at a canonical bulk deformation.  This
  isomorphism generalizes results of Givental \cite{gi:eq}, Iritani
  \cite{iritani:conv} and Fukaya-Oh-Ohta-Ono \cite{fooo:tms} for toric
  manifolds and Coates-Lee-Corti-Tseng \cite{coates:wps} for weighted
  projective spaces.  The proof uses a quantum version of Kirwan
  surjectivity (Theorem \ref{weaktoric2} below) and an equality of
  dimensions (Theorem \ref{equality} below) deduced using a toric
  minimal model program (tmmp).  We show that there is a natural
  decomposition of the quantum cohomology where summands correspond to
  singularities in the tmmp, each of which gives rise to a collection
  of Hamiltonian non-displaceable Lagrangian tori.
\end{abstract} 

 \maketitle

\tableofcontents
\parskip .1in

\section{Introduction}

According to results of Danilov and Jurkiewicz
\cite{danilov:toric,jurk:toric1,jurk:toric2}, the rational cohomology
ring of a complete rationally-smooth toric variety is the quotient of
a polynomial ring generated by prime invariant divisors by the
Stanley-Reisner ideal.  In addition to relations corresponding to
linear equivalence of invariant divisors, there are higher degree
relations corresponding to collections of divisors whose intersection
is empty.

One can reformulate this presentation of the cohomology ring in terms
of equivariant cohomology as follows.  Let $G$ be a complex reductive
group acting on a smooth polarized projective variety $X$.  If the
action on the semistable locus $X^{\ss}$ is locally free then the
geometric invariant theory (git) quotient $X \qu G = X^{\ss}/G $, by
which we mean the stack-theoretic quotient of the semistable locus by
the group action, is a smooth proper Deligne-Mumford stack with
projective coarse moduli space.  A result of Kirwan \cite{ki:coh} says
that the natural map $H_G(X,\Q) \to H(X \qu G,\Q)$ is surjective.
Under suitable properness assumptions the same holds for
quasi-projective $X$.  

In particular, let $G$ be a torus acting on a finite-dimensional
vector space $X$ with weights contained in an open half-space.  The
quotient $X \qu G$ is a smooth proper Deligne-Mumford toric stack as
in Borisov-Chen-Smith \cite{bcs:tdms} and any such toric stack with
projective coarse moduli space arises in this way.  The equivariant
cohomology $H_G(X)$ may be identified with the ring of polynomial
functions on $\g$ and each weight maps to a divisor class in
$H(X \qu G)$ under the Kirwan map.  The Stanley-Reisner ideal $SR_X^G$
is precisely the kernel of the Kirwan map.  For example, if
$G = \C^\times$ acts by scalar multiplication on $X = \C^k$, then
$H_G(X) = \Q[\xi]$ is a polynomial ring in a single generator $\xi$,
the git quotient is $X \qu G = \P^{k-1}$, and the intersection of the
$k$ prime invariant divisors is empty.  The Stanley-Reisner ideal is
the ideal $\lan \xi^k \ran$ generated by $\xi^k$.  This gives the
standard description of the cohomology ring of projective space
$H(\P^{k-1}) = H_G(X)/ SR_X^G = \Q[\xi]/ \lan \xi^k\ran .$

In this paper we give a similar presentation of the quantum cohomology
of compact toric orbifolds with projective coarse moduli spaces, via
the quantum version of the Kirwan map introduced in
\cite{qkirwan1,qkirwan2,qkirwan3}.  The results here generalize those
of Batyrev \cite{bat:qcr}, Givental \cite{gi:eq}, Iritani
\cite{iri:gmt,iri:integral,iritani:conv}, and Fukaya-Oh-Ohta-Ono
\cite{fooo:tms}, who use results of McDuff-Tolman \cite{mct:top}.  In
particular, Iritani \cite{iritani:conv} computed the quantum
cohomology of toric manifolds using localization arguments for toric
varieties that appear as certain complete intersections, while Fukaya
et al \cite{fooo:tms} gave a computation using open-closed
Gromov-Witten invariants defined via Kuranishi structures.  The
orbifold quantum cohomology of weighted projective spaces is computed
in Coates-Lee-Corti-Tseng \cite{coates:wps}.  After the first version
of this manuscript appeared a mirror theorem for toric stacks was
proved by Coates, Corti, Iritani, and Tseng \cite{coates:mirror} and
applied to give a Batyrev-style presentation in \cite[Theorem
5.13]{coates:hodge}.

A novel feature of the approach here is the appearance of minimal
model programs, which are used to prove injectivity of the quantum
Kirwan map modulo the quantum Stanley-Reisner ideal.  The critical
values of the Givental-Hori-Vafa potential acquire a natural geometric
meaning in our approach: their logarithms are the transition times in
the minimal model program, see Theorem \ref{jump} below, and the
dimension of the orbifold cohomology and the logarithm of the lowest
eigenvalue of quantum multiplication by the first Chern class decrease
under each transition.  We also obtain a more conceptual understanding
of the appearance of open families of non-displaceable Lagrangians in
toric orbifolds, as a consequence of the existence of infinitely many
minimal model programs, see Remark \ref{induced}.

We introduce the following notations.

\begin{notation} \label{notation}
\begin{enumerate} 
\item {\rm (Novikov coefficients)} Let $\Lambda$ denote the {\em
  universal Novikov field} of formal power series of $q$ with rational
  exponents
\[ \Lambda = \Set{ \sum_\rho c_\rho q^\rho \ |  \begin{array}{l}
                                                  c_\rho \in \C, \rho
                                                  \in \Q \\  \forall e > 0,
\# \{ \rho | c_\rho < e \} < \infty  \end{array} }  .\]
We denote by $\Lambda_0 \subset \Lambda$ the subring with only
non-negative powers of $q$.
\item {\rm (Equivariant quantum cohomology)} Let
\[QH_G(X) := H_G(X,\C)
  \otimes_\C \Lambda\]
  denote the (ungraded) equivariant quantum cohomology of $X$.  We
  denote by
$QH_G(X,\Q) := H_G(X,\Q) \otimes_\Q \Lambda$ 
the subspace with rational coefficients.  Equivariant enumeration of
stable maps to $X$ defines a family of products
\[ \star_\alpha : T_\alpha QH_G(X,\Q)^2 \to T_\alpha QH_G(X,\Q) \]
forming (part of) the structure of a {\em Frobenius manifold} on
$QH_G(X,\Q)$ \cite{gi:eq} for $\alpha$ in a formal neighborhood of a
symplectic class $\omega \in H_2^G(X,\Q)$.  Explicitly the product
$ \beta \star_{\alpha + \omega} \gamma$ is defined by 
\begin{equation} \label{star}
 \lan \beta \star_{\alpha + \omega} \gamma, \delta \ran = \sum_{d \in
   H_2(X,\Z), n \ge 0} \frac{q^{\lan d, \omega \ran}}{n!}
 \int_{[\ovl{\M}_{0,n+3}(X,d)_G]} \ev^* (\alpha,\ldots, \alpha, \beta, \gamma,
 \delta^\dual) \end{equation}
where the integral denotes push-forward to $BG$ using the equivariant
virtual fundamental class described in \cite{gr:loc}.
\item {\rm (Inertia stacks)} The {\em inertia stack} of $X \qu G$ is
  \begin{eqnarray*} I_{X \qu G} &=& \bigcup_{ r > 0}
    \Hom^{\on{rep}}(\P(r), X \qu G ) = \bigcup_{[g]} X^{g,\ss} / Z_g
    .\end{eqnarray*}
In the first union, $\Hom^{\on{rep}}(\P(r), \cdot)$ denotes
representable morphisms from $\P(r) = B\Z_r$ and the second union is
over conjugacy classes $[g]$ of elements $g \in G$, with $Z_g \subset
G$ the centralizer of $g$ and $X^{g,\ss}$ the intersection of the
semistable locus $X^{\ss}$ with the fixed point set 
\[X^g := \{ x \in X
\ | \ gx = x \} .\]  
The {\em rigidified inertia stack} is
\begin{eqnarray*} 
 \ovl{I}_{X \qu G} &=& \bigcup_{ r > 0} \Hom^{\on{rep}}(\P(r),
 X/G)/\P(r) = \bigcup_{[g]} X^{g,\ss} / (Z_g/ \lan g
 \ran) \end{eqnarray*}
where $\lan g \ran$ denotes the subgroup generated by $g$, as in
Abramovich-Graber-Vistoli \cite{agv:gw}, Chen-Ruan \cite{cr:orb}.
\item {\rm (Orbifold quantum cohomology of a git quotient)} Let
\[ QH(X \qu G) := H(I_{X \qu G},\C) \otimes \Lambda \] 
denote the orbifold quantum cohomology of $X \qu G$, or $QH(X \qu
G,\Q)$ the version with rational coefficients.  Enumeration of twisted
stable maps to $X \qu G$ (representable maps from orbifold curves to
$X \qu G$) defines a Frobenius manifold structure on $QH(X \qu G)$
\cite{agv:gw}, \cite{cr:orb} given by a family of products 
\[ \star_\alpha : T_\alpha QH(X \qu G,\Q)^2 \to T_\alpha QH(X \qu
G,\Q) .\] 
These products are defined in a formal neighborhood of an equivariant
symplectic class $\omega \in H^2(X \qu G, \Q)$ by
\begin{equation} \label{star2}
\lan \beta \star_{\omega + \alpha} \gamma, \delta \ran := \sum_{ 
\substack{ d
   \in H_2(X \qu G,\Q) \\ n \ge 0}} \frac{q^{\lan d , \omega \ran}}{n!}
 \int_{[\ovl{\M}_{0,n+3}(X \qu G,d)]} \ev^*(
 \alpha,\ldots,\alpha,\beta,\gamma,\delta^\dual) \end{equation}
for $\alpha, \beta, \gamma \in H(I_{X \qu G})$, extended by linearity
over $\Lambda$.  The pairing on the left-hand-side is a certain
re-scaled Poincar\'e pairing on the inertia stack $I_{X \qu G}$, see
\cite{agv:gw}.
\end{enumerate}
\end{notation} 

\begin{example} \label{qhex}
 To connect with the notation in 
\cite{agv:gw}, \cite{cr:orb} (where one works with different Novikov
fields) consider the following examples. 
\begin{enumerate} 
\item {\rm (Stacky half-point)} Let $G = \C^\times$ act on $X = \C$ with
  weight two so that $X \qu G = \P(2)$. The inertia stack
  $I_{X \qu G}$ is the union of two copies of $\P(2)$ corresponding to
  the elements $\pm 1$ of $\Z_2$.  Thus
\[QH(X \qu G) = \Lambda \oplus \Lambda \theta_-\]
  the sum of two copies of $\Lambda$, where $\theta_-$ is the
  additive generator of the twisted sector.  Representable morphisms
  from a stacky curve $C$ to $X \qu G = \P(2)$ correspond to double
  covers of the coarse moduli space $C$, with ramification at the
  stacky points.  Since there is a unique double cover of the
  projective line with two ramification points (up to isomorphism)
  multiplication is given by
$\theta_- \star_\omega \theta_- = 1 .$
\item {\rm (Teardrop orbifold)} \label{teardrop} Suppose that $G =
  \C^\times$ acts on $X = \C^2$ with weights $1,2$.  Then $X \qu G =
  \P(1,2)$ is a weighted projective line, $QH_G(X) \cong
  \Lambda[\xi]$ is a polynomial ring in a single generator, while
\[QH(X \qu G) =
  \Lambda \oplus \Lambda \theta_+ \oplus \Lambda \theta_-\] 
  where $\theta_+$ is the point class in
  $H(X \qu G) \subset H(I_{X \qu G})$ and $\theta_-$ is the class of
  the fixed point set $X^{-1}/ \lan -1 \ran = \P(2)$ in the twisted
  sector.  Identify $H^G_2(X,\Q) \cong \Q$ corresponding to the
  dual of the Euler class of the representation with weight one. The
  fundamental class in $H_2(X \qu G,\Q) \cong H_2^G(X,\Q)$ then maps
  to $1/2$.  The moduli space of twisted stable maps
  $u: C \to \P(1,2)$ of genus and class zero is either isomorphic to
  $\P(1,2)$ for no stacky points in the domain $C$, or isomorphic to
  $\P(2)$, for two stacky points in the domain $C$.  Furthermore there
  is a unique (up to isomorphism) homology class $1/2$ twisted map
  with two smooth marked points and one stacky marked point with
  $\Z_2$ automorphism group.  It follows that if the
  symplectic class $\omega$ has area $1/2$ on the fundamental
  class of $\P(1,2)$ then the quantum product is defined by
\[\theta_+ \star_\omega \theta_+ = q^{1/2} \theta_-/2, \quad \theta_-
\star_\omega \theta_+ = q^{1/2}/2, \quad \theta_- \star_\omega \theta_- =
\theta_+ .\]
Thus after inverting $q^{1/2}$, the orbifold quantum cohomology is
generated by $\theta_+$ with the relation $\theta_+^3 = q/4$.
\end{enumerate} 
\end{example} 

\begin{remark} \label{alt} {\rm (Alternative power series rings)} 
Some confusion may be caused by the multitude of formal power series
rings that one can work over; unfortunately almost every set of
authors has a different convention.
\begin{enumerate}
\item The equivariant quantum cohomology $QH_G(X)$ can be defined over
  the larger {\em equivariant Novikov field}
  $\Lambda_X^G \subset \Map(H_2^G(X,\Z),\Q)$ consisting of infinite
  sums $\sum_{i=1}^\infty c_i q^{d_i}$ with
  $\lan d_i, \omega \ran \to \infty$, where $q^{d_i}$ is the delta
  function at $d_i \in H_2^G(X,\Z)$.  Similarly, the quantum
  cohomology of the quotient $QH(X \qu G)$ can be defined over the
  Novikov field $\Lambda_{X \qu G} \subset \Map(H(X \qu G,\Q),\Q)$
  consisting of infinite sums $\sum_{i=1}^\infty c_i q^{d_i}$ with
  $\lan d_i, \omega \ran \to \infty$, where $q^{d_i}$ is the delta
  function at $d_i \in H_2(X \qu G,\Q)$.  The advantage of these rings
  is that the equivariant quantum cohomology $QH_G(X)$ becomes
  $\Z$-graded.
\item $QH_G(X)$ is also defined over the {\em universal Novikov ring}
  $\Lambda_0$.  If $\omega$ is integral, then $QH_G(X)$ is defined
  over $\Q[[q]]$.  Similarly, $QH(X \qu G)$ is defined over the
  Novikov ring $\Lambda_0$, and if $\omega$ is integral, over
  $\Q[[q^{1/n}]]$ for $n$ equal to the least common multiple of the
  orders of the automorphism groups in $X \qu G$.  However, it is
  convenient to work over the field $\Lambda$.  Invariance under
  Hamiltonian perturbation only holds for Floer/quantum cohomology
  over the Novikov field $\Lambda$, and so working over $\Lambda$ is
  more natural for the purposes of symplectic geometry.
\item Unfortunately, $\Lambda$ and $\Lambda_0$ are not finitely
  generated \label{changenonnoetherian} over $\C$ and so some care is
  required when talking about intersection multiplicities.  In
  practice, when we wish to talk about intersection multiplicities we
  assume that the symplectic form is integral in which case our
  algebras are defined over $\C[q,q^{-1}]$.
\item In algebraic geometry, one often uses the monoid-algebra of
  effective curve classes, but we prefer Novikov fields because of the
  better invariance properties. In fact, the cone of effective curve
  classes is not any more explicit than working over the Novikov field
  since it is the classes of {\em connected curves} that appear in the
  Gromov-Witten potentials, and these are rather hard to determine.
\end{enumerate} 
\end{remark} 

In \cite{qkirwan1,qkirwan2,qkirwan3} the second author studied the relationship between
$QH_G(X)$ and $QH(X \qu G)$ given by virtual enumeration of affine
gauged maps, called the {\em quantum Kirwan map}.  An $n$-marked {\em
  affine gauged map} is a representable morphism from a weighted
projective line $\P(1,r)$ for some $r > 0$ to the quotient stack $X/G$
mapping $\P(r) \subset \P(1,r)$ to the semistable locus $X \qu G$.
Some of the results of \cite{qkirwan1,qkirwan2,qkirwan3} are:
\begin{theorem} {\rm (Definition and properties of the quantum Kirwan map)}  
\begin{enumerate} 
\item The stack $\M_{n,1}^G(\bA,X,d)$ of $n$-marked affine gauged maps
  of class $d \in H_2^G(X,\Q)$ has a natural compactification
  $\ovl{\M}_{n,1}^G(\bA,X,d)$.  Denote by $\ev,\ev_\infty$ the
  evaluation maps
\[ \begin{diagram} 
\node{} \node{\ovl{\M}_{n,1}^G(\bA,X)} \arrow{sw,t}{\ev} \arrow{se,t}{\ev_\infty} 
\node{} \\ 
\node{ (X/G)^n } \node{} \node{ \ovl{I}_{X \qu G}} \end{diagram}
\]
and $\ev_d$, $\ev_{d,\infty}$ their restrictions to maps of class $d$.
The moduli stack $\ovl{\M}_{n,1}^G(\bA,X,d)$ has a perfect relative
obstruction theory over $\ovl{\M}_{n,1}(\bA)$ (the case of $X$ and $G$
trivial) where $\ovl{\M}_{n,1}(\bA)$ is the complexification of
Stasheff's multiplihedron.
\item 
For any $n \ge 0$, the map defined by virtual enumeration of stable
$n$-marked affine gauged maps
\begin{multline}
 \kappa_X^{G,n}: QH_G(X,\Q) \to QH(X \qu G,\Q) \\ \alpha \mapsto
 \sum_{d \in H_2^G(X,\Q)} q^{\lan d, \omega \ran} \ev_{d,\infty,*} \ev_d^*
 (\alpha, \ldots, \alpha) \end{multline}
is well-defined.
\item The sum 
\[ \kappa_X^G:\  QH_G(X,\Q) \to QH(X \qu G,\Q),  \quad \alpha \mapsto \sum_{n \ge 0}
 \frac{ \kappa_X^{G,n}(\alpha)}{n!} \]
 defines a formal map from $QH_G(X,\Q)$ to $QH(X \qu G,\Q)$ in a
 neighborhood of the symplectic class $\omega \in H_G^2(X,\Q)$ with
 the property that each linearization
\[ D_\alpha \kappa_X^G: T_\alpha QH_G(X,\Q) \to T_{\kappa_X^G(\alpha)}
QH(X \qu G,\Q) \]
is a $\star$-homomorphism with respect to the quantum products.
\end{enumerate}
\end{theorem} 

By analogy with the classical case one hopes to obtain a presentation
of the quantum cohomology algebra $T_{\kappa_X^G(\alpha)} QH(X \qu
G,\Q)$ by showing that $D_\alpha \kappa_X^G$ is surjective and
computing its kernel.  This hope leads to the following strong and
weak quantum version of Kirwan surjectivity.  In the strong form, one
might hope that $\kappa_X^G$ has infinite radius of convergence,
$\kappa_X^G$ is surjective, and $D_\alpha \kappa_X^G$ is surjective
for any $\alpha \in QH_G(X,\Q)$.  More modestly, one might hope that
$D_\alpha \kappa_X^G$ is surjective for $\alpha$ in a formal
neighborhood of a rational symplectic class $\omega \in H_2^G(X,\Q)$.

We now specialize to the toric case.  Suppose that $G$ is a complex
torus with Lie algebra $\g$ acting on a finite-dimensional complex
vector space $X$.  

\begin{notation} 
\begin{enumerate} 
\item {\rm (Weights)} Let $X_1,\ldots, X_k \subset X$ be the weight
  spaces of $X$ where $\dim(X_j) = 1$ and $G$ acts on $X_j$ with
  weight $\mu_j \in \g^\dual$ in the sense that for $x \in X_j$ and
  $\xi \in \g$ we have
  $\exp(\xi) x = \exp( i \lan \xi, \mu_j \ran), j = 1,\ldots, k$.  We
  assume that the weights $\mu_j \in \g^\dual$ are contained in an
  open half-space, that is, for some $ \nu \in \g$ we have
  $ \lan \nu, \mu_i \ran \in \R_{> 0}, \ i = 1,\ldots, k $.  We also
  assume that the weights $\mu_i$ span $\g^\dual$, so that $G$ acts
  generically locally free on $X$.
\item {\rm (Polarization and semistable locus)} We assume that $X$ is
  equipped with a polarization, that is, an ample $G$-line bundle $L
  \to X$, which we may allow to be rational, that is, an integer root
  of an honest $G$-line bundle.  Let $\omega \in \g^\dual_\Q$ be the
  vector representing the first Chern class of the polarization $
  c_1^G(L) \in H_2^G(X,\Q)$ under the isomorphism $\g^\dual_\Q \cong
  H^2_G(X,\Q)$.  The point $\omega$ determines a rational polarization
  on $X$ with semistable locus given as follows.  
Let  %
\begin{equation} \label{unstable}
\cI(\omega) = \left\{ I\subset \{1,\dots, k\} \ | \ \omega \notin
\sum_{i\in I} \R_{\geq 0} \mu_i \right\} \end{equation} 
be the set of subsets so that $\omega$ is not in the span of the
corresponding weights.  Let $X^I$ be the intersection of coordinate
hyperplanes
\[ X^I=\left\{ (x_1,\ldots, x_k) | x_i =0, \ \forall i \notin I \right\}
.\]
Then
\[ X^{\ss} = X \backslash \bigcup_{I \in \cI(\omega)} X^I .\] 
The stable=semistable condition assumption translates to the condition
for each $I \notin \cI(\omega)$ the weights $\mu_i, i \in I$ span
$\g^\dual$.  In this case the quotient $X \qu G = X^{\ss}/G$ is then a
smooth (possibly empty) proper Deligne-Mumford stack.  We suppose that
$X \qu G$ is non-empty.  
\item {\rm (Quantum Stanley-Reisner ideal)} The {\em quantum
  Stanley-Reisner ideal} is 
\[ QSR_{X,G}(\alpha) := \lan QSR_{X,G}(d,\alpha), d \in H_2^G(X,\Z)
  \ran  \subset QH_G(X,\Q) \]
where 
\[ QSR_{X,G}(d,\alpha) := \prod_{\lan \mu_j, d \ran \ge 0} \mu_j^{\lan
  \mu_j, d \ran} - q^{\lan d, \alpha \ran} \prod_{\lan \mu_j, d \ran
  \leq 0} \mu_j^{-\lan \mu_j, d \ran} .\]
If $\alpha$ is the given symplectic class $\omega$, we write
$QSR_{X,G} := QSR_{X,G}(\omega)$.  The quotient $T_\omega QH_G(X,\Q)/
QSR_{X,G}$ is the {\em quantum Stanley-Reisner a.k.a Batyrev ring}.
\end{enumerate}
\end{notation} 

\begin{example} 
\begin{enumerate} 
\item {\rm (Batyrev ring for projective space)} Let $ G = \C^\times$
  act on $X = \C^k$ by scalar multiplication.  All weights
  $\mu_1,\ldots,\mu_k$ are equal to $1 \in \g_\Z^\dual \cong \Z$ and
  the polarization vector
  $\omega = 1 \in \g^\dual_\Q \cong H_G^2(X,\Q)$.  There is a unique
  subset $I = \emptyset$ in $\cI(\omega)$ and
  $X^I = \{ 0 \} \subset X$.  Thus the semistable locus is
  $X^{\ss} = X - X^\emptyset = X - \{ 0 \}$ and the git quotient is
  $X \qu G = X^{\ss} / G = \P^{k-1}$.  The quantum Stanley-Reisner
  ideal is generated by the single element $QSR_{X,G}(1) = \xi^k - q$.
  The Batyrev ring is $\Lambda[\xi]/\lan \xi^k - q \ran $.
\item 
\label{teardrop2}  
{\rm (Batyrev ring for the teardrop orbifold)} Continuing Example 
\ref{qhex} \eqref{teardrop}, suppose that $G = \C^\times$ acts on $X = \C^2$ with 
weights $1,2$ so that $X \qu G = \P(1,2)$ is a weighted projective 
line.  The Batyrev ring is $\Lambda[\xi]/ \lan (\xi)(2\xi)^2 - q
\ran$.
\item {\rm (Batyrev ring for the $B\Z_2$)} Continuing Example
  \ref{qhex} \eqref{teardrop}, suppose that $G = \C^\times$ acts on
  $X = \C$ with weights $2$ so that $X \qu G = \P(2) \cong B\Z_2$.
  The Batyrev ring is $\Lambda[\xi]/ \lan (2\xi)^2 - q \ran$.
  After specializing $q$, the Batyrev ring is isomorphic to the group
  ring of $\Z_2$.
\label{BZ2}  
\end{enumerate} 
\end{example}

Our main result says that Batyrev's original suggestion \cite{bat:qcr}
for the quantum cohomology is true, after passing to a suitable formal
version of the equivariant cohomology and ``quantizing'' the divisor
classes:

\begin{theorem} \label{main0} For a suitable formal version
  $\widehat{QH}_G(X)$ of the equivariant quantum cohomology $QH_G(X)$
  (see Section \ref{formal}) the linearized quantum Kirwan map
  $D_\omega \kappa_X^G$ induces an isomorphism
  \[T_\alpha \widehat{QH}_G(X,\Q)/ \widehat{QSR}_{X,G}(\omega) \to
  T_{\kappa_X^G(\omega)} QH(X \qu G,\Q).\]
  at the tangent space to the rational symplectic class
  $\omega \in H_G^2(X,\Q)$.
\end{theorem} 

\begin{remark} 
\begin{enumerate}
\item Many earlier cases of this theorem were known.  Batyrev
  \cite{bat:qcr} proved a similar presentation in the case of convex
  toric manifolds, that is, in the case that the deformations of any
  stable map are un-obstructed.  In the semi-Fano case (that is,
  $ c_1(X \qu G)$ is non-negative on any curve class) a presentation
  was given by Givental \cite{giv:tmp}.  For non-weak-Fano toric
  manifolds, Iritani \cite[5.11]{iritani:conv} gave an isomorphism
  with the Batyrev ring, see also Brown \cite{brown:gw}.  From the
  symplectic point of view a presentation for the quantum cohomology
  of toric manifolds was given in Fukaya et al \cite{fooo:tms}, using
  results of McDuff-Tolman \cite{mct:top} on the Seidel
  representation.  The latter approach uses open-closed Gromov-Witten
  invariants to define a potential counting holomorphic disks whose
  leading order terms are the potential above.  The quantum
  Stanley-Reisner relations were proved by Coates, Corti, Iritani, and
  Tseng \cite[Theorem 5.13]{coates:hodge}, see also Woodward
  \cite{qkirwan1,qkirwan2,qkirwan3}, in papers that appeared after the
  first version of this manuscript.  That these relations generate the
  ideal was expected for some time, see Iritani \cite{iri:integral}.
  Thus the main content of this paper is that these relations suffice.
  A quantization of the Borisov-Chen-Smith presentation of the
  orbifold cohomology \cite{bcs:tdms} was given in Tseng-Wang
  \cite{tw:orb}.  The latter is {\em not} a presentation in terms of
  divisor classes; for example, for weighted projective spaces the
  typical number of generators is much larger than one, while the
  Batyrev ring has a single generator.
\item For the result above to hold the quantum cohomology must be
  defined over the Novikov {\em field}, or at least, that a suitable
  rational power of the formal parameter $q$ has been inverted: over a
  polynomial ring such as $\C[q]$, one does not obtain an surjection
  because certain elements in twisted sectors are not contained in the
  image for $q = 0$.  Thus one sees a Batyrev presentation of the
  quantum cohomology only for non-zero $q$.  The necessity of
  corrections to Batyrev's original conjecture, which involved the
  divisor classes as generators, was noted in Cox-Katz \cite[Example
  11.2.5.2]{ck:ms} for the second Hirzebruch surface and Spielberg
  \cite{sp:gw} for a toric three-fold.  The fact that the change of
  coordinates restores the original presentation was noted in Guest
  \cite{guest:dmod} for semi-Fano toric varieties, and Iritani
  \cite[Section 5]{iritani:conv}, for not-necessarily-Fano toric
  varieties in general, after passing to a formal completion.  See
  Iritani \cite[Example 5.5]{iritani:conv} and Gonz\'alez-Iritani
  \cite[Example 3.5]{gon:seidel} for examples in the toric manifold
  case.
\item Note that Danilov's results \cite{danilov:toric} do not require
  projectivity of the coarse moduli space.  It seems possible that
  quantum cohomology might also be defined for non-projective toric
  varieties.  Namely certain convergence conditions would remove the
  necessity of working over a Novikov ring, and one might have a
  theorem similar to \ref{main0}, but we lack any results in this
  direction.
\end{enumerate} 
\end{remark}

The presentation of the quantum cohomology in Theorem \ref{main0} can
be re-phrased in terms of Landau-Ginzburg potential as follows,
according to suggestions of Givental \cite{gi:eq} and the physicists
related to mirror symmetry.  This formulation will be essential in our
proof of the injectivity of the map in Theorem \ref{main0}.

\begin{notation} \label{bigtorus}
\begin{enumerate} 
\item {\rm (Residual torus)} Let $\ti{G} := (\C^\times)^{k}$ denote
  the ``big torus'' act on $X = \C^k$ in the standard way.  The {\em
    residual torus}
\[T := \ti{G} / G\]
has an induced action on $ X \qu G$.  The Lie algebra $\t$ of $T$
admits a canonical splitting into real and imaginary parts
$ \on{Re} \oplus \on{Im}: \t \to \t_\R \oplus i \t_\R .$ 
Let $T_\R \subset T$ denote the unitary part of $T$
\[ T_\R = \exp(  \t_\R), \quad \t_\R = \on{span}_\R \t_\Z, \quad  \t_\Z =
\exp^{-1}(1)\] 
given by exponentiating the real span $\t_\R$ of the coweights $\t_\Z$
of $T$.  We have an exact sequence of Lie algebras resp. finitely
generated abelian groups
\begin{equation} \label{ses} 0 \to \g \to \ti{\g} \to \t \to 0, \quad
  0 \to \g_\Z \to \ti{\g}_\Z \to \ti{\t}_\Z \to 0 \end{equation}
where $\ti{\t}_\Z := \ti{\g}_\Z / \g_\Z $.  We write $\ti{\t}_\Z$ as
the product of a {\em free part} $\t_\Z$ of $\ti{\t}_\Z$ which is a
lattice in $\t$, and a {\em torsion part} $\Gamma$ which is isomorphic
to the generic stabilizer of $G$ on $X$.  A canonical parametrization
of the residual torus $T = \ti{G}/G$ can be found by row-reduction on
the matrix of weights, see Example \ref{potexamples} below.
\item {\rm (Moment polytope)} The action of $T_\R$ on $X \qu G$ is
  Hamiltonian, with moment map $\Phi: X \qu G \to \t_\R^\dual$ induced
  by the choice of moment map for the action of $\ti{G}_\R = U(1)^{k}$
  on $X$.  Let $\Delta_{X \qu G} \subset \t^\dual_\R$ denote its image
  \[ \Delta_{X \qu G} := \Phi(X \qu G) \]
the {\em moment polytope} of $X \qu G$.
\item {\rm (Facets and spurious inequalities)} Let
  $\nu_j \in \t_\Z, j =1,\ldots, k$ be the inward normal vectors to
  the facets of $\Delta_{X \qu G}$; these are the images of minus the
  standard basis vectors $e_j$ of $\ti{\g}_\R \cong \R^{k}$ under the
  projection $\pi_\t$ to $\t_\R$:
\[  \nu_j = \pi_\t(- e_j), j = 1,\ldots, k .\]
The moment polytope $\Delta_{X \qu G}$ is of the form
\[ \Delta_{X \qu G} = \{ \mu \in \t_\R^\dual \ | \ \lan \mu, \nu_j \ran
  \ge - \omega_j, j = 1,\ldots, k \} \]
with positions of the facets determined by elements $- \omega_j \in
\Q$.  We say that $ \lan \mu, \nu_j \ran \ge - \omega_j $ is a {\em
  spurious inequality} if it does not correspond to a facet of
$\Delta_{X \qu G}$.
\item {\rm (Support constants)} The {\em support constants} $\omega_j$
  defining the positions of the possible facets $F_1,\ldots, F_k$ of
  $\Delta_{X \qu G}$ can be chosen as follows.  Given an extension of
  $\omega$ to $H^2_{\ti{G}}(X,\Q) \cong \Q^{k}$, the constants
  $\omega_j$ are the coefficients of $\omega$.
\item {\rm (Dual torus)} Recall from Iritani \cite[Section
  3]{iri:integral} that the Landau-Ginzburg potential for toric
  orbifolds has domain a certain formal version of a finite cover of
  the dual torus $T^\dual$ to $T$.  Define
\[ \begin{array}{llllll}  T^\dual &= & \Hom(\t_\Z,\C^\times)  &
                                                             \ti{T}^\dual
  & =& 
\Hom(\ti{\t}_\Z, \C^\times) \\ 
 G^\dual &=& \Hom(\g_\Z,\C^\times) & \ti{G}^\dual &=& \Hom( \ti{\g}_\Z^k,
\C^\times) . \end{array} \] 
We have $T^\dual = \ti{T}^\dual/\Gamma$.  Dualizing \eqref{ses} gives
a short exact sequence
\begin{equation}\label{ses2} 0 \to \ti{T}^\dual \to
  \ti{G}^\dual \to G^\dual \to 0 .\end{equation}
In particular $\ti{T}^\dual$ becomes a subgroup of
$\ti{G}^\dual \cong (\C^\times)^k$.  Define the dual group over
$\Lambda$
\[ \ti{T}^\dual(\Lambda) := \Hom(\ti{\t}_\Z,\Lambda - \{ 0 \}) . \]
and similarly for $\ti{G}^\dual(\Lambda)$.  Define an injection
\begin{equation} \label{qe} \iota_\omega: \ti{T}^\dual(\Lambda) \to
  \ti{G}^\dual(\Lambda), \quad \ti{g} \mapsto (q^{\omega_1}, \ldots,
  q^{\omega_k}) \ti{g} \end{equation}
that we call the {\em quantum embedding} of $\ti{T}^\dual$. (The map
$\iota_\omega$ is not a homomorphism.)
\item {\rm (Givental potential)} In the case of trivial generic
  stabilizer, the naive Landau-Ginzburg potential associated to the
  toric stack $X \qu G$ is the function on the dual torus given as a
  sum of monomials whose exponents are the normal vectors to the
  facets of $\Delta_{X \qu G}$ with coefficients $q^{\omega_j}$:
 \begin{equation} \label{givpot} W_{X,G}: \ti{T}^\dual(\Lambda) \to \Lambda, \quad y \mapsto 
  \sum_{j=1}^k q^{\omega_j} y^{\nu_j} .\end{equation} 
  More generally, in the case of not-necessarily trivial generic
  stabilizer let $W_{X,G}$ denote the restriction of the function
  $\ti{g}_1 + \ldots + \ti{g}_k$ to the subset
  $\iota_\omega \ti{T}^\dual \subset \ti{G}^\dual$.  The reader may
  wish to compare with the definition of potential in Fukaya et al
  \cite[Definition 2.1]{fooo:tms}, where the potential is an element
  of a completed power series ring in coordinates
  $y_j^{\pm}, j = 1,\ldots, \dim(T)$.  It was first noticed by
  Givental \cite{giv:hom} that this function is related to the
  Gromov-Witten theory of $X \qu G$.  An explanation from the point of
  view of mirror symmetry was given in Hori-Vafa \cite{ho:mi}, and a
  connection to Floer theory is described in Fukaya et al
  \cite{fooo:toric1}.  In the latter the potential appears as a count
  of holomorphic disks with boundary in a fiber of the moment map.  In
  the later version, the potential receives corrections from nodal
  holomorphic disks, whereas in Givental \cite{giv:hom} and Hori-Vafa
  \cite{ho:mi} there are no corrections.
\end{enumerate} 
\end{notation} 
\begin{remark} 
\begin{enumerate} 
\item {\rm (Elimination of negative powers of $q$)} As it stands, the
  values of $W_{X,G}$ have negative powers of $q$.  However, later we
  will always assume that $0$ is contained in the interior of the
  moment polytope $\Delta_{X \qu G}$.  In this case only positive
  powers $q^{\omega_j}$ of $q$ occur as coefficients in $W_{X \qu G}$.
\item {\rm (Naive small potential versus corrected small potential)}
  For the many purposes (non-displaceability, Batyrev presentation) it
  seems that the naive potential is ``as good as'' the corrected
  potential defined by disk counts in Fukaya et all
  \cite{fooo:toric1}.  A heuristic argument that the two potentials
  are related by a geometrically-defined change of coordinates was
  given in Woodward \cite{wo:gdisk}; for semi-Fano cases it is proved
  in Chan et al \cite{chan:open} that this coordinate change is the
  mirror map from Gromov-Witten theory, while Fukaya et al
  \cite[Theorem 11.1]{fooo:tms} show the existence of some coordinate
  transformation relating the two.  An approach to relating the
  potentials using an open version of the quantum Kirwan map is
  described in Woodward-Xu \cite{wx}.
\end{enumerate}
\end{remark} 

\begin{example} \label{potexamples}
\begin{enumerate} 
\item {\rm (Product of projective lines)} Let $X=\C^4$ and
  $G=(\C^\times )^2$ with weights
  $\mu_1 = (1,0)$, $\mu_2 = (1,0)$, $\mu_3 = (0,1)$, $\mu_4 = (0,1)$ and
  polarization vector $\omega = (0,1,0,1)$.  The git quotient is
  $X \qu G = \P^1 \times \P^1$.  The perpendicular space to the
  weights is found by row-reduction to be the span of the vectors
  $(1,-1,0,0), (0,0,1,-1)$.  With the corresponding parametrization of
  the dual torus $T = (\C^\times)^4/G \cong (\C^\times)^2$ the normal
  vectors to the facets $F_1,F_2,F_3,F_4$ are
\[
\nu_1=(1,0),\ 
\nu_2=(-1,0),\ \nu_3=(0,1),\ \nu_4=(0,-1)
.\]
The moment polytope is $\Delta_{X \qu G} = [0,1]^2$.  The potential is
\[ W_{X,G}(y_1,y_2) = y_1 + q/y_1 + y_2 + q/y_2. \]
\item \label{extra}  {\rm (Projective line with extra term)} The quotient of
  $X = \C^3$ by the action of $G = (\C^\times)^2$ with weights
  $(1,0), (1,1), (-1,1)$ and polarization vector $(3,0,1)$ (which
  projects to $(2,1) \in \g^\dual$) has semistable locus
\[ X^{\ss} = \{ (x_1,x_2,x_3), x_1 \neq 0, (x_2,x_3) \neq 0 \} \]  
and git quotient $X \qu G \cong \P^1$.  The residual torus $T$ has Lie
algebra $\t$ identified with the span of $(-2,1,-1)$ in $\ti{\g}$. The
moment polytope is 
\[ \Delta_{X \qu G} = \{ \mu \in \R \ | \ 2\mu \leq  3, \mu \ge 0, \mu \leq
  1 \} .\]
The first inequality $ 2\mu \leq 3$ is spurious, that is, may be removed
without changing $\Delta_{X \qu G}$.  The potential is
\[ W_{X,G}(y) = q^3/y^2 + y + q/y 
 .\]
\item \label{half} {\rm (Stacky half-point)} Let $X = \C$ with weight
  $\mu_1 = 2$ so that $X \qu G = \P(2) \cong B\Z_2$.  Then
  $\ti{T} \cong \Z_2$ and the embedding
  $\ti{T}^\dual \to \ti{G}^\dual$ is the standard one with image
  $\{ \pm 1 \} \subset \ti{G}^\dual$.  The potential is then the
  isomorphism
\[ W_{X,G}: \ti{T}^\dual(\Lambda) \cong \Z_2 \to \{ \pm 1 \} \subset
  \Lambda .\]
\end{enumerate} 
\end{example} 

\begin{definition} \label{critloc}  {\rm (Critical locus and Jacobian ring)}  
The {\em critical locus} $\Crit(W_{X,G})$ of $W_{X,G}$ is the set of
points with vanishing logarithmic derivatives with respect to the
coordinates on $T^\dual$,
\begin{eqnarray*} \Crit(W_{X,G}) &=&
 \Set{ y \in \ti{T}^\dual(\Lambda) \ |  \partial_\lambda W_{X,G}( y e^\lambda) |_{\lambda = 0} = 0 
\ \ \ \forall \lambda \in \t^\dual_\R } \\ 
&=&
\Set{   y \in \ti{T}^\dual(\Lambda) \ | 
 \sum_{i=1}^k \lan \nu_i, \lambda \ran
q^{\omega_i} y^{\nu_i} = 0 , \ \ \ \forall \lambda \in \t^\dual_\R
 } .\end{eqnarray*}
Define the ring of functions on $\ti{T}^\dual$
\[ \Lambda(\ti{T}^\dual) = \bigoplus_{\lambda \in
  \ti{\t}_\Z} \Lambda y^\lambda .\]
The ideal generated by the logarithmic partial derivatives 
of the potential is 
\[ \lan \partial_\lambda W_{X,G} (y e^\lambda)  |_{\lambda 
  = 0}, \quad \lambda \in \t_\R
 \ran \subset \Lambda(\ti{T}^\dual) .  \]
The {\em Jacobian ring} $\Jac_\Lambda(W_{X,G})$ of the Givental potential 
$W_{X,G}$ is the ring of functions on $\Crit(W_{X,G})$, or more
precisely the 
quotient 
\[ \Jac_\Lambda(W_{X,G}) = \Lambda(\ti{T}^\dual) / \lan 
\partial_\lambda W_{X,G}(y 
e^\lambda) _{\lambda = 0} \ran \]
 If the generic stabilizer is trivial then we have equivalently 
using the notation \eqref{givpot} 
 \[ \Jac_\Lambda(W_{X,G}) = \Lambda[ y^{\pm \nu_1}, \ldots, y^{\pm \nu_k} ] /
 \lan y_i \partial_{y_i} W_{X,G} \ran .\]
  Assuming that $0$ is contained in the interior of the moment 
  polytope and $\omega$ is integral then the potential $W_{X,G}$ is 
  defined over $\C[q]$.   Define the polynomial Jacobian ring
  \[ \Jac(W_{X,G}) = \C[q,y^{\nu_1}, \ldots, y^{\nu_k} ] /
   \lan y_i \partial_{y_i} W_{X,G} \ran .\]
\end{definition}

We wish to define a certain ``positive part'' of $\Crit(W_{X,G})$
whose coordinate ring corresponds to the quantum cohomology of $X \qu
G$.

\begin{definition} \label{jac}
\begin{enumerate} 
\item 
  {\rm (Positive part of the Jacobian ring)}
\label{jacplus} 
Let $\J \subset \Jac(W_{X,G})$ denote the ideal generated by the
elements $q^{\omega_j} y_j, j = 1,\ldots, k$, and
$\widehat{\Jac}(W_{X,G})$ the completion of $\Jac(W_{X,G})$ with
respect to $\J$,
\[
\widehat{\Jac}(W_{X,G}) := \varprojlim_{ m} \Jac(W_{X,G})/ \J^m .\]
Let $\Jac_+(W_{X,G})$ denote the ring obtained from the formal completion 
by inverting $q$ and the variables $y^{\nu_j}$: 
\[ \Jac_+(W_{X,G}) := \widehat{\Jac}(W_{X,G})[q^{-1},y^{-\nu_1},
  \ldots, y^{-\nu_k}] .\]
\item {\rm (Positive part of the critical locus)} The filtered rings
  $\widehat{\Jac}(W_{X,G})$ respectively $\Jac_+(W_{X,G})$ are the
  ring of functions on the formal scheme $\widehat{\Crit}(W_{X,G})$
  resp.  $\Crit_+(W_{X,G})$ obtained by taking a formal neighborhood
  of $(y,q) = (0,0)$ in the closure of $\Crit(W_{X,G})$ with respect
  to the embedding \eqref{qe} resp. and removing the fiber over
  $q =0 $.  The scheme $\Crit_+(W_{X,G})$ represents the locus of
  critical points $y(q)$ of $W_{X,G}$ that have limit $y(q) \to 0$ as
  $q \to 0$ with respect to the injection \eqref{qe}, that is, each
  expression $ q^{\omega_j} y_j, j = 1,\ldots,k $ has only positive
  powers of $q$.  After passing to a cover
  $\Spec \C[q^{1/n},q^{-1}] \to \Spec \C[q,q^{-1}]$ for some $n$ we
  may write each solution near $q =0 $ as a function $y(q)$ of a
  variable $q^{1/n}$, that is, each component of $\Crit(W_{X,G})$
  becomes unramified over $\Spec \C[q,q^{-1}]$.  By a simple case of
  the Grothendieck Existence Theorem \cite{ega3}, any point of
  $\Crit_+(W_{X,G})$ is obtained by completion and removing the locus
  with $q = 0$ from a point of $\Crit(W_{X,G})$ containing
  $(y,q) = (0,0)$.  This ends the Definition.
\end{enumerate} 
\end{definition}

\begin{example} \label{critex} 
\begin{enumerate} 
\item {\rm (Critical locus for a product of projective lines)}
  Continuing the example of a product of projective lines $X \qu G =
  \P^1 \times \P^1$ from Example \ref{potexamples} (a), the critical
  locus is defined by
\begin{eqnarray*} 0 &=& y_1 \partial_{y_1} W_{X,G}(y_1,y_2) = y_1 - q/y_1
                        \\
0 &=& y_2 \partial_{y_2}
W_{X,G}(y_1,y_2) = y_2 - q/y_2 .\end{eqnarray*} 
The solutions are
$ (y_1,y_2) = (\pm \sqrt{q}, \pm \sqrt{q}) \in \Crit(W) \subset
\T^\dual(\Lambda)$.
Under the map \eqref{qe} these map to
$(\pm \sqrt{q}, \pm \sqrt{q}, \pm \sqrt{q}, \pm \sqrt{q}) \in
\ti{G}^\dual$.  All of these solutions approach $y = 0$ as $q \to 0$.
\item {\rm (Critical locus for a projective line with extra term)}
  Continuing the example of the projective line $X \qu G$ with
  potential with extra term from Example \ref{potexamples}
  \eqref{extra} $ W_{X,G}(y) = q^3/y^2 + y + q/y$.  The critical
  points are $y\sim \pm q^{1/2}$ and $y \sim -2q^2$.  Under the
  injection \eqref{qe} these map to
\[y(q) \sim (q^2, \pm q^{1/2}, \pm q^{1/2})  \] 
which converges to $0$ as $q \to 0$; and 
\[y(q) \sim (q^{-1}/4, -2q^{-2}, -q^{-1}/2) \] 
which does not converge to $0$ as $q \to 0$.  
\item \label{half2} {\rm (Critical locus for the stacky half-point)}
  Continuing Example \ref{potexamples} \eqref{half}, let
  $X \qu G \cong \P(2)$ so that
\[ W_{X,G}: \ti{T} \cong \Z_2 \to \{ \pm 1 \} \subset \Lambda \] 
is the potential for the half-point. Then $\Crit(W_{X,G}) = \ti{T}$
and $\Jac(W_{X,G})$ is the group ring on $\Z_2$, isomorphic to the
orbifold cohomology of $\P(2) = B\Z_2$.
\end{enumerate} 
 \end{example}

 An interpretation in terms of critical points that lie over the
 interior of the moment polytope is given in Proposition
 \ref{closure}.  We will prove the following identification with the
 Jacobian ring:

\begin{theorem}   \label{main}  For any rational symplectic class $\omega \in H_2^G(X)$, 
there is a canonical isomorphism 
\begin{equation} \label{ciso} T_{\kappa_X^G(\omega)} QH(X \qu G) 
\to \hJac(W_{X,G}) . \end{equation} 
\end{theorem}  

\begin{remark} 
\begin{enumerate} 
\item The left-hand-side $T_{\kappa_X^G(\omega)} QH(X \qu G)$ of
  \eqref{ciso} is independent of the presentation of $X \qu G$ as a git
  quotient of $X$ by $G$.  On the other hand, the right-hand-side
  $\hJac(W_{X,G}) $ depends on the presentation. 
\item That the rings
  $ T_{\kappa_X^G(\omega)} QH(X \qu G) , \hJac(W_{X,G})$ have the same
  dimension follows in the Fano case $c_1(X \qu G) > 0$ from
  Kouchnirenko's theorem \cite{kouch:poly,at:ang} Theorem \ref{kouch}
  below.  In general, we deduce the dimension equality
  \[ \dim( T_{\kappa_X^G(\omega)} QH(X \qu G) ) =
  \dim(\hJac(W_{X,G})) \]
  from the toric minimal model program and an induction in Theorem
  \ref{equality} below.  A similar procedure is used by Kawamata
  \cite{kaw:der} to show the existence of an exceptional collection in
  the derived category $D^b \on{Coh}(X \qu G)$ of any toric orbifold
  $X \qu G$.
\item The Frobenius manifold structure $QH(Y)$, including the pairing,
  is expected to be equivalent to Saito's Frobenius structure
  corresponding to the Landau-Ginzburg potential $W$, see for example
  Fukaya et al \cite{fooo:tms}.  However, we do not discuss the
  Frobenius inner product in this paper.
\end{enumerate}
\end{remark} 

  We end the introduction with examples 
of the projective plane, written in different ways as a quotient:

\begin{example} 
\begin{enumerate}
\item {\rm (Projective plane as a quotient by a circle action)}
  Suppose that $G = \C^\times$ acts on $X = \C^3$ by scalar
  multiplication.  Suppose that the polarization corresponds to a
  trivial line bundle with a negative weight on the fiber at the
  origin.  The semistable locus is $X^{\ss} = X - \{ 0 \}$ and the git
  quotient is $X \qu G = \P^2$.  We take the residual action of
  $T = (\C^\times)^3/\C^\times$ to have moment polytope in
  $\t^\dual \cong \R^2$ equal to
\[ \Delta_{X \qu G} = \{ (\lambda_1,\lambda_2) \in \R^2 \ | \  \lambda_1 \ge 0,
\lambda_2 \ge 0, \lambda_1 + \lambda_2 \leq 1 \} .\]
The corresponding potential is 
\[ W_{X,G}(y_1,y_2) = y_1 + y_2 + q/y_1y_2 .\]
The critical points are the solutions to 
\begin{eqnarray*} 
 y_1 \partial_{y_1} W_{X,G}(y_1,y_2) &=& y_1 - q/y_1 y_2 = 0 \\  y_2 \partial_{y_2} W_{X,G}(y_1,y_2) 
&=& y_2 - q/y_1 y_2 = 0 .\end{eqnarray*}
Solutions are $ y_1 = y_2, \quad y_1^3 = y_2^3 = q $.  These generators and
relations give a presentation of the quantum cohomology of $\P^2$.
\item 
{\rm (Projective plane as a quotient by a two-torus action)} The
projective plane $\P^2$ can be realized as a git quotient by a
two-dimensional torus as follows.  Suppose that $G = (\C^\times)^2$
acts on $X = \C^4$ with weights $(-1,1),(0,1),(1,1),(1,0)$.  The
symplectic quotient $X \qu G$ is the ``symplectic cut'' of $\C^2$ by
the circle actions with directions $(-1,1), (1,1)$ in the sense of
Lerman \cite{le:sy2}.  The polytope $\Delta_{X \qu G}$ is the
intersection of a quadrant with two half-spaces with directions
$(-1,1), (1,1)$:
\[ \Delta_{X \qu G} = \Set{ (\lambda_1,\lambda_2) \in \R_{\ge 0}^2  |
\lambda_1 + \lambda_2 \leq c_1, - \lambda_1 + \lambda_2 \leq c_2 } \]
for some constants $c_1 < c_2$.  Suppose that the polarization
corresponds to the weight $(2,1)$; this is the right-most chamber in
Figure \ref{3chambers}.

\begin{figure}[ht]
\small \begin{picture}(0,0)%
\includegraphics{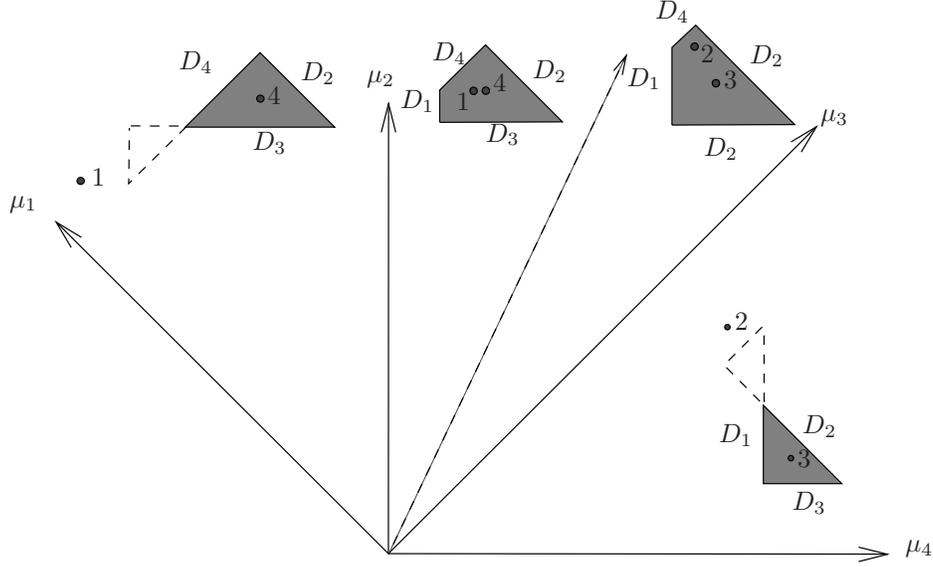}%
\end{picture}%
\setlength{\unitlength}{3947sp}%
\begingroup\makeatletter\ifx\SetFigFont\undefined%
\gdef\SetFigFont#1#2#3#4#5{%
  \reset@font\fontsize{#1}{#2pt}%
  \fontfamily{#3}\fontseries{#4}\fontshape{#5}%
  \selectfont}%
\fi\endgroup%
\begin{picture}(5656,3492)(1505,3480)
\put(6614,6227){\makebox(0,0)[lb]{{{{$\mu_3$}%
}}}}
\put(1520,5688){\makebox(0,0)[lb]{{{{$\mu_1$}%
}}}}
\put(3764,6470){\makebox(0,0)[lb]{{{{$\mu_2$}%
}}}}
\put(7146,3519){\makebox(0,0)[lb]{{{{$\mu_4$}%
}}}}
\put(3350,6488){\makebox(0,0)[lb]{{{{$D_2$}%
}}}}
\put(5398,6456){\makebox(0,0)[lb]{{{{$D_1$}%
}}}}
\put(5882,6017){\makebox(0,0)[lb]{{{{$D_2$}%
}}}}
\put(6170,6600){\makebox(0,0)[lb]{{{{$D_2$}%
}}}}
\put(5976,4216){\makebox(0,0)[lb]{{{{$D_1$}%
}}}}
\put(6502,4270){\makebox(0,0)[lb]{{{{$D_2$}%
}}}}
\put(6437,3788){\makebox(0,0)[lb]{{{{$D_3$}%
}}}}
\put(3047,6053){\makebox(0,0)[lb]{{{{$D_3$}%
}}}}
\put(2584,6560){\makebox(0,0)[lb]{{{{$D_4$}%
}}}}
\put(4177,6614){\makebox(0,0)[lb]{{{{$D_4$}%
}}}}
\put(5576,6877){\makebox(0,0)[lb]{{{{$D_4$}%
}}}}
\put(4512,6107){\makebox(0,0)[lb]{{{{$D_3$}%
}}}}
\put(4806,6505){\makebox(0,0)[lb]{{{{$D_2$}%
}}}}
\put(3976,6314){\makebox(0,0)[lb]{{{{$D_1$}%
}}}}
\put(4558,6448){\makebox(0,0)[lb]{{{{4}%
}}}}
\put(4328,6329){\makebox(0,0)[lb]{{{{1}%
}}}}
\put(3141,6369){\makebox(0,0)[lb]{{{{4}%
}}}}
\put(2036,5854){\makebox(0,0)[lb]{{{{1}%
}}}}
\put(5863,6630){\makebox(0,0)[lb]{{{{2}%
}}}}
\put(6008,6469){\makebox(0,0)[lb]{{{{3}%
}}}}
\put(6077,4945){\makebox(0,0)[lb]{{{{2}%
}}}}
\put(6470,4083){\makebox(0,0)[lb]{{{{3}%
}}}}
\end{picture}%

\caption{Chamber structure for git quotients and the moment images of
  the critical values, with multiplicities}
\label{3chambers}
\end{figure}

The semistable locus is 
\[ X^{\ss} = \{ x = (x_1,x_2,x_3,x_4) | x_4 \neq 0, (x_1,x_2,x_3) \neq
0 \} . \]  
The git quotient is $X \qu G = \P^2$.  In particular, the equation $x_4 =0$ does
not define a divisor in $X \qu G$.  The potential is
\[ W_{X,G}(y_1,y_2) = y_1 + y_2 + q/y_1y_2 + q^2 y_1/y_2 .\]
The partial derivatives are
\begin{eqnarray*} \partial_{y_1} W_{X,G}(y_1,y_2)  &=& 1 - q /y_1^2 y_2 + q^2 /y_2, \\  \partial_{y_2} W_{X,G}(y_1,y_2) &=& 1 - q /y_1 y_2^2 - q^2 y_1/ y_2^2
.\end{eqnarray*}
The critical points, to leading order, are 
\[ y_1 \sim y_2 \sim \exp(  2 \pi i k /3)  q^{1/3}, k = 0,1,2 \]
and the two critical points 
%
\[ y_1 \sim \pm i q^{- 1/2}, \quad y_2 \sim - 2 q^2 \]
as shown in Figure \ref{3chambers}.  The first three (resp. second
two) points (resp. do not) define elements of $\Crit_+(W_{X,G})$.
Hence $\Crit_+(W_{X,G})$ consists of three reduced points,
$QH(X \qu G) \cong \C^{\oplus 3}$.  The other pictures in Figure
\ref{3chambers} show the quotients for the other polarizations; the
dotted line represents the ray $\R_{\ge 0} c_1^G(TX)$ generated by the
equivariant first Chern class $c_1^G(TX)$, for which the quotient
$X \qu G$ has a potential with all critical points located at
$0 \in \t^\dual$.  The number of critical points
$y \in \Crit(W_{X,G})$ mapping to each point in
$\Psi(y) \in \Delta_{X \qu G}$ is indicated in Figure \ref{3chambers}.
This ends the example.
\end{enumerate} 
\end{example} 

We thank D. Cox, H. Iritani, D. McDuff, and C. Teleman for helpful
comments.

\section{Quantum Kirwan surjectivity for toric orbifolds} 
\label{formal} 

In this section we prove surjectivity for the linearization of the
quantum Kirwan map on a formal completion of equivariant quantum
cohomology; the surjectivity also holds for the uncompleted cohomology
but does not lead to an isomorphism.  Let $X$ be a smooth polarized
projective $G$-variety, or more generally, a smooth polarized
quasiprojective $G$-variety convex at infinity in the sense of
\cite{qkirwan1,qkirwan2,qkirwan3}, such as a finite-dimensional vector
space with the action of a torus whose weights are contained in a
half-space.  The version of quantum Kirwan surjectivity we need
involves a formal completion of the equivariant quantum cohomology.
In this completion not only the powers of $q$ but also the degrees of
the cohomology classes can go to infinity:

\begin{definition} {\rm (Formal equivariant quantum cohomology ring)} 
\label{completed} 
Let $\widehat{QH}_G(X)$ be the vector space of infinite sums
\[ \widehat{QH}_G(X) = 
\Set{ \sum_{i=1}^\infty q^{\rho_i} \alpha_i
| \alpha_i \in H_G^{a_i}(X), \ \inf_i \rho_i > -\infty, \ \lim_{i
  \to \infty} \rho_i + a_i = \infty }. \]
Equivalently, $\widehat{QH}_G(X)$ is obtained by completing $H_G(X)
\otimes \Lambda_0$ with respect to the degree filtration on $H_G(X)$,
and then inverting $q$.
\end{definition}  

\begin{remark}  {\rm (Other completions)}   
Note that there are various other natural completions.  For example,
completing $H_G(X) \otimes \Lambda$ with respect to the degree
filtration on $H_G(X)$ gives a space of formal sums whose $q$-degree
is not necessarily bounded below.  The quantum Kirwan map does not
extend to this formal completion.  In the toric case the relationship
between various completions is discussed by Fukaya et al in
\cite[Section 12]{fooo:tms}.
\end{remark}

\begin{proposition} {\rm (Extension of the quantum Kirwan map to the
    formal equivariant quantum cohomology)} \label{qkextends} Each
  Taylor coefficient
  $ \kappa_X^{G,n} : QH_G(X)^n \to QH(X \qu G), n \in \Z_{\ge 0} $
  extends to a map $ \widehat{QH}_G(X)^n \to QH(X \qu G)$, still
  denoted $\kappa_X^{G,n}$.
\end{proposition} 

\begin{proof} The statement of the proposition follows from the
  properness result for scaled gauged maps with bounded energy: for
  any $e > 0$, the set of non-empty $\ovl{\M}_{n,1}^G(\bA,X,d)$ for
  which $\lan d,\omega \ran <e $ is finite \cite[Theorem
  1.1]{pscale}. In particular, for any energy bound $e$, the virtual
  dimensions of the components $\ovl{\M}_{n,1}^G(\bA,X,d)$ of energy
  $\lan d, \omega \ran < e$ are bounded from above by some number
  $f(e)$.  Thus if $\alpha_i \in H_G^{a_i}(X), i = 1,\ldots, n$
  satisfy
\[\sum_{i=1}^n a_i > f(e) + \dim(X \qu G)\] 
then the push-forward of $\ev^* (\alpha_1,\ldots,\alpha_n) \in
H(\ovl{\M}_{n,1}^G(\bA,X,d))$ to $\ovl{I}_{X \qu G}$ is zero for reasons
of dimension.  If $\alpha \in H_G(X)$ has degree bounded from below by
$f(e) + \dim(X \qu G)$ then the contribution of
$\kappa_X^{G,n}(\alpha)$ contains only terms with $q$-degree at least
$e$.  The claim follows.
\end{proof} 

We now partially compute the quantum Kirwan map in the toric case.
Let $X$ be a finite-dimensional complex vector space with an action of
a complex torus $G$, with weights $\mu_1,\ldots,\mu_k$ contained in an
open half-space and equipped with a polarization so that the quotient
$X \qu G$ is locally free.  Our first step is to classify the affine
gauged maps which appear in the definition of $\kappa_X^G$:

\begin{theorem} \label{cvort} {\rm (Classification of affine gauged maps in the toric
case)} An affine gauged map to $X/G$ of homology class $d \in
  H_2^G(X,\Q)$ is equivalent to a morphism $u = (u_1,\ldots,u_k): \bA
  \to X$ satisfying the conditions that
\begin{enumerate}
\item \label{first} the degree of $u_j$ is at most $\lan \mu_j,d \ran$; and
\item \label{second} if $ u (\infty) = \left( u_j(\infty) = \begin{cases} u_j^{ \lan \mu_j, d \ran }/
  \lan \mu_j , d \ran !  & \lan \mu_j, d \ran \in \Z_{\ge 0} \\ 0 &
  \text{otherwise} \end{cases} \right)_{j=1}^k $
denotes the vector of leading order coefficients with integer
exponents, then $u(\infty) \in X^{\ss}$.
\end{enumerate} 
Thus $\M_{1,1}^G(\bA,X,d)^G$ resp.  $\M_{1,0}^G(\bA,X,d)^G$ is the
quotient of the space of such morphisms by 
the action of $G$ resp.  the action of
$G$ and translation.
\end{theorem} 

\begin{example} \label{classex} 
{\rm (Examples of the classification of affine gauged maps)} 
\begin{enumerate}
\item \label{jt} {\rm (Jaffe-Taubes classification)} If
  $G = \C^\times$ acts on $X = \C$ with weight $1$ then
  $\M_{1,1}^G(\bA,X,d)$ consists of polynomials $u(z)$ of degree
  exactly $d$ quotiented by the action of $G$.  Since any such
  polynomial $u(z)$ is classified by its roots $z, u(z) =0$, we have 
  $\M_{1,1}^G(\bA,X,d) \cong \Sym^d(\bA)$.
\item \label{half3} {\rm (Stacky half-point)} Continuing Examples
  \ref{potexamples} \eqref{half} and \ref{critex} \eqref{half2},
  suppose $G = \C^\times$ acts on $X = \C$ with weight $2$.  Let $d$
  to be a half-integer.  The moduli stack $\M_{1,1}^G(\bA,X,d)$
  consists of non-zero polynomials $u(z)$ of degree $2d$ quotiented by
  the action of $G$.  Since any such polynomial is classified by its
  roots, $\ovl{\M}_{1,1}^G(\bA,X,d) \cong \Sym^{2d}(\bA)$.  The
  evaluation map to $X \qu G = \P(2)$ maps to the trivial
  resp. twisted sector if $2d$ is even resp. odd.
\item {\rm (Projective space)} If $G = \C^\times$ acts on $X = \C^k$
  by scalar multiplication then $\M_{1,1}^G(\bA,X,d)$ consists of
  tuples $(u_1,\ldots,u_k)$ of polynomials of degree most $d$ such
  that at least one of the polynomials $u_j$ is degree exactly $d$,
  quotiented by the action of $G$.  One sees that
  $\M_{1,1}^G(\bA,X,d)$ is a vector bundle over $\P^{k-1}$ of rank
  $dk$.
\end{enumerate} 
\end{example} 

\begin{proof}[Proof of Theorem]   
  By definition, a morphism $u: \P(1,r) \to X/G$ consists of a
  $G$-bundle $P \to \P(1,r)$ and a section of the associated
  $X$-bundle $u: \P(1,r) \to P \times_G X$.  The bundle can be
  described by a clutching function $z \mapsto z^{\lambda/r}$ for some
  $\lambda \in \g_\Z^\dual$ with $\lambda/r = d$.  The first condition
  \eqref{first} is the condition that a map $\bA \to X$ extend to a
  global section.  The second condition \eqref{second} is that the
  extension maps $\P(r)$ to the semistable locus $X \qu G$.  The
  representability condition for the morphism $u$ is that the image of
  $\P(r)$ is a point $u(\infty)$ in $X \qu G$ with automorphism group
  containing a group $\Z_r$ generated by $\exp(\lambda/r)$, so that
  $\lambda/r$ is the minimal representation of $d$.
\end{proof}

\begin{theorem} \label{weaktoric2}
 {\rm (Quantum Kirwan surjectivity, toric case)}
For any rational symplectic class $\omega \in H_2^G(X,\Q)$, the map
$D_\omega \kappa_X^G: T_\omega \widehat{QH}_G(X) \to
T_{\kappa_X^G(\omega)} QH(X \qu G)$ is surjective.
\end{theorem} 

\begin{example} {\rm (The case of a free quotient)} If $X \qu G$ is a
  smooth {\em variety}, that is, has no orbifold points, then the
  statement of Theorem \ref{weaktoric2} follows from Kirwan's
  surjectivity result from \cite{ki:coh}, or an explicit description
  of the classical Kirwan map in the toric case. Indeed the leading
  order term (setting the Novikov parameter $q$ to zero) is the
  classical Kirwan map.  The novelty of the above theorem is that in
  case that $X \qu G$ is an orbifold, the twisted sectors in
  $QH(X \qu G)$ are also in the image of the quantum Kirwan map, so
  that $T_{\kappa_X^G(\omega)} QH(X \qu G)$ is a quotient of the usual
  ring of polynomial invariants
  $T_\omega QH_G(X) \cong \on{Sym}(\g^\dual)^G \otimes \Lambda$.
\end{example}

The proof of Theorem \ref{weaktoric2} in general relies on the
following computation which we call a {\em fractional Batyrev
  relation} based on similarity with Batyrev \cite{bat:qcr}.

\begin{notation} 
\begin{enumerate}
\item \label{fc} {\rm (Ceiling)} 
The {\em ceiling} $\lceil x \rceil$ is the smallest integer greater
  than or equal to $x$.
\item {\rm (Classification of twisted sectors)} We identify
  $H_2^G(X) \cong \g$ and $H^2_G(X) \cong \g^\dual$.  Any
  $d \in H_2^G(X)$ thus defines an element $\exp(d) \in G$.  This
  element corresponds to a summand in $H(I_{X \qu G})$ if it has
  non-trivial fixed point set in $X^{\ss}$.
\item {\rm (Identity in each twisted sector)} For such $d$ we denote
  by $1_{\exp(d)} \in H(I_{X \qu G})$ the degree zero class in the
  twisted sector (which will have non-zero degree with respect to the
  grading on $QH(X \qu G)$.)
\item {\rm (Divisor classes in twisted sectors)} For each
  $j = 1,\ldots, k$, we denote by $\nu_j 1_{\exp(d)}$ the (possibly
  empty) divisor class in the twisted sector for $\exp(d)$, obtained
  by setting the $j$-th coordinate $x_j$ equal to $0$.  The product of
  the classes $\nu_j$ with $1_{\exp(d)}$ is the product of these
  cohomology classes in the twisted sector in $H(I_{X \qu G})$.
\end{enumerate}
\end{notation}  

\begin{proposition} [Fractional Batyrev relation] 
\label{fractional} 
For any $d \in \g$ such that $\exp(d)$ has non-empty fixed point set
in $X^{\ss}$ and $\M_{1,1}(\bA,X,d)$ is non-empty,
\[D_\omega \kappa_X^G \left( \prod_{\lan \mu_j, d \ran \ge 0} \mu_j^{\lceil \lan \mu_j, d \ran
  \rceil} \right) = 
\prod_{\lan \mu_j, d \ran \leq 0} \nu_j^{- \lceil
  \lan \mu_j, d \ran \rceil} q^{ \lan \omega,d \ran} 
1_{\exp(d)} + \text{\rm higher order in $q$}
.\]
\end{proposition} 

\begin{proof}[Proof of Proposition \ref{fractional}]  
  Recall that $\ovl{\M}_{1,1}^G(\bA,X,d)$ is the stack of once-marked
  stable scaled affine gauged maps to $X$
  \cite{qkirwan1,qkirwan2,qkirwan3}.  In general, this
  compactification allows bubbles in $X \qu G$ and ghost bubbles in
  $X$ when the markings $z_1,\ldots,z_n$ on the domain come together.
  However, since $X$ is affine, there are no non-constant holomorphic
  maps $u: \P^1 \to X$ from projective lines $\P^1$ to $X$.  Thus any
  element of $\ovl{\M}_{1,1}^G(\bA,X,d)$ consists of components that
  are affine gauged maps to $X$, and components that
  are \label{arechange} stable maps to $X \qu G$.  We wish to compute
  the virtual push-forward under
  $\ev_\infty: \ovl{\M}_{1,1}^G(\bA,X,d) \to \ovl{I}_{X \qu G}$ of
\begin{equation} \label{push} 
 \ev^* \prod_{\lan \mu_j, d \ran \ge 0} \mu_j^{\lceil \lan \mu_j, d
   \ran \rceil} = 
\ev^* \Eul \left( \prod_{\lan \mu_j, d \ran \ge 0}
 \C_{\mu_j}^{\lceil \lan \mu_j, d \ran \rceil} \right) .
\end{equation} 
 On the families of affine gauged maps considered here, there is a
 lift of $\ev_\infty$ from the rigidified inertia stack $\ovl{I}_{X \qu
   G}$ to $I_{X \qu G}$, and we may ignore the rigidification.
 Consider the section
\[ \sigma: \ovl{\M}_{1,1}^G(\bA,X,d) \to \ev^* \prod_{\lan \mu_j, d
  \ran \ge 0} \C_{\mu_j}^{\lceil \lan \mu_j, d \ran \rceil} , \quad u
\mapsto ( u_j^{(i)}(z_1))_{j=1,i= 0}^{k,\lceil \lan \mu_j, d \ran \rceil
}\]
consisting of the derivatives $u^{(i)}$ of $u$ at the marking $z_1$;
note this is well-defined because of the scaling on the domain, that
is, we are modding out by translation on the domain only.  On the
stratum $\M_{1,1}^G(\bA,X)$, $\sigma$ has zeroes corresponding to maps
with all lower-order terms vanishing.  The restriction of $\ev_\infty$
to $\sigma^{-1}(0)$ defines an isomorphism
\[ \ev_\infty |_{\sigma^{-1}(0)} \ : \  \sigma^{-1}(0) \to \ovl{I}_{X \qu G}(\exp(d)) \cap \{
x_j = 0, \lan \mu_j, d \ran \leq - 1 \} \]
where $\ovl{I}_{X \qu G}(\exp(d))$ is the sector with stabilizer
$\exp(d)$, defined as the git quotient of the subspace of $X$
corresponding to coordinates $x_j$ with $\lan d , \mu_j \ran \in \Z$.
The obstruction space at any morphism $u$ is the higher cohomology of
the vector bundle $P \times_G X$.  This higher cohomology may be
identified via duality with the span of monomials whose $j$-th
component has degree strictly between $0$ and
$- \lceil \lan \mu_j, d \ran \rceil $.  Thus the obstruction bundle
has Euler class
\[ \Eul\left( \prod_{\lan \mu_j, d \ran \leq -1} \C_{\mu_j}^{- \lceil
  \lan \mu_j, d \ran \rceil - 1} \right) .\]
It follows that the contribution to the pushforward \eqref{push}
from the stratum with irreducible domain is 
\[ 
 \left( \prod_{\lan \mu_j, d \ran \leq -1}
 \nu_j^{- \lceil \lan \mu_j, d \ran \rceil - 1 + 1} \right) 
1_{\exp(d)} 
\]
as claimed. 

Next we examine contributions from the boundary.  Any boundary
configuration in $\ovl{\M}_\Gamma(\bA,X,d) - \M_\Gamma(\bA,X,d)$
contains a component with a marking and a gauged map $u: \bA \to X$ of
class $d'$ with $\lan d', \omega \ran < \lan d,\omega \ran$ together
with other components that are morphisms to $X \qu G$ and affine
gauged maps without markings.  The zero set $\sigma^{-1}(0)$ of
$\sigma$ on such a stratum $\M_\Gamma(\bA,X)$ consists of
configurations where $u: C \to X/G$ has $j$-th component $u_j = 0$ if
$\lan \mu_j, d \ran > \lan \mu_j, d' \ran \in \Z$.  So the components
$u_j$ with non-zero leading order correspond to $j$ with
$\lan \mu_j, d \ran \leq \lan \mu_j, d' \ran$.  Clearly, the convex
hull $\on{hull}(\{ \mu_j \} )$ of such weights cannot contain
$\omega$, since
$ \lan d - d' , \omega \ran > 0 \ \text{but} \ \lan d - d', \mu_j \ran
\leq 0 .$
By the description of the unstable locus in \eqref{unstable}, the
asymptotic limit of points in $\sigma^{-1}(0)$ consist of unstable
points in $X$.  So the zero set $\sigma^{-1}(0)$ is empty.

Finally we consider the integral 
\begin{equation} \label{integral} 
\int_{[\ovl{\M}_{1,1}^G(\bA,X,d')]} \ev^* 
\prod_{\lan \mu_j, d \ran
  \ge 0} \mu_j^{\lceil \lan \mu_j, d \ran \rceil} \end{equation} 
for $d \neq d'$.  The same argument as in the previous paragraph shows
that the integral \eqref{integral} is zero unless
$\lan d,\omega \ran \leq \lan d', \omega \ran$.  But since $\omega$ is
generic, this inequality implies strict inequality.  Hence the
contributions \eqref{integral} from such degrees $d'$ are of higher
energy than the leader order terms.
\end{proof}

\begin{corollary} \label{unit} {\rm (Surjectivity onto twisted units)}  
 For any $g \in G$ with non-trivial stabilizer in $X^{\ss}$, there
 exists an element $d \in \g$ with $\exp(d) = g$ and $ \lan d,\mu_j\ran > 0$
 for all $j = 1,\ldots, k$ and thus
\[ D_\omega \kappa_X^G \left( \prod_{j=1}^k \mu_j^{\lceil \lan \mu_j,
    d \ran \rceil} \right) = 1_{\exp(d)} q^{\lan d, \omega \ran} \quad 
\text{\rm mod higher order in $q$}.\]
\end{corollary} 

\begin{proof} Since the weights $\mu_j, j = 1,\ldots, k$ are contained
  in a half-space, there exists a vector $\zeta \in \g_\R$ such that
  $ \lan \zeta,\mu_j \ran > 0$ for $j = 1,\ldots, k$.  Let
  $U \subset \g$ be a compact subset such that $\exp(U) = G$.  Then
  $c \zeta + U$ contains the desired vector $d$, for $c \gg 0$.
\end{proof}

\begin{notation} \label{tsect}  {\rm (Cohomology classes in twisted sectors)} 
  For any $j \in \{1, \ldots, k\}$ such that
  $\lan \mu_j, d \ran \in \Z$ denote by
  $1_{\exp(d)} \delta_j \in H^2(I_{X \qu G})$ the corresponding
  divisor class in the twisted sector corresponding to $\exp(d)$.  Let
\[ 1_{\exp(d)} \delta_J = 1_{\exp(d)} \prod_{j \in J} \delta_j \in
  H(I_{X \qu G}) \] 
  be the classical product of divisor classes in the twisted sector
  for $\exp(d)$.  Since each component of $I_{X \qu G}$ is itself a
  rationally smooth toric stack, any cohomology class of $I_{X \qu G}$
  arises in this way by the classical description by
  Danilov-Jurkiewicz \cite{danilov:toric,jurk:toric1,jurk:toric2}.
\end{notation} 

\begin{corollary} \label{tsurj} {\rm (Surjectivity onto twisted sectors)}  With $d$ as in Corollary 
\ref{unit}, for any subset $J$ of $\{ j, \lan \mu_j, d \ran \in \Z \}$, 
\[ D_\omega \kappa_X^G \left( \prod_{j=1}^k \mu_j^{\lceil \lan \mu_j,
    d \ran \rceil}
\prod_{j \in J} \mu_j \right) = 1_{\exp(d)} \delta_J q^{\lan d, \omega
  \ran} \quad \text{\rm mod higher order in $q$}.\]
\end{corollary} 

\begin{proof}[Proof of Theorem \ref{weaktoric2}] 
  For each $(g,J)$ let
  $\alpha_{g,J} := 1_{g} \delta_J \in H(I_{X \qu G})$ and choose an
  element $\ti{\alpha}_{g,J}$ such that
  $D_\omega \kappa_X^G(\ti{\alpha}_{g,J})$ is equal to $\alpha_{g,J}$
  plus terms of higher energy.  A recursion produces an inverse.
  (Note that by choosing a basis, one obtains an inverse involving a
  finite number of classes in $H_G(X)$, that is, the formal completion
  is not necessary for surjectivity.)
\end{proof}  

We give several examples.

\begin{example} \label{teardrop3}
\begin{enumerate} 
\item {\rm (Stacky half-point)} Continuing Example \ref{classex}
  \eqref{half3}, let $X = \C$ with $G = \C^\times$ acting with weight
  $2$, so that $X \qu G = \P(2) = B \Z_2$.  The quantum cohomology is
  $QH(X \qu G) = \Lambda \oplus \Lambda$, corresponding to the
  stabilizers $1,-1$, called the {\em untwisted} and {\em twisted
    sectors}.  We identify $\g^\dual_\Z = \Z$ in the standard way.  By
  the classification Theorem \ref{cvort} the class zero component
  consists of constant maps
\[\ovl{\M}^G_{1,1}(\bA,X,0) = (X - \{ 0 \})/G \cong \P(2).\]
The class-$1/2$ component is of dimension one:
\[{\M}^G_{1,1}(\bA,X,1/2) = \{ c_1 z+ c_0, c_1 \neq 0 \}/G \cong \C/\Z_2
.\]  
If $QH_G(X) = \Lambda[\xi]$ with generator the Euler class $\xi$ of
the weight one representation then the fractional Batyrev relations
give
\[D_\omega \kappa_X^G(1) = 1, \quad D_\omega \kappa_X^G(\xi) =
(q^{1/2}/2) \theta_-, \quad D_\omega \kappa_X^G(\xi^2) = q/4 .\]
Thus $D_\omega \kappa_X^G$ is surjective and we obtain a presentation
\[ T_{\kappa_X^G(\omega)} QH(\P(2)) \cong \Lambda[\xi]/(\xi^2 -
q/4) .\]
Because $c_1^G(X)$ is positive on curve classes, in this case
$\kappa_X^G(0)= 0$ and $\kappa_X^G(\omega)$ is the reduced symplectic
class.  Note that under the identification $\xi \mapsto q^{1/2}/2
\theta_-$ this agrees with the isomorphism of $QH(\P(2))$ with the
group ring of $\Z_2$ in Example \ref{qhex}.
\item 
 {\rm (Teardrop orbifold)} If $G = \C^\times$ acts on $X = \C^2$ with
 weights $1,2$, so that $X \qu G = \P(1,2)$ is the teardrop orbifold,
\[D_\omega \kappa_X^G(\xi) = \theta_+, \quad 
  D_\omega \kappa_X^G(\xi^2) = q^{1/2} \theta_- /2, \quad D_\omega \kappa_X^G(\xi)^3 =
  q/4 .\]  
Thus $D_0 \kappa_X^G$ is surjective and we obtain a presentation
\[ T_{\kappa_X^G(\omega)} QH(\P(1,2)) = \Lambda[\xi]/(\xi^3 - q/4)
.\]
See Coates-Lee-Corti-Tseng \cite{coates:wps} and Mann \cite{mann:wps}
for more on the quantum cohomology of weighted projective spaces.
\end{enumerate} 
\end{example} 

\begin{remark} {\rm (Generation by divisor classes)} 
  In particular, our version of quantum Kirwan surjectivity implies
  that $QH(X \qu G)$ is generated by the ``divisor classes''
  $D_\omega \kappa_X^G (\mu_i)$.  Note that this is the case even if
  $X \qu G$ has no divisors, for example, when $X \qu G$ is the stacky
  half-point $\P(2)$ then $QH(X \qu G) \cong \Lambda[\Z_2]$ is the
  group ring of $\Z_2$.  These ``divisor classes'' are degree two only
  when working over the larger Novikov ring $\Lambda_X^G$ to achieve a
  $\Z$-grading.
\end{remark} 

\section{Quantum Stanley-Reisner ring and Jacobian ring} 

In this section we identify the quantum Stanley-Reisner ring with the
Jacobian ring, and discuss various extensions \label{extchange} to formal versions.
Consider a git quotient $X \qu G$ of a finite-dimensional complex
vector space $X$ by the action of a complex torus $G$ with weights
$\mu_1,\ldots,\mu_k$ contained in an open half-space and spanning
$\g^\dual$.

\begin{definition} \label{csr}
\begin{enumerate} 
\item {\rm (Classical Stanley-Reisner ideal)} The {\em Stanley-Reisner
    ideal} $SR_X^G$ in $QH_G(X)$ is the ideal generated by products of
  weights 
  \[ \mu_I = \prod_{i \in I} \mu_i \in QH_G^{2|I|}(X) \]
where $I$ is a {\em primitive} collection $I \subset \{1 ,\ldots, k\}$
with respect to the fan of $X \qu G$: the set
$X_I = \{ x_i = 0, i \in I \}$ is contained in the unstable locus of
$X$ and $I$ is a minimal subset with this property.
\item
 {\rm (Classical Stanley-Reisner ring)} The {\em Stanley-Reisner ring}
 is the quotient of $H_G(X)$ by the Stanley-Reisner ideal $SR_X^G$.
\end{enumerate} 
\end{definition} 

\begin{theorem} {\rm (Rational cohomology of a projective 
simplicial toric variety)} 
  Suppose that stable=semistable for the $G$-action on $X$. Then the
  Kirwan map $H_G(X,\Q) \to H(X \qu G,\Q)$ induces an isomorphism
  $H_G(X,\Q)/SR_X^G \to H(X \qu G,\Q) .$ \end{theorem}

\begin{proof} This is essentially the 
Danilov-Jurkiewicz description of the cohomology ring
\cite{danilov:toric,jurk:toric1,jurk:toric2}, using the fact that each
weight function $\mu_i \in \g^\dual_\Q \cong H^2_G(X,\Q)$ maps to the
corresponding divisor class in $H ( X \qu G,\Q)$ under the classical
Kirwan map.  A description from the point of view of equivariant
cohomology can be found in Bonavero-Brion \cite{bb:toric}.
\end{proof}

Note that there are no ``linear relations'' in the above description;
these are in the standard description the kernel of the map
$H_{\ti{G}}(X,\Q) \to H_G(X,\Q)$, where $\ti{G}$ is the ``big torus''
from Notation \ref{bigtorus}.

\begin{definition}  {\rm (Formal quantum Stanley-Reisner ring)}
\label{qsr} 
\begin{enumerate} 
\item The {\em formal quantum Stanley-Reisner ideal}
  $\widehat{QSR}_X^G \subset \widehat{QH}_G(X,\Q)$ is the completion
  (or equivalently, the closure) of $QSR_{X,G}$ in
  $\widehat{QH}_G(X,\Q)$.  The quotient 
$\widehat{QH}_G(X,\Q)/
  \widehat{QSR}_X^G$ is the {\em formal Batyrev ring}.
\item The {\em equivariant resp. formal equivariant quantum
  Stanley-Reisner ideal resp. Batyrev ring} are obtained by replacing
  the expressions $\mu_j$ by their unrestricted versions.  Denote by
  $\eps_1,\ldots, \eps_k \in \ti{\g}^\dual \cong H^2_{\ti{G}}(X)$ the
  coordinates (weights) on the big torus $\ti{G}$
acting on $X$ and so that $\mu_1,\ldots,\mu_k \in \g^\dual \cong
QH_G^2(X)$ are their restrictions to $\g$.  The equivariant quantum
Stanley-Reisner ideal is the closure of the ideal generated by
\[ QSR_X^{G,\ti{G}}(d) := \prod_{\lan \mu_j, d \ran \ge 0}
  \eps_j^{\lan \mu_j, d \ran} - q^{\lan d, \omega \ran} \prod_{\lan
    \mu_j, d \ran \leq 0} \eps_j^{-\lan \mu_j, d \ran} .\]
\end{enumerate}
\end{definition}

\begin{example} {\rm (Batyrev relations for a quotient of affine four-space by a two-torus)} 
  Let $G = (\C^\times)^2$ acting on $X = \C^4$ with weights
  $(-1,1),(0,1),(1,1),(1,0)$.  The corresponding chamber structure and
  polytopes of the quotients are shown in the Figure \ref{3chambers}.
  The equivariant quantum Stanley-Reisner relations include
\[ \eps_1^{-d_1 + d_2} \eps_2^{d_2} \eps_3^{d_1 + d_2} \eps_4^{d_1} = 
q^{d_1 + d_2}, \quad -d_1 + d_2,d_1 \ge 0 \]
%
where $\eps_i$ are the equivariant generators from Definition
\ref{qsr}.  In particular we have
\begin{equation} \label{rels} \eps_1 \eps_2 \eps_3 = q, \quad \eps_2
  \eps_3^2 \eps_4 = q^2 .\end{equation}
Notice the first relation in \eqref{rels} defines the quantum
cohomology $QH(X \qu G)$ for the quotient $X \qu
G$ in the right-most chamber in Figure \ref{3chambers}.  The second
relation in \eqref{rels} defines quantum cohomology for the quantum
cohomology in the left-most chamber.  The non-equivariant relations
are
\[ (-\xi_1 + \xi_2)(\xi_2)(\xi_1 + \xi_2) = q, \quad \xi_2 (\xi_1 + \xi_2)^2 \xi_1 =
q^2 .\]
%
\end{example} 

\begin{proposition} {\rm (Inclusion of the Batyrev relations in the kernel
of the quantum Kirwan map)} \cite{qkirwan1,qkirwan2,qkirwan3} \label{ker} The kernel of
  the linearized quantum Kirwan map
\[ D_\alpha \kappa_X^G: T_\alpha QH_G(X) \to T_{\kappa_X^G(\alpha)}
QH(X \qu G)\]
contains the quantum Stanley-Reisner ideal for $\alpha$ in a formal
neighborhood of the symplectic class $\omega \in QH_G^2(X)$.  So there
is a quotient map
\begin{equation} \label{mapsto}
T_\alpha QH_G(X)/QSR_{X,G}(\alpha) \to T_{\kappa_X^G(\alpha)} QH(X \qu
G) .\end{equation}
\end{proposition} 

\begin{proof}
  The adiabatic limit theorem of \cite{qkirwan1,qkirwan2,qkirwan3}
  relates the localized genus zero graph potential
  $\tau_{X \qu G,-} \in H(X\qu G)[[\hbar^{-1}]] $ of $X \qu G$ with
  the localized gauged graph potential
  $\tau_{X,-}^G \in H_G(X)[[\hbar^{-1}]] $ of $X$,
\[ \tau_{X,-}^G(1,\hbar,q) 
= \sum_{d \in H_2^G(X)} q^d \frac{ \prod_{j=1}^k \prod_{m = -\infty}^{\lan \mu_j, d \ran} (\mu_j
  + m\hbar) 
}
{ \prod_{j=1}^k \prod_{m=-\infty}^0 (\mu_j + m \hbar) 
}.
\]
The localized gauged graph potential $\tau_{X,-}^G(1,\hbar,q)$ is
defined by virtual enumeration of quasimaps, or more generally,
Mundet-stable maps to the quotient stack.  The localized gauged
potential is the solution to the Gelfand-Kapranov-Zelevinsky
hypergeometric system, see for example Iritani \cite{iri:gmt},
Cox-Katz \cite[(11.92)]{ck:ms}.  The Batyrev relations correspond to
differential operators of that system,
\[ \Box_d = \prod_{\lan \mu_j, d \ran \ge 0} \partial_j^{\lan \mu_j, d \ran} - q^{\lan \omega, d \ran}
 \prod_{\lan \mu_j, d \ran \le 0} \partial_j^{-\lan \mu_j, d \ran} .\]
Now $\tau_{X \qu G,-}$, as in Givental \cite{gi:eq} is a fundamental
solution for the quantum differential equation on $X \qu G$.  It
follows that any differential operator that annihilates the localized
gauged graph potential, is transformed via $D_\omega \kappa_X^G$ to a
differential operator annihilating the fundamental solution, and so
defines a relation.

The following is a direct, geometric argument using the fundamental
property of the Euler class.  Define bundles
\[E_\pm := \bigoplus_{(\pm \mu_j(d) \ge 0} \ev_1^* \C_{\mu_j}^{\oplus
  \mu_j(d)} .\]
The Euler class of \(E_\pm\) is
\[ \zeta_\pm(d) := \Eul(E_\pm) = \prod_{ \pm \mu_j(d) \ge 0} \mu_j^{
  \mu_j(d)} .\]
Let \(\sigma\)
denote the section of \(E_+\)
given by the taking the derivatives
\(\sigma_{i,j}, i = 1,\dots, d_j = \min(\mu_j(d),\mu_j(d'))\)
of $u$ at $z = z_1$.  Consider the diagram
\begin{equation*}
\begin{tikzcd}[every arrow/.append style={-latex}]
  \sigma^{-1}(0) \arrow{r}{\iota} \arrow{d}{\delta}  & \ovl{\M}_1^G(\bA,X,d')  \\
.\ovl{\M}_1^G(\bA,X,d'-d) &
\end{tikzcd}	
\end{equation*}
defined as follows.  The map \(\iota\)
is the inclusion.  To construct \(\delta\),
note that \(\sigma^{-1}(0)\)
consists of maps whose \(j\)-th
component vanishes to order \(d_j\)
at the marking \(z_1\).
Therefore, for any \([u] \in \sigma^{-1}(0)\)
we may define new map of degree \( d' - d\)
by dividing by the \(i\)-th
component of \(u:C \to X/G\)
on the component of \(C\)
containing \(z_1\)
by \((z-z_1)^{d_i}\)
on the component $C_1$ containing \(z_1\),
and otherwise leaving the map unchanged, to obtain a map denoted
\((z - z_1)^{-d} u \).
Note that this is a change of \(u\)
by the action of an element of the one-parameter subgroup
corresponding to \(d\).  Thus there is a canonical map
\[ \delta: \sigma^{-1}(0) \to \ovl{\M}_1^G(\bA,X,d'-d), \quad u
\mapsto [u / (z - z_1)^d] .\]
Similarly the Euler class of the normal bundle to $\delta$ is, if
non-empty, the product of classes
$\mu_j^{ \min(\mu_j(-d), \mu_j(d'- d))}$ for $\mu_j(d) < 0$.

The remaining factors are accounted for by the difference in perfect
obstruction theories.  We denote by \(p^{d'}\)
the restriction of the map \(p\)
from the universal curve to maps of homology class \(d'\).
To compute the difference in classes we note that (if
\(e^{d'}, e^{d' - d}\) denote the universal evaluation maps)
\begin{eqnarray*} 
 \iota^* [R p^{d'}_{*} e^* T(X/G)] - \delta^* [Rp^{d' - d}_* e^* T(X/G)]
&=& \iota^* [ R p^{d'}_*
     \sum_j (\mO_{z_1} (\mu_j(d')) ] \\&& 
- \delta^* [R p^{d'  - d}_* \sum_j \mO_{z_1}(\mu_j(d' - d)) ] \\
 &=& \iota^* [E_+]- \delta^*  [ E_-]
   .\end{eqnarray*}
 Hence for any class \(\alpha_0 \in H(X \qu G)\) we obtain
\[ 
\int_{\ovl{\M}_1^G(\bA,X,d') } \ev_0^* \alpha_0 \otimes \ev_1^*
\zeta_+(d) = \int_{ \ovl{\M}_1^G(\bA,X,d' - d) } \ev_0^* \alpha_0
\otimes \ev_1^* \zeta_-(d) .\]
By definition of the quantum Kirwan map this implies 
\[ D_0 \kappa_X^G (\zeta_+(d)) = q^d D_0 \kappa_X^G (\zeta_-(d)) .\]
which proves the claim.
\end{proof}

Similar results were obtained for toric manifolds under the name of
mirror theorems in, for example, Iritani \cite{iritani:conv}, by
writing the toric variety as a complete intersection in another Fano
toric variety and applying the Givental formalism, and for toric
orbifolds in Coates, Corti, Iritani, and Tseng \cite{coates:mirror},
\cite[Theorem 5.13]{coates:hodge}.

\begin{example} 
 {\rm (The second Hirzebruch surface, as in \cite[Example
     3.5]{gon:seidel})} The second Hirzebruch surface $F_2$ is a
 quotient of $X = \C^4$ by the $G = \C^{\times,2}$-action with weights
 $(0,1),(-2,1),(1,0),(1,0)$.  Let $p_1,p_2 \in H_2(X \qu G)$ be the
 the zero section and fiber classes.  The divisor classes are
\[ D_1 = p_2, \quad  D_2 = p_2 - 2p_1, \quad  D_3 = D_4 = p_1 .\]
In the case of nef toric varieties it follows from the adiabatic limit
theorem of \cite{qkirwan1,qkirwan2,qkirwan3} that the quantum Kirwan map intertwines the
$I$-function and $J$-function of Givental \cite{gi:eq}, and so must
agree with the mirror transformation for the second Hirzebruch
surface, computed in Cox-Katz \cite[Example 11.2.5.2]{ck:ms}.  Let
$s_1,s_2$ resp $r_1,r_2$ be the coordinates on the torus with Lie
algebra $H^2_G(X,\C)$ resp. $H^2(X \qu G,\C)$ corresponding to the
basis above.  These are isomorphic via Kirwan's map; the variables
$r_1,r_2$ are called $q_1,q_2$ in most of the mirror symmetry
literature but we wish to avoid confusion with the universal Novikov
parameter $q$. The mirror transformation is
\[ s_1 = r_1 / (1 + r_1)^2, \quad s_2 = r_2(1 + r_1) ;\]
the quantum Kirwan map is obtained from inverting this coordinate
transformation and inserting suitable powers of $q$ in the power
series expansion, determined by the symplectic class $\omega$.  Here
we set $q = 1$ for simplicity; if $\log(r_i) = \log(s_i) +
g_i(r_1,r_2)$, $\log(r) = (\log(r_1),\log(r_2))$ then 
\[ \kappa_X^G( \log(s)) = \log(r) = \log(s) + g(s) \]
so that the ``quantum correction'' to the Kirwan map is $g(s)$.  The
image of the divisors classes under the linearized Kirwan map are
given by differentiating the mirror transformation.  If $\ti{D}_j =
D_\omega \kappa_X^G (\mu_j)$, where $\mu_j$ is the $j$-th weight,
then
\[ 
\ti{D}_1 = D_1, \quad 
\ti{D}_2 = D_2 + 2 \frac{r_1}{1 - r_1}  D_2,  \quad 
\ti{D}_3 = D_3 -  \frac{r_1}{1 - r_1}  D_2,  \quad 
\ti{D}_4 = D_4 - \frac{r_1}{1 - r_1} D_2 ;\]
these are called {\em Batyrev elements} in Gonz\'alez-Iritani
\cite{gon:seidel}.  It was noted in Guest \cite{guest:dmod} for
semi-Fano toric varieties, and Iritani \cite[Section 5]{iritani:conv},
Gonz\'alez-Iritani \cite[Example 3.5]{gon:seidel} these elements satisfy
the Batyrev relations with respect to the variables $s_1,s_2$, which
for $d = (d_1,d_2) \in \Z^2_{\ge 0}, d_2 - 2 d_1 \ge 0$ read
\[ \ti{D}_1^{d_2} \star_{\log(s)} \ti{D}_2^{d_2 - 2d_1}
\star_{\log(s)} \ti{D}_3^{d_1} \star_{\log(s)} \ti{D}_4^{d_1} =
s_1^{d_1} s_2^{d_2} .\]
The effect of using the bulk-deformed quantum product
$\star_{\log(r(s))}$ instead of the small quantum product
$\star_{\log(s)}$ can be computed using the divisor equation, and
leads to the replacement of $s_1^{d_1} s_2^{d_2}$ by $r_1^{d_1}
r_2^{d_2}$ on the right-hand-side.  Indeed, for any classes $\alpha,
\beta \in H(X \qu G)$ and basis $\{ \gamma \}$ for $H(X \qu G)$ with
dual basis $\{ \gamma^\dual \}$
\begin{eqnarray*} 
\alpha \star_{\log(r(s))} \beta  &=& \sum_{n,d,\gamma} \frac{s_1^{d_1}
s_2^{d_2}}{n!} \lan \alpha, \beta, \gamma^\dual, g(s),\ldots, g(s)
\ran_{0,d,n+3}  \ \gamma \\
 &=& \sum_{d,\gamma} s_1^{d_1} s_2^{d_2} \exp( g_1(s)d_1 +
g_2(s) d_2 ) \lan \alpha, \beta, \gamma^\dual \ran_{0,d,3} \ \gamma \\
 &=& \sum_{d,\gamma} 
 (s_1 \exp(g_1(s))^{d_1}  (s_2 \exp(g_2(s)))^{d_2}
\lan \alpha, \beta, \gamma^\dual \ran_{0,d,3} \ \gamma  \\
 &=& \sum_{d,\gamma} 
 (s_1 \exp(g_1(s)))^{d_1}  (s_2 \exp(g_2(s)))^{d_2}
\lan \alpha, \beta, \gamma^\dual \ran_{0,d,3} \ \gamma  \\
 &=& \sum_{d,\gamma} 
 r_1^{d_1}  r_2^{d_2}
\lan \alpha, \beta, \gamma^\dual \ran_{0,d,3}  \ \gamma \\
&=& ( \alpha \star_{\log(s)} \beta)  |_{y = r} .
 \end{eqnarray*}
Thus for the deformed product the Batyrev elements satisfy the Batyrev
relation
\[ \ti{D}_1^{d_2} \star_{\log(r(s))} \ti{D}_2^{d_2 - 2d_1}
\star_{\log(r(s))} \ti{D}_3^{d_1} \star_{\log(r(s))} \ti{D}_4^{d_1} =
r_1(s)^{d_1} r_2(s)^{d_2} \]
which is a special case of Proposition \ref{ker}. It would interesting
to derive the formula for $\kappa_X^G$ above directly from the
geometric definition of the quantum Kirwan map.
\end{example} 

Motivated by considerations from mirror symmetry, Givental
\cite{giv:tmp} and later Hori-Vafa \cite{ho:mi}, proposed a description
of the quantum cohomology in terms of the {\em Jacobian ring} of
functions on the critical locus of a certain function, arising as the
{\em Landau-Ginzburg potential} of the mirror sigma model.  In
particular, Givental \cite{giv:tmp} proved an isomorphism of the
quantum cohomology of a smooth Fano toric variety with the Jacobian
ring.

\begin{proposition} {\rm (Isomorphism of the Batyrev ring with the Jacobian ring
of the naive potential)} \label{ident} 
\[
QH_G(X)/QSR_{X,G} \to \Jac(W_{X,G}), \ [\mu_j] \mapsto [q^{\omega_j}
y_j], j = 1,\ldots, k
\]
 is well-defined and induces an isomorphism.  
\end{proposition}  

\begin{proof} Without the Novikov field, the result is Iritani
  \cite[3.9]{iri:integral}: the linear relations among the weights for
  $\g$ on $X$ correspond to the relations on the coordinate ring of
  $\ti{T}^\dual$ given by the derivatives of the Landau-Ginzburg
  potential $W_{X,G}$. Any such relation is of the form
\[ \sum_{i=1}^k \mu_i \lan \lambda,\nu_i \ran = 0 \]
for some $\lambda \in \t$.  Furthermore the quantum Stanley-Reisner
relations $QSR_{X,G}$ correspond to the relations on the various
coordinates on the big dual torus $\ti{G}^\dual$ restricted to
$\iota_\omega \ti{T}^\dual$:
\begin{eqnarray}
\prod_{\lan \mu_j, d \ran > 0 } 
 \mu_j^{\lan \mu_j, d \ran}  &\mapsto& \prod_{\lan \mu_j, d \ran > 0 } 
q^{  \omega_j \lan \mu_j, d \ran}  
y_j ^{
  \lan \mu_j, d \ran } 
\\ &=& 
q^{  \sum_{\lan \mu_j, d \ran > 0 }  
 \omega_j \lan \mu_j, d \ran}  
\prod_{\lan \mu_j, d \ran > 0 } 
y_j^{\lan \mu_j, d \ran }   \label{line1} \end{eqnarray}
\begin{eqnarray} 
q^{\lan d, \omega \ran} \prod_{\lan \mu_j, d \ran < 0 } 
 \mu_j^{\lan \mu_j, d \ran}  &\mapsto&
 q^{\lan d, \omega \ran} 
\prod_{\lan \mu_j, d \ran < 0 } 
q^{   \omega_j \lan \mu_j, d \ran}  
 y_j^{  \lan \mu_j, d \ran } \\&=& 
q^{d  + \sum_{\lan \mu_j, d \ran < 0 }  
 \omega_j \lan \mu_j, d \ran}  
 \prod_{\lan \mu_j, d \ran < 0 } 
y_j^{ \lan \mu_j, d \ran } . \label{line2} \end{eqnarray}
The quantities \eqref{line1} and \eqref{line2} are equal since
\[ \prod_{j=1 }^k  y_j^{\mu_j} =  1, \quad 
 \sum_{j=1 }^k  \omega_j \lan \mu_j, d \ran =  \lan \omega, d \ran.\]
The claimed isomorphism follows.
\end{proof}  

To compute the dimension of the Jacobian ring we recall the following
theorem of Kouchnirenko's \cite{kouch:poly,at:ang} describing the
number of critical points of a polynomial in several variables.
Consider a function given by restricting a finite
sum of monomials on $\ti{G}^\dual$ to $\ti{T}^\dual$:
\[
W: \ti{T}^\dual \to \C, \quad y \mapsto \sum_{\lambda \in \ti{\g}_\Z}
c_\lambda y^\lambda , \quad c_\lambda \in \C .\]
Let $\pi: \ti{\t}_\Z \to \t_\Z$ denote the projection onto the
free part $\t_\Z$ of $\ti{\t}_Z$.  

\begin{definition}
\begin{enumerate} 
\item {\rm (Newton polytope)} The convex polyhedron
\[\Delta(W) := \on{hull} \{ \pi(\lambda) \in \t_\Z, c_\lambda \neq 0 \} \]
is the {\em Newton polytope} of $W$. 
\item {\rm (Non-degeneracy at infinity)} The function $W$ is {\em
  non-degenerate at infinity} if for any face $F \subset \Delta(W)$,
  the {\em face polynomial}
\[ W_F: T^\dual \to \C, \quad y \mapsto \sum_{\lambda \in F} c_\lambda
  y^\lambda \]
has no critical points.  
\item {\rm (Multiplicity of a critical point)} The {\em multiplicity}
  of an isolated critical point $y \in \Crit(W)$ is the intersection
  multiplicity of $\d W(y)$ at $0$ as in Fulton \cite[Lemma
  12.1]{fu:int}. 
\end{enumerate} 
\end{definition}

\begin{theorem} \label{kouch} {\rm (Kouchnirenko theorem)} 
 Suppose
  that $W : T^\dual \to \C $ is non-degenerate at infinity,
  $\{ \lambda \in \t_\Z, c_\lambda \neq 0 \}$ generate $\t_\Z$, an
  $\{ \lambda - \mu, c_\lambda \neq 0, c_\mu \neq 0 \}$ generate
  $\t_\Z$.  Then the number $\# \Crit(W)$ of zeroes of $\d W$ counted
  with multiplicity is equal to
\[ \# \Crit(W) = \dim(T)!\on{Vol}(\Delta(W)) \# \Gamma
.\]
\end{theorem} 

\begin{proof}[Sketch of proof] 
  We sketch the idea of proof, following the survey of Atiyah
  \cite{at:ang} who reduces the equality to a computation of the
  volume of a certain toric variety.  It suffices to consider the case
  that the generic stabilizer is $\Gamma = \{ 1 \}$, by considering
  the components of $\ti{T}^\dual$ separately.  Consider the closure
  of an orbit of $T^\dual$ on
  $\P( \oplus_{c_\lambda \neq 0} \C_\lambda)$ of the sum of weight
  spaces $C_\lambda$ with non-zero coefficient $c_\lambda$.  The
  function $W: T^\dual \to \C$ extends to a section on the the
  hyperplane bundle $\mO(1)$.  The critical locus $\Crit(W)$ is a
  subset of the intersection of a collection of hyperplanes
  $H_j \subset \P( \oplus_{c_\lambda \neq 0} \C_\lambda)$ defined by
  $y_j \partial_{y_j} W = 0$.  The non-degeneracy condition implies
  that all intersection points $y \in \Crit(W)$ occur in the open
  torus orbit.  It follows that $\# \Crit(W)$ equals the degree of the
  toric variety $X(W)$ associated to $\Delta(W)$.  A standard
  computation shows that the Duistermaat-Heckman measure
  $\Phi_* \Vol_{X(W)} \in \D'(\t_\R^\dual)$, the push-forward of the
  measure $\Vol_{X(W)}$ defined by the Fubini-Study form on $X(W)$ to
  $\t^\dual_\R$, is the characteristic measure $\mu_{\Delta(W)}$ of
  the polytope $\Delta(W)$.  Hence
\begin{eqnarray*}
 \# \Crit(W) &=& \int_{X(W)} \Eul(\mO_{X(W)}(1)^{\oplus \dim(T)})   \\
            &=& \dim(T)! \int_{X(W)} \exp(\Eul(\mO_{X(W)}(1)))  \\
            &=& \dim(T)! \Vol(\Delta(W)). \qedhere
\end{eqnarray*} 
\end{proof} 

We apply Kouchnirenko's Theorem \ref{kouch} to the potential.  Let
$\ti{T}^\dual_g \cong \ti{T}^\dual$ denote the fiber of
$\ti{G}^\dual \to G^\dual$ over $g \in G$ in \eqref{ses2}.  Let

\[ W_{X,G,g}: \ti{T}^\dual_g \to \C , \quad y \mapsto \sum_{j=1}^k
y_j \]
denote the restriction of the sum of coordinate functions
$y_j, j = 1,\ldots, k$ to $\ti{G}^{\dual}_g$. Let $\Jac(W_{X,G,g})$ be
the coordinate ring of $\Crit(W_{X,G,g})$.  Let $G^{\dual,\circ}$ be
the space of parameters $g \in G^\dual$ for which $W_{X,G,g}$ is
non-degenerate at infinity.  Let $\Delta_{X \qu G}^\dual$ given as the
convex hull of the normal vectors $\nu_j$ of facets of
$\Delta_{X \qu G}$.  Then

\begin{lemma} \label{nondeg} {\rm (Iritani \cite[Propositions 3.7,3.10]{iri:integral})}
  $G^{\dual,\circ}$ is a Zariski-open subset of $G^\dual$, and for
  $g \in G^{\dual, \circ} $ the number $\# \Crit(W_{X,G,g})$ of
  critical points of $W_{X,G,g}$, counted with multiplicity is equal
  to
  \[ \# \Crit(W_{X,G,g}) = \# \Gamma \dim(T)!  \Vol(\Delta_{X
    \qu G}^\dual). \]
  Furthermore, the set of points
  $G^{\dual, \circ \circ} \subset G^{\dual, \circ}$ where the critical
    points are non-degenerate is open and dense.
\end{lemma} 

We now compare the number of critical points to the dimension of the
quantum cohomology.  For each cone $C$ of maximal dimension in the fan
$\cC(X \qu G)$, let $\Sigma(C) \subset \t_\R$ denote the simplex
spanned by its generating vectors $\mu_j$ and the origin $0$, and
$\Vol(\Sigma(C)) \in (0,\infty)$ its volume.  An example is shown in
Figure \ref{dual}, where $X \qu G$ is a Hirzebruch surface, the dual
polytope to $\Delta_{X \qu G}$ contains four simplices of volume
$1/2$, each contributing one to the Euler characteristic
$\chi(X \qu G) = 4$, but these simplices do not cover the dual
polytope and so the dual polytope is ``bigger than expected''.

The following computation of the Euler characteristic the quantum
cohomology of a toric stack can be found in Iritani
\cite{iritani:conv} or less explicitly in the earlier paper of
Borisov-Chen-Smith \cite{bcs:tdms}:

\begin{proposition} \label{maxcones} \cite[Chapter 5]{iritani:conv}  We have 
\[ \dim QH(X \qu G) = \dim(X \qu G)! \sum_{C \in \cC(X \qu G)}
\Vol(\Sigma(C)) \] 
where the sum is over cones of maximal dimension.
\end{proposition} 

\begin{proof} Since the odd cohomology vanishes we have
\[ \dim(QH(X \qu G)) = \chi(I_{X \qu G}) = \chi( I_{X \qu G}^T ) =
  \chi(I_{ (X \qu G)^T }) .\] 
  The correspondence between fixed points and cones of maximal
  dimension proves the identity in the case of smooth toric varieties.
  In the stack case one uses in addition that the order of the
  stabilizer $T_{[x]}$ at the fixed point $[x] \in X \qu G$
  corresponding to $C$ is 
\[ \# T_{[x]} = \dim(X \qu G)! \Vol(\Sigma(C)) \] 
see \cite[5.10]{iritani:conv}.
\end{proof}

\begin{corollary} \label{equal} (\cite[3.10]{iri:integral},
  \cite[5.10]{iritani:conv}) The order of $\Crit(W_{X,G,g})$ for
  general $g$ is at least $\dim(QH(X \qu G))$, with equality if and
  only if $c_1(X \qu G) \ge 0$.
\end{corollary} 

Because the dimensions of $QH(X \qu G)$ and $\Jac(W_{X,G,g})$ for
generic $g$ do not match in the non-semi-positive case, see Corollary
\ref{equal}, there must be additional relations, that is, generators
in the kernel of the map $QH_G(X) \to QH(X \qu G)$ in \eqref{mapsto}.
\begin{figure}[ht]
\includegraphics[width=4in]{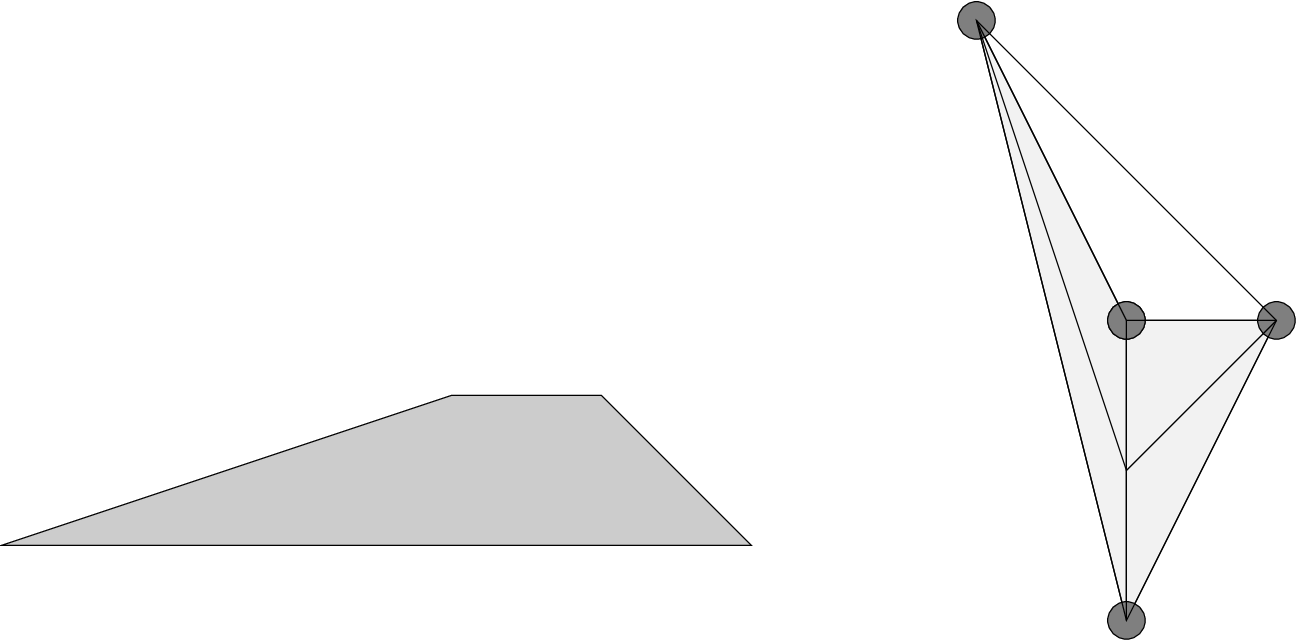}
\caption{A polytope whose dual polytope has too much volume}
\label{dual} 
\end{figure}
\noindent Several authors suggested, and Iritani \cite{iritani:conv} and Fukaya
et al \cite{fooo:tms} proved in the case of toric manifolds, that the
additional relations in the kernel of \eqref{mapsto} correspond to
functions on the critical points $y \in \Crit(W_{X,G})$ outside of the
moment polytope $\Delta_{X \qu G}$.  These extra generators can be
removed by using the formal version of the Jacobian ring of the
potential introduced in Definition \ref{jac} \eqref{jacplus}.

\begin{proposition} {\rm (Isomorphism of formal Batyrev and Jacobian rings)}  
\label{batjac} 
The map from $T_\omega QH_G(X)/QSR_{X,G}$ to $\Jac(W_{X,G})$ of Proposition
\ref{ident} extends to an isomorphism from 
$  T_\omega \widehat{QH}_G(X)/\widehat{QSR}_X^G $ to $ \hJac(W_{X,G})
.$
\end{proposition} 

\begin{proof} It suffices to check that the degrees giving the
  filtration used to define completions on both
  $T_\omega QH_G(X)/QSR_{X,G}$ and $\Jac(W_{X,G})$ of Proposition
  agree.  The degree of each weight $\mu_i \in QH_G(X)$ is one, while
  the degree of its image $q^{\omega_i} \ti{g}_i$ in the grading of
  the Jacobian ring given by \eqref{ses2} is also one, by definition.
  Extension of the isomorphism to the completions
  $ T_\omega \widehat{QH}_G(X)/\widehat{QSR}_X^G $ and
  $ \hJac(W_{X,G})$ follows.
 \end{proof}

 We now explore the meaning of the positive part of the critical locus
 more geometrically, in terms of the moment map.

\begin{definition}  {\rm (Tropical moment map)} 
\label{tropmom}
The {\em tropical moment map} is the map obtained by taking
$q$-valuations
\[ \Psi: \ \ti{T}^\dual(\Lambda) \to \t_\R^\dual, \quad ( y_1,\ldots,
y_k) \mapsto (\on{val}_q(y_1),\ldots, \on{val}_q(y_k)) .\]
where $\t_\R^\dual$ is considered a subspace of $\ti{\g}_\R^\dual$ via
\eqref{ses2}.
\end{definition} 

\begin{definition}  \label{minfacet}  {\rm (Minimal facets)}  
Let $\lambda \in \t^\dual_\R$.   Let 
\begin{equation} \label{closest} 
I(\lambda) = \Set { i \in \{ 1, \ldots, k \} | \lan \lambda , \nu_i \ran + \omega_i =
  \inf \{ \lan \lambda , \nu_j \ran + \omega_j \in \R, \quad j = 1,\ldots,
  k } \end{equation}
denote the indices of the ``closest'' facets to $\lambda$.  More
generally, for any subspace $\lie{s} \subset \t^\dual_\R$, denote by
$I(\lambda,\lie{s})$ the set of facets minimal for $\lambda$ among
those with $ \nu_j | \lie{s} \neq 0$.
\end{definition}  

\begin{remark} For a generic point $\lambda \in \t^\dual_\R$, there
  will be a unique closest facet so $I(\lambda)$ will have order $1$.
  Each point $\lambda$ on the boundary of $\Delta_{X \qu G}$ has
  minimal facets $I(\lambda)$ equal to the set of facets of
  $\Delta_{X \qu G}$ containing $\lambda$. The same holds in a
  neighborhood of the boundary $\partial \Delta_{X \qu G}$, by continuity.  For examples of points
  $\lambda$ with more than $\dim(T)$ minimal facets, see Figure
  \ref{p1p1movie}.
\end{remark}

\begin{proposition}
\label{tropmommap}
Let $y \in \Crit(W_{X,G})$ be a critical point.  The tropical moment
map $\zeta = \Psi(y) \in \t^\dual_\R$ has the property that for any
$\lie{s}$ with $I(\zeta,\lie{s}) \neq \emptyset$, the normal vectors
$ \nu_j, j \in I(\zeta,\lie{s}) $ are linearly dependent after
restriction to $\lie{s}$ in $\t^\dual_\R$.
\end{proposition}  

\begin{proof} We take the derivative of the potential: For
  $\lambda \in \lie{s}$ and $y \in \Crit(W_{X,G})$
  \begin{equation} \label{deriv} 0 = \partial_\lambda W_{X,G}(y) =
    \sum_{j=1}^k q^{\omega_j} y^{\nu_j} \lan \nu_j, \lambda \ran
    .\end{equation}
  In particular the leading order powers of $q$ in \eqref{deriv} must
  cancel.  Thus
\[\forall \lambda, \  \sum_{j \in I(\zeta,\lie{s})} y^{\nu_j} \lan
\nu_j,\lambda \ran = 0 
\quad \text{so}  \sum_{j \in I(\zeta,\lie{s})} \on{val}_q(y^{\nu_j})
\nu_j = 0  \] 
so the vectors $\nu_j$ are dependent after restriction to the subspace $\lie{s}$.
\end{proof}

\begin{example}  \label{prodlines} 
\begin{enumerate} 
\item 
{\rm (Tropical moment map for critical locus for a product of
    projective lines)} Let $X= \C^4$ with $G = (\C^\times)^2$ acting
  with weights $(1,0),(1,0)$, $(0,1),(0,1)$. Consider the reduction at
  $\omega =(2,1)$ then $X \qu G = \P^1 \times \P^1$ with moment
  polytope $[4,0] \times [2,0]$ and normal vectors
  $(1,0),(-1,0),(0,1),(0,-1)$. The critical points are
  $(y_1,y_2) = (\pm q^2, \pm q)$ which have valuations (leading order
  $q$-powers) $(2,1)$, all mapping to the barycenter of the moment
  polytope.  We have $I(\lambda) = \{ 3,4 \}$, the facets closest to
  the critical point.  If $\lie{s} = \on{span} (1,0)$ then
  $I(\lambda,\lie{s}) = \{ 1,2 \}$.  Note that the vectors
  $\nu_j, j \in I(\lambda)$ or $I(\lambda,\lie{s})$ are dependent.
\item {\rm (Tropical moment map for critical locus for a family of
  toric surfaces)} Suppose that $X = \C^5 $ with $G = (\C^\times)^3$
  acting with weight matrix
\[\left[
\begin{array}{rrrrr} 
1&  1&  0&  0&  0 \\
0&  0&  1&  1&  0\\
-1&  0&  -1&  0&   1 \end{array} \right] .\]
For a suitable choice of $\omega$ the quotient $X \qu G$ is the
blow-up of projective lines $\P^1 \times \P^1$ at a fixed point, say
$([1,0], [1,0])$, with moment polytope
\[ \Delta_{X \qu G} = \{ (\mu_1,\mu_2) \in [0,4] \times [0,2] \ |
  \ \mu_1 + \mu_2 \ge \eps \} .\]
  For $\eps < 1$, there are two possible values of $\Psi $ on
  $\Crit(W_{X,G})$: one critical point maps to $(\eps,\eps)$, while
  four other critical points map to $(2,1)$.  For $1 < \eps$, one
  critical point maps to $(2 - \eps,\eps)$.  The others map to
  $( (3 + \eps)/2,1)$.  See Figure \ref{p1p1movie}.  The case
  $\eps = 1$ is special: in this case, one can obtain a line segment
  of critical values $\Psi(y), y \in \Crit(W_{X,G})$ by varying the
  ``bulk deformation'', see \cite{fooo:toric2}.  This is shown as a
  dotted line connecting the two critical values in the Figure
  \ref{p1p1movie}.
\end{enumerate} 
\end{example} 

\begin{figure}[ht]
\includegraphics[width=4in]{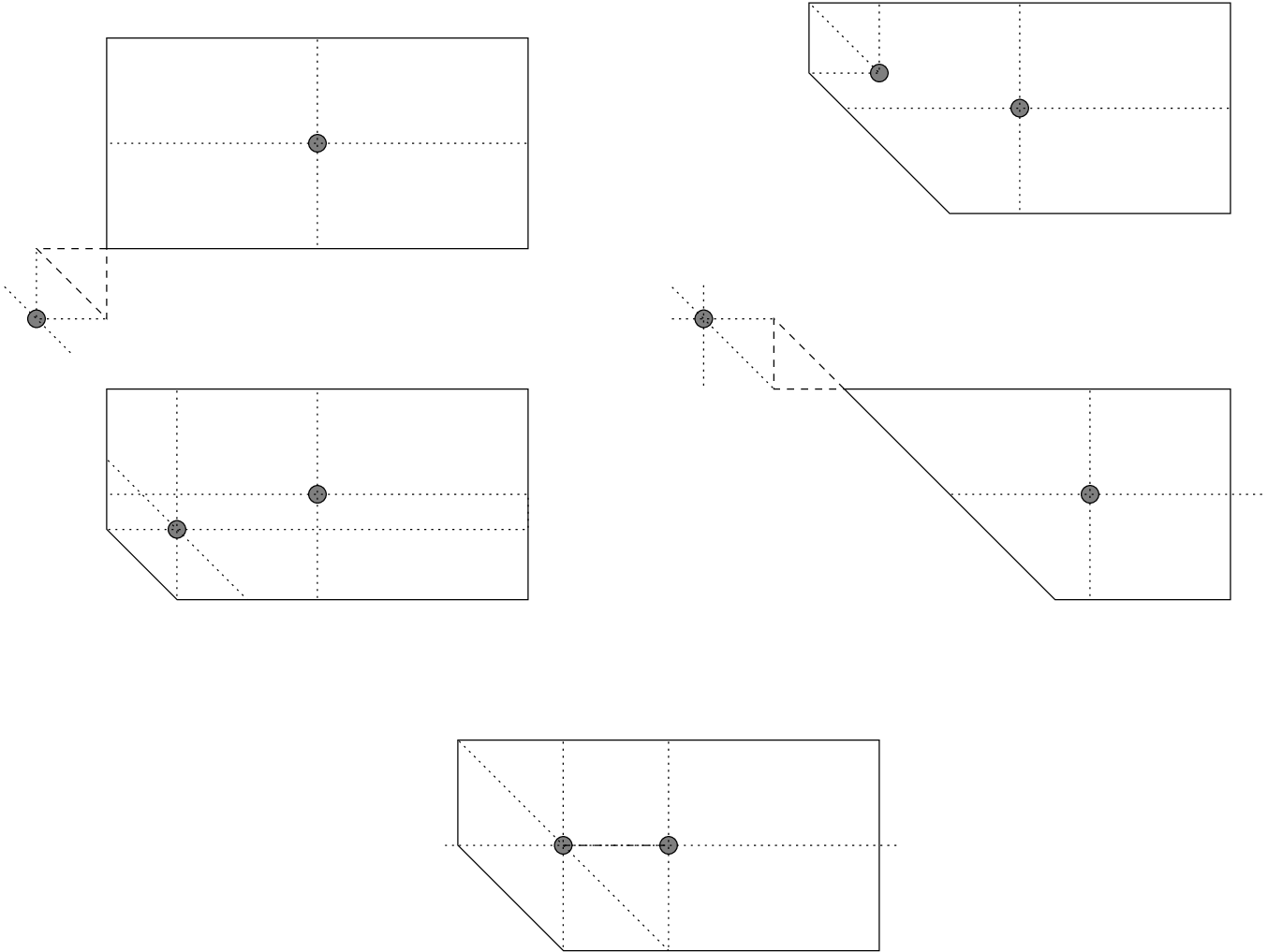}
\caption{Values of the tropical moment map on the critical locus for
  a family of toric surfaces}
\label{p1p1movie}
\end{figure}

\begin{proposition} \label{closure} A point $y \in \Crit(W_{X,G}) $
  lies in $\Crit_+(W_{X,G})$ iff $\Psi(y) \in \Delta_{X \qu G}$.
\end{proposition}

\begin{proof} It suffices to consider the case $\Gamma $ trivial.  The
  $j$-th coordinate of $y$ under the embedding
  $T^\dual \to \ti{G}^\dual$ is $y_j = y^{\nu_j}$.  The shift
  $y_j q^{\omega_j}$ has $q$-valuation
  $\lan \nu_j, \Psi(y) \ran + \omega_j$.  Thus $y_j q^{\omega_j}$ goes
  to zero as $q \to 0$ iff $\lan \nu_j, \Psi(y) \ran > - \omega_j$.
\end{proof}

\begin{remark} By the results of \cite{fooo:toric1},
  \cite{fooo:toric2}, \cite{wo:gdisk}, the image of $\Crit_+(W_{X,G})$
  in $\t_\R^\dual$ consists of moment values such that the
  corresponding Lagrangian moment fiber is Hamiltonian
  non-displaceable, see Section \ref{discuss}.
\end{remark}

\section{Dimensional equality via a toric minimal model program} 

In this section we show that the linearized quantum Kirwan map is
injective after passing to the formal completion and modding out by
the quantum Stanley-Reisner ideal.  By the surjectivity result in
Theorem \ref{weaktoric2} it suffices to show the equality of
dimensions
\begin{equation} \label{equaldims}
 \dim QH(X \qu G) = \dim
 {\Jac_+}(W_{X,G}) .\end{equation}
In the case that $X \qu G$ is Fano and minimally presented as a
quotient of $X$ by $G$ (that is, every weight space in $X$ defines a
prime divisor of $X \qu G$) this is a consequence of Kouchnirenko's
theorem, see Corollary \ref{equal}.  

To reduce to the Fano case, we apply the toric minimal model program
introduced by M. Reid \cite{reid:decom}.  More precisely, we vary the
K\"ahler class $[\omega]$ by a multiple $-t c_1(X \qu G)$ of the
canonical class $c_1(X \qu G)$ until we obtain a Fano fibration,
showing that the wall-crossings on both sides of \eqref{equaldims} are
the same.  We wish to emphasize that, although we are using the
language of toric minimal models, in fact all of our results are
completely {\em combinatorial}, that is, could be phrased entirely in
terms of fans.  However, we find the geometric story accompanying the
combinatorics rather helpful.  First, recall the general phenomenon of
wall-crossing in the context of geometric invariant theory quotients,
as in Dolgachev-Hu \cite{do:va} and Thaddeus \cite{th:fl} in which
{\em flips} occur as the polarization defining the quotient is varied.

\begin{notation} {\rm (Family of git quotients)}  
Let
$\omega_t \in H^2_G(X,\Q), t \in [0,1]_\Q$
be an affine linear path of K\"ahler classes, corresponding to a path
of rational polarizations (ample $G$-line bundles) $L_t \to X$.  For
any $t \in [0,1]_\Q$ let
\[ X^{t,\ss} = \bigcup_{k > 0 , s \in H^0(X,L_t^k)^G} \{ s \neq 0 \} \]
be the semistable locus, and assume that $G$ acts with finite
stabilizers on $X^{t,\ss}$ for $t = 0,1$.  Then the $G$ acts with
finite stabilizers on $X^{t,\ss}$ for generic $t \in [0,1]$.  For such
$t$ denote by
\[X \qu_t G :=X^{t,\ss}/G\] 
the stack-theoretic git quotient with respect to the corresponding
polarization.  The stack $X \qu_t G$ is a smooth proper
Deligne-Mumford stack with projective coarse moduli space.
\end{notation}

\begin{proposition}  {\rm (Wall-crossing for git quotients)}  
\label{with} 
With $G,X,L_t$ as above.
\begin{enumerate} 
\item {\rm (Walls)} there exists a finite collection $t_1,\ldots,t_n
  \in (0,1)$ of {\em singular values a.k.a walls} such that there
  exist semistable points that are not stable;
\item \label{chambers} {\em (Chambers)} the isomorphism class of the
  quotient $X \qu_t G$ is independent of $t$ for
  $t \in (t_j,t_{j+1}), j = 1,\ldots, n - 1$;
\item {\em (Wall-crossing)} Suppose that stable=semistable for the
  $G$-action on $\P(L_0 \oplus L_1)$.  Then as $t$ passes through a
  singular value $t_j$, the quotient $X \qu_t G$ goes through a
  stacky-weighted blow-down and blow-up over a {\em center}
  $Z \subset X \qu_{t_j} G$.
\end{enumerate} 
\end{proposition} 

See Figure \ref{mori} for an example of the change in moment polytopes
under such a variation; the toric case is discussed further in
\cite{reid:decom}, \cite[Chapter 14]{mats:mori}.  

\begin{remark} \label{weighted} Suppose that $X \qu_t G$ is a family
  of toric quotients and $t \in (0,\infty)$ a singular value.  The
  singularity in $X \qu_t G$ is necessarily created by an intersection
  of facets of $\Delta(X \qu_t G)$ with linearly dependent normal
  vectors $\nu_j, j \in I(t)$.  Let 
\[ Z  = \bigcap_{j \in I(t)} D_j  \subset X \qu_t G \]  
denote the intersection of the prime divisors $D_j, j \in I(t)$.
Denote the morphisms to the singular quotients on the level of coarse
moduli spaces
  \[\pi_\pm: X \qu_{t_\pm} G \to X \qu_t G .\]
  Denote the exceptional loci $E_\pm := \pi_\pm^{-1}(Z)$.  The
  restrictions of $\pi_\pm$ to the exceptional loci
\[ \pi_\pm | E_\pm: E_\pm \to Z \] 
are fiber bundles.  The fibers $\pi_\pm^{-1}(z), z \in Z$ are weighted
projective stacks $\P(I_\pm)$ corresponding to the subset of weights
\[ I_\pm \subset \{ \nu_j, j \in I(t) \} \]
corresponding to prime divisors not containing $E_\pm$.  The fan
$\cC(Z)$ of the center $Z$ is given by the projections of the cones
$\cC_Z(X \qu_t Z)$ of $\cC(X \qu G)$ corresponding to orbits meeting
$Z$ to the Lie algebra $\g' = \g/\on{span}(\nu_1,\ldots,\nu_k)$. The
morphism of toric varieties $\pi_\pm$ corresponds to a morphism of
fans 
\[ \cC(X \qu_{t_\pm} G) \to \cC(X \qu_t G) \] 
that is an isomorphism over the complement of the cones in
$\cC( X \qu_t G)$ containing $C(Z) \in C(X \qu_t G)$.
\end{remark}

We will be particularly interested in the variation of git quotient
for toric quotients corresponding to the anti-canonical class.

\begin{notation} {\rm (Anticanonical variations)} A path
  $\omega_t \in H^2_G(X), t \in [0,T]$ is an {\em anticanonical
    variation of symplectic class} if $ \ddt \omega_t = - c_1^G(TX) $,
  see Figure \ref{mori}.  Note that since $TX$ is a sum of line
  bundles with $\ti{G}$-weights $\eps_i$,
\[  c_1^{\ti{G}}(TX) = \sum_{i=1}^k \eps_i .\]  
The variation of symplectic class $\omega$ can be taken to be
\[ (\omega - t c_1^G(TX))_i = \omega_i - t, \quad i = 1,\ldots,k .\]
Hence the facets of the polytope $\Delta_{X \qu_t G}$ ``vary at the
same rate'':
\[ \Delta_{X \qu_t G} = \{ \mu \in \t_\R^\dual \ | \ \lan \mu, \nu_j 
\ran \ge - \omega_j + t, j = 1,\ldots, k \} .\]
\end{notation} 

\begin{figure}[ht]
\includegraphics[height=1.3in]{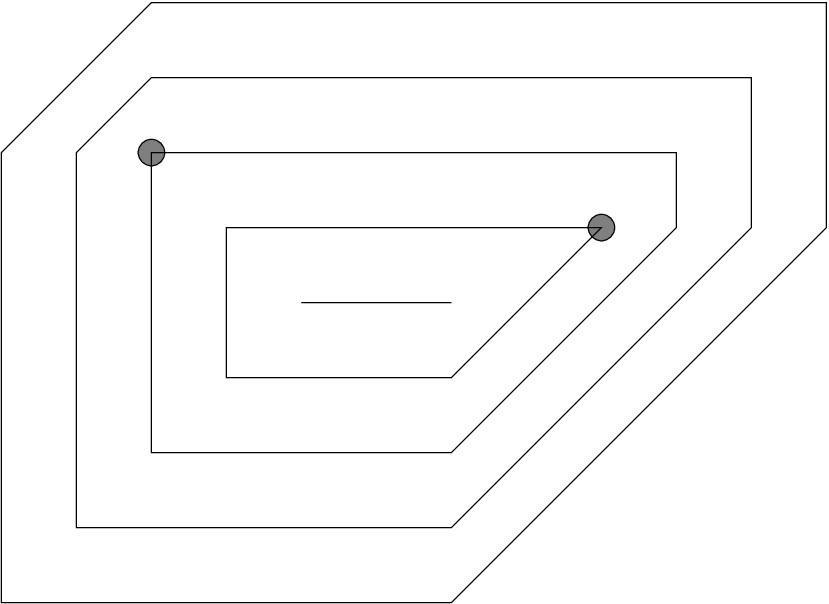}
\caption{Polytopes for a toric minimal model program}
\label{mori} 
\end{figure}

The sequence of a toric varieties obtained in this way is a special
case of the minimal model program described in the toric case by Reid
\cite{reid:decom}; see \cite[Chapter 15]{cox:toric} for more
references.  An example of the family of polytopes
$\Delta_{X \qu_t G}$ obtained in this way is shown in Figure
\ref{mori}; the corresponding fans $\cC(X \qu_t G)$ are shown in
Figure \ref{morifans}. 

\begin{figure}[ht]
\includegraphics[height=1.2in]{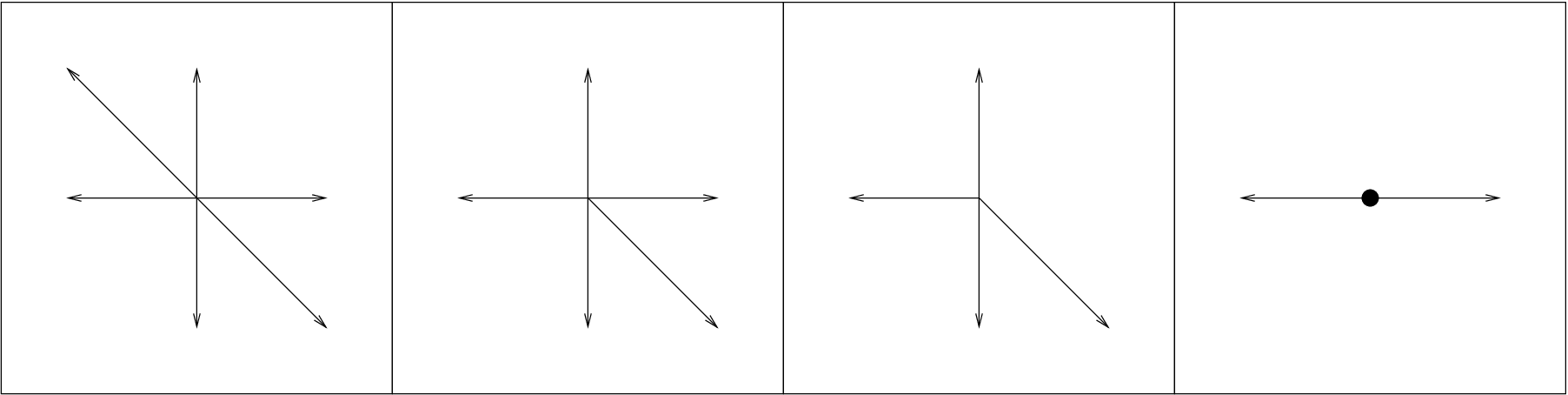}
\caption{Fans for a toric minimal model program}
\label{morifans} 
\end{figure}

The following is the key property of generic runnings of the toric
minimal model program.

\begin{proposition} \label{generic} For a generic symplectic class
  $\omega$, at any singular value $t = t_1,\ldots, t_n $, given a
  collection of normal vectors $\nu_j, j \in I(t)$ that span a
  subspace of dimension
  \[ \dim(\on{span}\{ n_j, j \in (t) \})  = \ell \] 
  the following holds: any point $\mu$ in $\t^\dual$ satisfies at most
  $\ell + 1$ inequalities
  $\lan \mu, \nu_j \ran \ge \omega_j + t, j \in I(t)$ with strict
  equality.
\end{proposition} 

\begin{proof} For each subset $I$ such that the vectors
  $\nu_j, j \in I$ span a subspace of $\t_\R$ of size $l$, the space
  $\sS_I$ of tuples $(\omega,\lambda,t)$ where such that at least
  $l + 2$ equalities \eqref{closest} hold is a union of affine
  subspaces of dimension at most $\dim(H^2_G(X)) - 1$.  The projection
  $p(\sS_I)$ of $\sS_I$ under
  $H^2_G(X) \times \t_\R^\dual \times \R \to H^2_G(X)$ is a proper
  affine subspace of codimension at least $1$.  Taking the union over
  all possible subsets $I \subset \{ 1, \ldots, k \}$ proves the
  claim.
\end{proof}

We will need the following explicit description of the flips arising 
from variation in the anticanonical direction.  

\begin{definition} {\rm (Flipping simplex)}
  \label{flip} 
  Suppose that $\omega_t$ is a path of classes as above with
  $\ddt \omega_t = - c_1^G(X)$.  Let $t \in (0,\infty)$ be a singular
  value, corresponding to an intersection
  $\cap_{j \in I(t)} F_{j,t} \subset \Delta_{X \qu_t G}$ of facets
  $F_{j,t}$ with indices $j \in I(t)$:
\[ F_{j,t} = \Set{ \mu \in \Delta_{X \qu_t G} | \lan \mu, \nu_j \ran =
  - \omega_j + t } .\]
  Let $\Sigma_t$ denote the {\em flipping simplex}
  \begin{equation} \label{fsim} \Sigma_t := \on{hull}( \nu_j , j \in
    I(t)) \subset \t_\R \end{equation}
  that is the convex hull of normal vectors $\nu_j$ corresponding to
  the facets $F_{j,t}$.
\end{definition}

\begin{lemma} \label{stacky} For each singular time
  $t = t_1,\ldots, t_n $, the leading order term potential
  $W_t(y) = \sum_{j \in I(t)} y_j, $ has only non-degenerate critical
  points.
\end{lemma} 

\begin{proof} We may suppose $\Gamma = \{ 1 \}$ and
  $I = \{ 1,\ldots, \ell + 1\}$.  By the linear dependence assumption,
  there exist $c_1,\ldots, c_\ell \in \R$ such that
  $\nu_{\ell+1} = c_1 \nu_1 + \dots + c_\ell \nu_\ell$.  The equations
  defining the critical locus $\Crit(W_{X,G})$ are the partial
  derivatives with respect to the local coordinates
  $y^{\nu_1},\ldots, y^{\nu_\ell}$
  \begin{equation} \label{tcut} y^{\nu_i} = c_i y^{\nu_{\ell+1}}, \quad i
    =1 ,\ldots, \ell . \end{equation}
 These are solutions to the single equation
\[z^{c_1 + \ldots + c_\ell + 1} = 1, \quad z= c_i y^{\nu_{\ell+1}}, \ i =
1,\ldots, \ell \] 
and as such are transversally cut out. 
\end{proof}

\begin{proposition} [Explicit description of flips for the toric minimal model program]
\label{explicit}
Let $X \qu_t G$ be as above so that the condition of Proposition
\ref{generic} is satisfied.  For each singular value $t$, one of the
two possibilities holds:
\begin{enumerate} 
\item {\rm (Fibration case)} Suppose that the flipping simplex
  $ \Sigma_t$ of \eqref{fsim} contains the origin $0$.  Then
  $X \qu_t G$ undergoes a {\em Mori fibration} with fiber a Fano toric
  stack $(X \qu_t G)'$ over a toric stack $(X \qu_t G)''$ of lower
  dimension, and the symplectic class on the fiber is a multiple of
  the first Chern class $c_1( ( X \qu_t G)') $.
\item \label{flipordiv} {\rm (Flip or divisorial contraction case)}
  Suppose that the flipping simplex $\Sigma_t $ does not contain $0$.
  Then $X \qu_t G$ undergoes a {\em flip} (resp. {\em divisorial
    contraction}).  That is, $X \qu_t G$ undergoes a stacky-weighted
  blow-down followed by stacky-weighted blow-up (resp. stacky-weighted
  blow-down only) over a center $Z \subset X \qu_{t} G$. 
\end{enumerate} 
\end{proposition}

\begin{proof} 
  The fibration case is straight-forward and left to the reader.  For
  the flip/divisorial contraction case, see \cite[Lemma
  15.3.11]{cox:toric}, \cite[Proof of Proposition 14-2-11]{mats:mori}
  and especially \cite[Figure 14-2-12]{mats:mori}.  One can give a
  proof using variation of git as follows: Let $G \times \C^\times$
  act on $\ti{X} = X \times \C = \C^{k+1}$ with weights
  $\ti{\mu}_j = (\mu_j, 1)$ and $(0,1)$ and polarization vector
  $\ti{\omega} = (\omega,0)$.  The ``master space'' given by the
  quotient $\ti{X} \qu G$ has a residual $\C^\times$-action whose
  quotients are the git quotients $X \qu_t G$.  The transition times
  correspond to the fixed point components
  $Z \subset (\ti{X} \qu G)^{\C^\times}$; each is necessarily a
  subvariety of $X \qu_t G$ obtained by first restricting to the locus
  $X^\chi$ where $\C^\times$ acts by a character $\chi$ of $G$ and
  taking the git quotient.  The normal bundle $N$ to
  $(\ti{X} \qu G)^{\C^\times}$ splits into the sum of negative
  resp. positive weight sub-bundles $N_-$ resp. $N_+$, each of which
  is quotient of the sum of negative resp. positive weights in
  $X/X^\chi$ under the action of $\C^\times$.  The weighted blow-down
  and blow-up involved from passing to $X \qu_{t_-} \C^\times $ to
  $X \qu_{t_+} \C^\times $ replaces the projectivization $\P_-$ of the
  sum of the weight bundles $N_-$ with the projectivization $\P_+$ of
  the sum of the weight bundles $N_+$.  The claim follows.
\end{proof} 

An example is shown in Figure \ref{morifans2}, continuing that in
Figures \ref{mori}, \ref{morifans}, where the flipping simplices for
each step are shaded.

\begin{figure}[ht]
\includegraphics[height=1.2in]{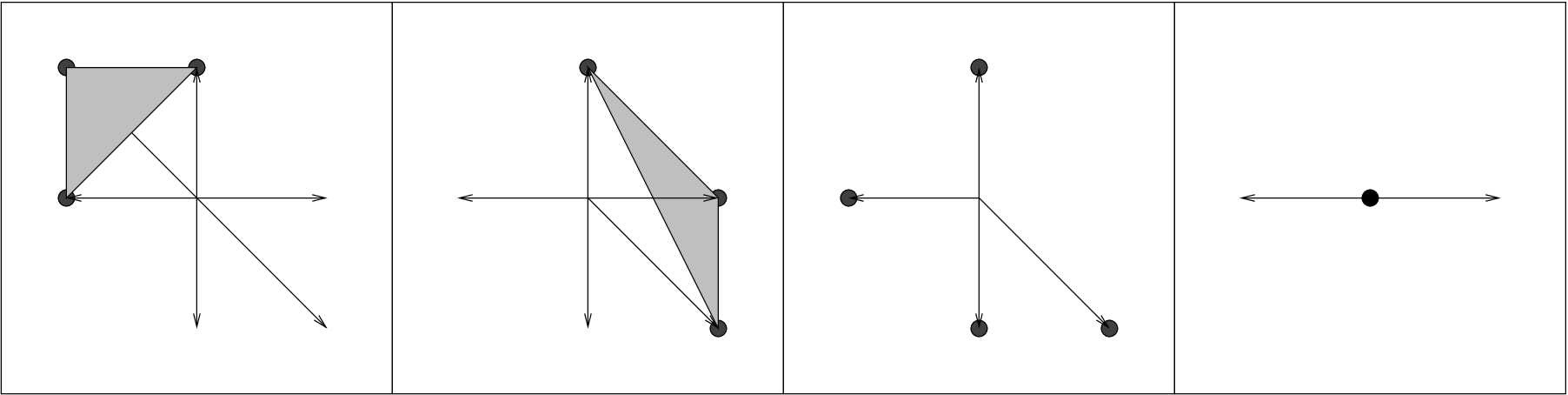}
\caption{Flipping simplices for a toric minimal model program}
\label{morifans2} 
\end{figure}

We now describe the change in the critical locus of the potential as
the toric stack undergoes an mmp transition.  In the fibration case
the fan $\cC( X \qu_{t_-} G)$ admits a morphism to the fan
$\cC( (X \qu_t G)')$ of the base, that is, each cone for
$X \qu_{t_-} G$ maps to a cone, possibly of lower dimension, of
$(X \qu_t G)'$; and the cones of $\cC( X \qu_{t_-} G)$ that map to the
origin form the fan of the fiber $(X \qu_t G)''$.
  \label{stackychange} By Proposition \ref{generic}, the fan
  $\cC( (X \qu_t G)')$ for the fiber $(X \qu_t G)'$ has a minimal
  number of generators $\nu_j$.  That is, the corresponding polytope
  is a simplex $\Delta( (X \qu_t G)')$ and the corresponding toric
  stack $(X \qu_t G)'$ is a stacky weighted projective space.

\begin{lemma} \label{Wfiber} {\rm (Critical loci of the
    Landau-Ginzburg potentials for fibrations)} Let $X,G, \omega$ be
  as above.  Suppose that $t \in (0,\infty)$ is a singular value so
  that $X \qu_t G$ with potential $W_{X \qu_t G}$ undergoes a
  fibration over a toric variety $ (X \qu_t G)'$ of lower dimension
  with fiber $(X \qu_t G)''$, with potentials $W_{(X \qu_t G)'}$ and
  $W_{(X \qu_t G)''}$.  Then there is a canonical bijection
\[ \Crit_+(W_{X \qu_{t_-} G}) \to \Crit_+(W_{ (X \qu_t G)'}) \times
  \Crit_+(W_{(X \qu_t G)''}) .\]
\end{lemma} 

\begin{proof} We may assume that $\Gamma = 1$, by treating the
  components of $\ti{T}^\dual$ individually.  The fibration of
  $X \qu_{t_-} G$ induces a fibration of tori and dual tori,
\[ 1 \to T'' \to T \to T' \to 1, \quad 1 \to T^{\dual,'}
\stackrel{p}{\to} T^\dual \stackrel{r}{\to} T^{\dual,''} \to 1 \]
for which we may choose a splitting.  We write
\[ W_{X \qu_{t_-} G} = W_{X \qu_{t_-} G}' + W_{X \qu_{t_-} G}'' \]
where 
\[ W_{X \qu_{t_-} G}'' = 
  \sum_{F_j \supset \Delta(X \qu_{t_-} G)}  q^{\omega_j} y^{\nu_j} \]
  is the sum of terms corresponding to hyperplanes containing the
  polytope at the singular time, that is, the hyperplanes describing
  the fiber, and $W_{X \qu_{t_-} G}'$ is the sum of the remaining
  terms of $W_{X,G}$, corresponding to facets of the base.  The
  leading order terms in $W_{X \qu_{t_-} G}$ for
  $y \in \Delta_{X \qu_{t_-} G}$ are $W_{X \qu_{t_-} G}''$, which
  restricts to a constant on $T^{\dual,'} \subset T^\dual$.  Thus
  $ \Crit_+(W_{X \qu_{t_-} G})$ maps to $\Crit_+(W_{(X \qu_t G)'})$.
  The fiber of the projection of $\Crit_+(W_{X \qu_{t_-} G})$ onto
  $\Crit_+(W_{(X \qu_t G)'})$ may be identified with
  $\Crit_+(W_{(X \qu_t G)''})$, since $W_{(X \qu_t G)}''$ are the
  non-degenerate terms of leading order, by Proposition
  \ref{stagewise} below.
\end{proof}

\begin{remark} 
  The fibration exact sequence Lemma \ref{Wfiber} only holds for
  critical points lying over the interior of the moment polytope.  In
  other words, there is no fibration of $\Crit(W_{X \qu_{t_-} G})$
  similar to that stated in Lemma \ref{Wfiber}.  Indeed suppose that
  $X \qu G$ is a Hirzebruch surface
  $\P(\mO_{\P^1}(n) \oplus \mO_{\P^1})$ for some $n \ge 0$.  Then
  $ \Crit(W_{X \qu_{t_-} G}) $ has order $2+ 2n$ for $n \ge 2$ (twice
  the volume of the polytope with vertices $(-1,0),(0,-1),(1,n)$) but
  $\Crit(W_{ (X \qu_t G)'}) \times \Crit(W_{(X \qu_t G)''}) $ has
  order $4$.
\end{remark} 

\begin{lemma} \label{fibration} {\rm (Dimension lemma for fibrations)}
  Suppose that $X \qu_t G$ as above with generic initial symplectic
  class $\omega_0$.  Let $t \in (0,\infty)$ be a singular value so
  that $X \qu_t G$ undergoes a fibration over a toric stack
  $ (X \qu_t G)'$ of lower dimension with fiber $(X \qu_t G)''$.  Then
\begin{enumerate}
\item $\dim( QH( X \qu_{t_-} G) ) = \dim( QH( (X \qu_t G)')) \dim(
  QH( (X \qu_t G)'')) $ and
\item $\dim( \hJac( W_{X \qu_{t_-} G} ) = \dim( \hJac( W_{(X \qu_t G)'}))
  \dim( \hJac( W_{(X \qu_t G)'')}) $.
\end{enumerate} 
\end{lemma} 

\begin{proof} (a) 
By Proposition \ref{explicit}, the cones of $X \qu_{t_-} G$ of maximal
dimension are products of the maximal dimensional cones of $(X \qu_t
G)'$ and $(X \qu_t G)''$.  It follows the sum of the volumes of the
maximal dimensional cones of $X \qu_{t_-} G$ is the product of the
corresponding sums for $(X \qu_t G)'$ and $(X \qu_t G)''$.
The conclusion follows from Lemma \ref{maxcones}.   (b)
follows from Lemma \ref{Wfiber}.
\end{proof} 

In order to deal with flips with non-trivial centers with we describe
a stage-wise implicit function theorem due that was communicated to us
by S. Venugopalan. 

\begin{definition} 
  A function $W: \ti{T}^\dual(\Lambda) \to \Lambda$ has a {\em
    stage-wise non-degenerate critical point}
  $y \in \ti{T}^\dual(\Lambda)$ with exponents
  $\tau_1 < \ldots < \tau_n \in \R$ if the following holds: There exist decompositions
\begin{equation*}
  \ti{T}^\dual(\Lambda) =\ti{T}_1^\dual(\Lambda) \times \dots \times
  \ti{T}_n^\dual(\Lambda),
\quad 
  W=W_1+\dots + W_n , \quad 
   y =(y_1,\dots,y_n) 
\end{equation*}
where $T_i^\dual(\Lambda)$ are products of tori with finite groups
$\Gamma_i$,  and 
\begin{enumerate}
\item the map $W_i: \ti{T}^\dual(\Lambda) \to \Lam_{\geq 0} $ factors
  through the projection onto the first $i$ factors
\[ \pi_i: \ti{T}^\dual(\Lambda) \to \Pi_{j=1}^i
\ti{T}^\dual_j(\Lambda), \]
\item the leading order term of $W_i$ has $q$-valuation $\tau_i$, and
\item for $i \ge 2$, the point $y_i \in \ti{T}_i^\dual(\Lambda)$ is a
  non-degenerate critical point of
\[ \ti{W}_i:   := W_i |_{\ti{T}_i^\dual(\Lambda)} :\ti{T}_i^\dual(\Lambda) \to \Lam
.\] 
\end{enumerate}
This ends the Definition.
\end{definition}

\begin{proposition} \label{stagewise} {\rm(Stage-wise implicit
    function theorem)} Suppose $W: \ti{T}^\dual(\Lambda) \to \Lam$ is
  a potential function such that every critical point
  $y \in \Crit_+(W)$ is stage-wise non-degenerate.  For each
  decomposition as above $\ti{T}^\dual =\Pi_{i=1}^n \ti{T}^\dual_i$
  there exists a bijection
 \begin{equation} \label{inject} \prod_{i=1}^n \Crit_+(\ti{W}_i) 
\to \Crit_+(W)  \end{equation}
for sufficiently small $\tau_1,\ldots, \tau_n$.
\end{proposition}
\begin{proof} 
  Given a critical point for the potentials in each direction, wesolve
  for a critical point of the full potential order by order, using
  non-degeneracy of the Hessians.  Let
  $y = (y_1,\ldots,y_n) \in \prod_{i=1}^n \Crit_+(\ti{W}_i) $ be a
  critical point of the stage-wise leading order terms.  Suppose that
  the summands of
  \begin{equation*}
    D_yW=\sum_{i=1}^n D_yW_i
  \end{equation*}
  have leading order terms divisible by
  $q^{\tau_1+\delta}, \dots,q^{\tau_n+\delta}$ for some $\delta>0$.
  We solve by taking a Taylor expansion of $D_{y\exp(z)} W$ at $y$
  \begin{eqnarray*}\label{eq:solve}
    0 & =& D_{y\exp(z)}W  \\&=&  D_yW + \frac{1}{2!} D_y^2W(z,\cdot) + \frac 1 {3!} D_y^3W(z,z,\cdot)+ \text{ higher order in } z.
  \end{eqnarray*}
  We work out the details in case $n=2$. The proof extends naturally
  to higher values of $n$.  We claim that we can solve for $z$
  satisfying 
\[ D_yW + \hh D_y^2W(z,\cdot)=0 .\]  
The operator
\[ \hh D^2_yW:T_y \ti{T}^\dual(\Lambda) \times T_y \ti{T}^\dual(\Lambda)
  \to \Lam \]  
  after a choice of basis respecting the splitting, has a block matrix
  representation
  \begin{equation*}
    \hh D^2_yW=\begin{pmatrix}
      q^{\tau_1}G_{11}& q^{\tau_2}G_{12}\\
      0&q^{\tau_2}G_{22}
    \end{pmatrix}
  \end{equation*}
  where $G_{ij}$ is a $\Lam_{\geq 0}$-valued matrix. By non-degeneracy
  $G_{11}$ and $G_{22}$ are invertible. The matrix $G$ has an inverse
  $G^{-1}:=
  \begin{pmatrix}
    q^{-\tau_1}G^{11}& q^{-\tau_1}G^{12}\\
    0&q^{\tau_2}G^{22}
  \end{pmatrix}
  $, where for any $i$, $j$, $G^{ij}$ is a $\Lam_{\geq 0}$-valued
  matrix.  In matrix notation, the solution $z$ is given by
 \begin{equation*}
    z=((D_yW)G^{-1})^t.
  \end{equation*}
  Since the leading order terms in $D_yW$ are divisible by
  $\tau_1+\delta$, $\tau_2+\delta$, the leading order terms of $z$
  have $q$-valuations at least $\delta$.  The equation
  \eqref{eq:solve} fails to hold only because of the terms quadratic
  or higher order in $z$.  Since the splitting $W=(W_1,W_2)$ is
  divisible by $(q^{\tau_1},q^{\tau_2})$ and $z$ is divisible by
  $q^{\delta}$, the term $\frac 1 {3!} D^3_y(z,z,\cdot)$ and higher
  order terms in \eqref{eq:solve} are divisible by
  $(q^{\tau_1+2\delta},q^{\tau_2+2\delta})$.  Replacing $y$ with
  $y \exp(z)$ and continuing by induction one obtains a solution to
  all orders.  Conversely, given a stage-wise non-degenerate critical
  point $y \in \Crit(W)$ we obtain a critical point $y_i$ for each
  $\ti{W}_i$ by projection.  Since the $q$-valuation of the critical
  point $y(q)$ is the minimum of the $q$-valuations of the elements
  $y_i(q)$, the element $y(q)$ has a limit as $q \to 0$ iff the
  elements $y_i(q)$ also have a limit as $q \to 0$.
\end{proof}

\begin{lemma} \label{stacky2} For generic $[\omega] \in H^2_G(X,\Q)$, every critical
  point $y \in \Crit(W_t)$ for $t \in \R$ is stage-wise
  non-degenerate.
\end{lemma}

\begin{proof} The proof is an application of Proposition \ref{generic}
  and Lemma \ref{stacky}.  For any $\eps \in \R$ let $I(\mu,\eps)$
  denote the set of indices of facets at distance $\eps$ from $\mu$:
  \[ I(\mu,\eps) := \{ i \in \{ 1, \ldots, k \} \ | \ \lan \mu , \nu_i
  \ran + \omega_i = \eps \} .\]
  The dimension count in Proposition \ref{generic} shows that for
  generic $\omega_i$
\[ \# I(\mu, \eps) \leq 1 + \dim( \on{span}(\nu_i), i \in I(\mu,\eps)
)  .\]
Let $\eps = \eps_1$ be the minimum value for which the vectors
$\nu_i, i \in I(\mu,\eps)$ is linearly dependent.  The sum of the
terms $\sum_{i \in I(\mu,\eps)} y^{\nu_i}$ has non-degenerate critical
locus by Lemma \ref{stacky}.  Taking the quotient of $\t^\dual$ by the
span of $\nu_i, i \in I(\mu,\eps)$ and repeating the computation for
the remaining stages implies stage-wise non-degeneracy.
\end{proof}

\begin{lemma} \label{cross} {\rm (Wall-crossing for dimensions)} Let
  $X,G, \omega_t$ be as above.  In the case that $X \qu_t G$ undergoes
  a flip with center $Z$ at a singular value $t \in (0,\infty)$, with
  $t_\pm = t \pm \eps$ for $\eps$ small and $W_Z$ the Landau-Ginzburg
  potential of $Z$, we have
\begin{eqnarray*} 
\label{itema}
  \dim (QH( X \qu_{t_+} G)) - \dim (QH (X \qu_{t_-} G)) &=& \dim(\Sigma_t)!
  \Vol(\Sigma_t) \dim(QH(Z)); \\ 
\label{itemb}
  \dim(\hJac( W_{X,G,t_+})) - \dim(\hJac( W_{X,G,t_-})) &=& \dim(\Sigma_t)!
  \Vol(\Sigma_t) \dim(\Jac_+(W_Z)). \end{eqnarray*}
\end{lemma}

\begin{figure}[ht]
\begin{picture}(0,0)%
\includegraphics{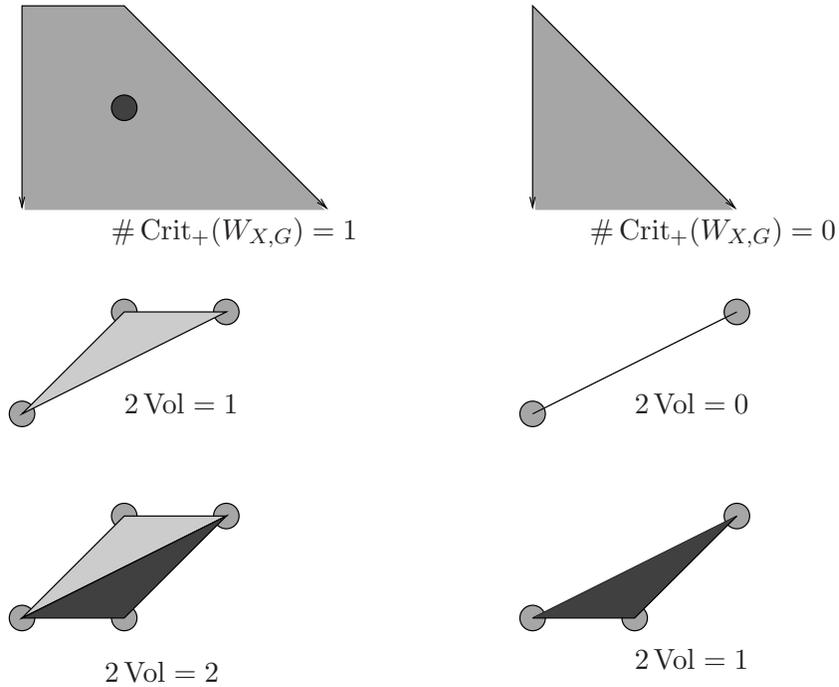}%
\end{picture}%
\setlength{\unitlength}{3947sp}%
\begingroup\makeatletter\ifx\SetFigFont\undefined%
\gdef\SetFigFont#1#2#3#4#5{%
  \reset@font\fontsize{#1}{#2pt}%
  \fontfamily{#3}\fontseries{#4}\fontshape{#5}%
  \selectfont}%
\fi\endgroup%
\begin{picture}(4666,4276)(2843,-5825)
\put(3452,-5810){\makebox(0,0)[lb]{{{{$2 \Vol = 2 $}%
}}}}
\put(6779,-5730){\makebox(0,0)[lb]{{{{$2 \Vol = 1 $}%
}}}}
\put(3573,-4126){\makebox(0,0)[lb]{{{{$2 \Vol = 1 $}%
}}}}
\put(6779,-4126){\makebox(0,0)[lb]{{{{$2 \Vol = 0 $}%
}}}}
\put(3492,-3084){\makebox(0,0)[lb]{{{{$\# \Crit_+(W_{X,G}) = 1$}%
}}}}
\put(6499,-3084){\makebox(0,0)[lb]{{{{$\# \Crit_+(W_{X,G}) = 0$ }%
}}}}
\end{picture}%

\caption{Wall-crossing for dimensions} 
\label{eulerwall}
\end{figure}

\begin{proof} 
  The first equality in \ref{cross} concerns the change in dimension
  of the quantum cohomology. Denote by $(\partial \Sigma_t)_\pm$ the
  union of facets of $\Sigma_t$ defined by half-spaces that do not
  resp. do contain $0$.  The corresponding partition of facets
  determines a partition $\{ I_+,I_- \}$ of $\{ 1 ,\ldots, l+1 \}$.
  The morphism of coarse moduli spaces from $X \qu_{t_\pm} G$ to
  $X \qu_t G$ is a stacky-weighted blow-down of the orbit that is the
  intersection of divisors corresponding to $I_\pm$ onto the center
  $Z$.  Consider the polytope $\Sigma_{t,\pm}$ given as the convex
  hull of $\{ \nu_j, j \in I_\pm \}$ and $0$.  The volume of the
  polytope $\Sigma_{t,\pm}$ is the dimension of the cohomology of a
  fiber of the inertia stack $I_{E_\pm} \to Z$ by Lemma
  \ref{maxcones}, and $QH(E_\pm) = H(I_{E_\pm})$ by definition.
  Furthermore, $E_\pm$ fibers over $Z$ with fiber a toric stack
  $E_{\pm,z}$.  The fan of $E_{\pm,z}$ is the union of cones generated
  by non-zero vertices in $\Sigma_{t,\pm}$.  Hence
 \begin{eqnarray*}   \dim(QH(E_\pm)) &=& \dim(QH(E_{\pm,z})) \dim(QH(Z)) \\
&=&  \dim(\Sigma_{t,\pm})! \Vol(\Sigma_{t,\pm}) \dim(QH(Z)) .\end{eqnarray*}
Since $X \qu_{t_\pm} G$ are isomorphic away from the exceptional loci, 
\begin{eqnarray*}
  \frac{ \dim QH(X \qu_{t_+} G) - \dim 
  QH(X \qu_{t_-} G)}{\dim(QH(Z))}  &=& \dim(\Sigma_t)! \Vol(\Sigma_{t,+})
                                       -  \dim(\Sigma_t)! \Vol( \Sigma_{t,-})
  \\
                                   &=& \dim(\Sigma_t)! \Vol(\Sigma_t) \end{eqnarray*} 
                                 which proves the claim. 

                                 The second equality in \ref{cross}
                                 describes the change in the dimension
                                 of the Jacobian ring.  We may suppose
                                 that $t_\pm$ are sufficiently close
                                 to the critical value $t_j$ so that
                                 there exists a number $c > 0$ such
                                 that $\Psi( \Crit_+( W_{X,G,t}))$
                                 consists of a single value in
                                 $(-c,c)$, which crosses the boundary
                                 of $\Delta_{X \qu_t G}$ as $t$
                                 crosses $t_j$, and the other
                                 components of
                                 $\Psi(\Crit_+(W_{X,G,t}))$ stay
                                 outside of $(-c,c)$ for all
                                 $t \in (t_-,t_+)$.  The critical
                                 value that crosses the boundary
                                 corresponds to the intersection of
                                 $\dim(\Sigma_t) + 1$-hyperplanes
                                 varying linearly in $t$.  By
                                 Kouchnirenko's theorem \ref{kouch}
                                 and Proposition \ref{stagewise} the
                                 number of the critical points
                                 $y \in \Crit(W_{X,G,t})$ mapping to
                                 the singular set in
                                 $\Delta_{X \qu_t G}$ is
                                 $\dim(\Sigma_t)!  \Vol(\Sigma_t)
                                 \dim(\Jac_+(W_Z))$.
                               \label{infI}
                                 For any interval $I \subset \R$, let
                                 $\Crit_I(W_{X,G})$ denote the subset
                                 of the critical locus
                                 $\Crit(W_{X,G})$ consisting of points
                                 $y$ with
                                 $\inf_{j = 1,\ldots, k} \lan \Psi(y),
                                 \nu_j - \omega_j \ran \in I$.
                                 By the previous paragraph and
                                 Propositions \ref{kouch} and
                                 \ref{stagewise} the difference in the
                                 number of critical points of the
                                 potential before and after the
                                 critical value is equal to
\begin{eqnarray*}
\dim \hJac( W_{X,G,t_+}) & -&  \dim \hJac( W_{X,G,t_-}) 
=
| \Crit_+( W_{X,G,t_+})| - | \Crit_+( W_{X,G,t_-})| \\
&=& 
| \Crit_{(c,\infty)}( W_{X,G,t_+})| - | \Crit_{(c,\infty)}( W_{X,G,t_-})| \\
&& + 
| \Crit_{(0,c)}( W_{X,G,t_+})| - | \Crit_{(0,c)}( W_{X,G,t_-})| \\
&=& 0 + \dim(\Sigma_t)! \Vol(\Sigma_t)  \dim (\Jac(W_Z)) 
 \end{eqnarray*} 
as claimed. \label{crossend}
\end{proof} 

\begin{example} In Figure \ref{eulerwall} we show the flipping simplex
  for a blow-up of $\C^2$ at $0$.  More precisely the top figures show
  the polytope of a blow-up of $\C^2$ and of $\C^2$ respectively; the
  image of the critical point under the tropical moment map $\Psi$ is
  shown as a point in the interior on the upper left.  The middle
  figures show the polytopes $\Sigma_{\pm}$, spanned by
  $(-1,0), (0,1), (1,1)$ resp. $(-1,0), (1,1)$ and their volumes $1/2$
  resp. $0$; the last figures show the Newton polytopes
  $\Delta_{X \qu_{t \pm \eps} G}$ of the potentials before and after
    the blow-down respectively, with the polytopes $\Sigma_-$ shown as
    sub-polytope.
\end{example}

\begin{remark} \label{term}  
\begin{enumerate} 
\item {\rm (Termination of a toric minimal model program)} In
  particular, one sees from the above wall-crossing formula that
  $\dim(QH(X \qu_t G))$ decreases at each wall-crossing.  This is one
  of the proofs of the eventual termination of the toric minimal model
  program, discussed in \cite{cox:toric}, \cite{mats:mori}, where
  $\dim(QH(X \qu_t G))$ is described in combinatorial terms.
\item {\rm (Dependence of the Jacobian ring on the symplectic class)}
  The location of the critical values $\Psi(y), y \in \Crit(W)$ varies
  with the choice of symplectic class $\omega$, and at certain affine
  linear hyperplanes (occurring when more than $n+2$ normal vectors
  have a common value) the critical values $\Psi(y)$ can ``collide''
  as the polarization class $\omega$ varies.  
\end{enumerate} 
\end{remark}

\begin{lemma} \label{monotone} {\rm (Equality of dimensions in the
    Fano case)} Suppose that the family $X \qu_t G$ as above has a
  unique singular point at $t_0$, and undergoes a Mori fibration {\em
    over a point} at $t_0$.  Then
  $\dim \hJac(W_{X,G}) = \dim QH(X \qu_t G)$.
\end{lemma} 

\begin{proof} In the absence of spurious facets, the Fano case reduces
  to Kouchnirenko's Theorem \ref{kouch}.  The assumption that
  $X \qu_t G$ undergoes a fibration over a point means that every
  facet of $\Delta_{X \qu_t G}$ is equidistant from some point
  $\mu \in \t_\R^\dual$, so that $X \qu_t G$ is Fano.  Without
  loss of generality we may assume that $\mu = 0$.  Suppose that the
  presentation of $X$ as a symplectic quotient is the minimal one,
  that is, each weight of $X$ corresponds to a facet of
  $\Delta_{X \qu_t G}$ (no spurious facets).  In this case all
  $y \in \Crit_+(W_{X,G})$ have $\Psi(y) = 0$.  Thus we may omit the
  parameters $q$ from the definition of the potential and the number
  of critical points is equal to $\dim QH(X \qu_t G)$ by
  Kouchnirenko's Theorem \ref{kouch}.

  In the case that the presentation is not minimal, the equality
  follows from the formal implicit function theorem.  Let
  $W_{X \qu G}: \ti{T}^\dual(\Lambda) \to \Lambda$ denote the Givental
  potential associated to the minimal presentation.  That is, if
  \[ \cT = \{ i | \codim( \{ \lan \nu_i, \mu \ran = - \omega_i \} \cap
  \Delta_{X \qu G} ) = 1 \} \subset \{ 1, \ldots, k \} \]
denotes the indices of the inequalities defining facets of $\Delta_{X
  \qu G}$ then
\[ W_{X \qu G}(\ti{g}) = \sum_{i \in \cT} q^{\omega_i} y_i \]
and each term corresponds to a facet of $\Delta_{X \qu G}$. Let
$W_{X,G,\fake} = W_{X,G}  - W_{X \qu G} $
denote the terms arising from the ``fake facets'', that is, weights of
$X$ that do not define facets of $X \qu G$ so that
\[ W_{X,G} = W_{X \qu G} + W_{X,G,\fake} .\]
Let $c_1^{G}(X)_{\true} \in \g^\dual_\Q \cong H^2_G(X)$ be the sum of
the weights corresponding to divisors of $X \qu G$, that is, the true
facets.  By the genericity assumption the critical locus
$\Crit(W_{X \qu G})$ is non-degenerate.  The formal criterion for
smoothness (that is, the formal implicit function theorem as used in
the proof of Proposition \ref{stagewise}) implies that there is an
isomorphism $ \Crit_+(W_{X \qu G}) \to \Crit_+(W_{X,G}) $.  That is,
adding in the higher order terms give a deformation of the critical
locus $\Crit(W_{X \qu G})$ to $\Crit(W_{X,G})$ lying over the interior
of the moment polytope $\Delta_{X \qu G}$.
\end{proof} 

By combining the Fano dimensional equality \ref{monotone} and the
wall-crossing and fibration formulas \ref{cross} and \ref{fibration}
we have the equality of dimension in general:

\begin{theorem} {\rm (Equality of Dimensions)} \label{equality} 
 For $X \qu G$ as in the statement of Theorem \ref{main}, $\dim
 \hJac(W_{X,G}) = \dim QH(X \qu G)$.
\end{theorem} 

\begin{proof} By induction we may assume that the dimensional equality
  of Theorem \ref{equality} holds for toric stacks of dimension
  smaller than $\dim(X \qu G)$.  By Lemma \ref{explicit} as $t$ varies
  the toric orbifold $X \qu_t G$ undergoes a finite sequence of flips
  or weighted blow-downs
  $X \qu_{t_j - \eps} G \dashrightarrow X \qu_{t_j + \eps} G$,
  followed by a fibration to a toric stack of $(X \qu_{t_n} G)'$ of
  smaller dimension with Fano fibers $(X \qu_{t_n} G)''$.  The
  wall-crossing terms are the same, by Lemma \ref{cross}.  In the case
  of a fibration, the equality follows from Lemma \ref{fibration}.
\end{proof} 

\begin{proof}[Proof of Theorems \ref{main} and \ref{main0}]   
By \cite{qkirwan1,qkirwan2,qkirwan3}, the linearization $D_\alpha \kappa_X^G$ of the
quantum Kirwan map descends to a map
\[T_\omega \kappa_X^G/QSR_{X,G} :QH_G(X)/ QSR_{X,G} \to QH( X\qu G)
\]
is an isomorphism.  By Theorem \ref{weaktoric2},
$T_\alpha \kappa_X^G/QSR_{X,G}$ is surjective.  By Theorem
\ref{equality}, the induced map from $\hJac(W_{X,G})$ to
$T_{\kappa_G(\omega)}QH(X \qu G)$.  Theorem \ref{main0} follows from
the identification with the Jacobian ring in Proposition \ref{ident}.
\end{proof} 

\begin{corollary}\label{semisimple} The quantum cohomology $QH(X \qu G)$ of any proper toric orbifold
  with projective coarse moduli space $X \qu G$ over the universal
  Novikov field $\Lambda$ is semisimple at bulk deformation
  $\kappa_X^G(0)$ for generic symplectic classes
  $ \omega \in H^2_G(X,\Q)$.
\end{corollary}

\begin{proof} As explained in \cite[Proposition 4.9]{iri:integral},
  semisimplicity at the bulk deformation $\kappa_X^G(0)$ follows from
  the identification with the Batyrev ring
  $\widehat{QH}_G(X)/ \widehat{QSR}_X^G$, or rather, the Jacobian ring
  $\Jac_+(W_{X,G})$ and the fact that for generic $\omega$, the
  potential $W_{X,G}$ has only stagewise non-degenerate critical
  values $y \in \Crit(W)$, see Proposition \ref{stagewise} and also Iritani
  \cite[Proposition 3.10]{iri:integral}.  Note that semisimplicity for
  generic values of $q$ also follows from Lemma \ref{nondeg}.
\end{proof}

\begin{remark} \label{nss} 
\begin{enumerate} 
\item {\rm (Non-semisimple cases)} An example of a non-generic
  symplectic structure with non semi-simple quantum cohomology
ring 
  $QH(X \qu G)_{q =1}$ is given in  Ostrover-Tyomkin \cite{os:qh}.
\item {\rm (Dubrovin conjecture for toric orbifolds)} Semisimplicity
  is related by a conjecture of Dubrovin, see \cite{bay:ss}, to the
  existence of a full exceptional collection in the bounded derived
  category of coherent sheaves $D^b \on{Coh}(Y)$ of $Y$. In the toric
  case the existence of such a collection is proved by Kawamata
  \cite{kaw:der}.
\item \label{c1}  {\rm (Equivariant first Chern class maps to the potential)}
  Under the isomorphism from $QH_G(X)/QSR_{X,G} $ to $\Jac(W_{X,G})$,
  the coset of the first Chern class
  $[c_1^G(X)] \in QH_G(X)/QSR_{X,G}$ maps to the potential
  $W_{X,G} \in \Jac_+(W_{X,G})$ itself, by definition of the
  isomorphism.  However, $c_1^G(X)$ does not map to
  $c_1(X \qu G) \in H^2(X \qu G) \subset QH(X \qu G)$ in general.
  \label{halfpoint6} Consider the example of $G = \C^\times$ acting on
  $X = \C$ with weight two, so that
  $X \qu G = \C \qu \C^\times = \P(2) = B\Z_2$ is the stacky
  half-point.  In this case $c_1^G(X)$ is a degree two class and maps
  under $D_\omega \kappa_X^G$ to the twisted sector in $X \qu G$,
  since the contributing maps in $\M^G_{1,1}(\bA,X)$ have degree one.
  On the other hand, $c_1(X \qu G)$ is trivial since the tangent
  bundle is rank zero.
\end{enumerate}
\end{remark} 

As a corollary to the second part of the Remark and the discussion
above we have the following, which ``quantifies'' the sense in which
flips in the minimal model program ``make the variety more Fano''.

\begin{corollary} {\rm (Decrease in the eigenvalues of $c_1 \star$
    under mmp flips)} Suppose that $X \qu_t G$ undergoes a flip at
  $t = t_j$.  Then the minimal $q$-valuation
  $ \min( \on{val}_q(\lambda_i))$ of the eigenvalues
  $\lambda_i \in \Lambda$ of quantum multiplication by
  $D_\omega \kappa^G_X (c_1^G(TX))$ on
  $T_{\kappa_X^G(\omega)} QH(X \qu G)$ increases.
\end{corollary}

\begin{proof}   
  The critical values of $W_{X,G}$ moving outside the moment polytope
  $\Delta_{X \qu G}$ at the time $t_j$ of the flip are those values
  $W_{X,G}(y) , y \in \Crit(W_{X,G})$ with lowest $q$ valuation for
  $t$ slightly smaller than $t_j$.  On the other hand, by Remark
  \ref{nss} \eqref{c1} the $q$-valuations of such $y$ are the lowest
  $q$-valuations of eigenvalues of $D_\omega \kappa^G_X c_1^G(TX)$
\end{proof}

\begin{corollary} {\em (Equivariant version of the Batyrev presentation)} 
  \label{maineq}  
There is a canonical isomorphism
$T_\omega \widehat{QH}_{\ti{G}}(X)/
\widehat{QSR}_X^{G,\ti{G}}(\omega) \to T_{\kappa_X^{\ti{G},G}(\omega)}
QH_{\ti{G}/G}(X \qu G)$
for any rational symplectic class $\omega \in H_{\ti{G}}^2(X)$.
\end{corollary}  

\begin{proof}  We have already shown in Theorem \ref{main0}
the non-equivariant version of the statement in Corollary
\ref{maineq}, that is, setting the equivariant parameters for $T =
\ti{G}/G$ to zero.  By equivariant formality, $QH_{\ti{G}/G}(X \qu G)$ is a free $QH_{\ti{G}/G}(\pt)$-module, and this implies that
$\widehat{QH}_{\ti{G}/G}(X \qu G)$ is a free
$\widehat{QH}_{\ti{G}/G}(\pt)$ module.  Since the same is true for
the left-hand-side, it follows that the linearization of the
equivariant quantum Kirwan map map $T_\omega \widehat{QH}_{\ti{G}}(X)/
\widehat{QSR}_X^{G,\ti{G}} \to T_{\kappa_X^{\ti{G},G}(\omega)}
QH_{\ti{G}/G}(X \qu G)$ is also an isomorphism.
\end{proof} 

\subsection{ Invariance of quantum cohomology under weighted toric
  flops}
 \label{flops} 

 In this section we digress to show that quantum cohomology is
 invariant under weighted toric flops.  Let $X \qu_t G$ be a family of
 toric quotients as above (not necessarily in the direction of the
 canonical class $c_1(X \qu_t G)$) so that $X \qu_t G$ is a locally
 free quotient for $t$ generic.

 \begin{definition} The variation of git quotient $X \qu_t G$
   undergoes a {\em flop} at a singular time $t_i$ if there is a
   unique point $x \in X$ semistable for $t_i$ with
   positive-dimensional stabilizer $G_x$, the group $G_x$ is
   one-dimensional and the sum $\sum_{i=1}^l \nu_i$ of the weights
   $\nu_i \in \Z, i =1,\ldots, l$ for $G_x$ on $T_x X$ is zero.
\end{definition}

\begin{lemma} If $X \qu_t G$ undergoes a flop at $t_i$ then the
  quotients $X \qu_{t_i \pm \eps} G$ on either side of the critical
  value $t_i$ are $K$-equivalent in the sense that the canonical
  bundles are pull-backs of the same (rational) bundle under the
  morphism of coarse moduli spaces
  $X \qu_{t_i \pm \eps} G \to X \qu_{t_i} G$.
\end{lemma}

\begin{proof} By Kempf's descent criterion, see \cite[2.3]{drezet},
  the canonical bundle $K_X$ descends to the singular quotient
  $X \qu_{t_i} G$ and similarly $K_X$ descends to the canonical bundle
  $K_{X \qu_{t_i \pm \eps} G}$ on $X \qu_{t_i \pm \eps} G$; where the
  quotients are stacks these bundles exist rationally on the
  corresponding moduli spaces.  The maps
  $X \qu_{t_i \pm \eps} G \to X \qu_{t_i} G$ are induced by inclusions
  of semistable loci and the claim follows.
\end{proof}

\begin{proposition} \label{analytic} If $X \qu_t G$ undergoes a flop
  at $t = t_i$ then the quantum cohomologies
  $QH(X \qu_{t_i \pm \eps} G)$ are isomorphic as vector spaces, and
  the quantum products $\star_t$ are related by analytic continuation.
\end{proposition} 

\begin{proof} It suffices to show, by the description of the quantum
  cohomology in Theorem \ref{main} that no critical values cross the
  boundary of the moment polytope at the critical time. We suppose
  that the facets
  $F_{j,t_i \pm \eps} \subset \Delta(X \qu_{t_i \pm \eps} G)$ from
  Definition \eqref{flip} meeting the singular moment value (which we
  may assume maps to $0$) at time $t_i$ (which we may assume equals
  $0$) are numbered $1,\ldots, l+1$.  By the genericity assumption in
  Proposition \ref{generic} and Lemma \ref{stacky2}, the critical
  values $y_t \in \Crit(W_t)$ for each $t$ mapping to $0$ at $t = 0$
  are stagewise non-degenerate.  By the analysis in Lemma
  \ref{stagewise}, each such $y_t$ maps under $\Psi$ to a point
  $\mu t \in \t^\dual$ such that
  $ \lan \mu t, \nu_i \ran - c_i t, i = 1,\dots, l+1 $ are independent
  of $i$.  By local triviality of the canonical bundle, the divisor
  $\sum_{i=1}^{l+1} c_i D_i$ with coefficients $c_i$ is locally
  linearly equivalent to the divisor $\sum_{i=1}^{l+1} (c_i - c) D_i$
  with coefficients $c_i - c$, for any constant $c$.  Thus for any
  $c \in \R$ there exists a $\mu' \in \t^\dual_\R$ such that
  $ \lan \mu' t, \nu_i \ran = ( c_i - c) t, i = 1,\dots, n+1 $ is
  independent of $i$.  Taking e.g. $c_1 = c$ we obtain that
  $\lan \mu',t \nu_i \ran =0$ for all $i = 1,\ldots, l+1$ vanish for
  all $i$, and the inequalities
  $\lan \mu' t,\nu_i \ran \ge (c_i - c) t$ are satisfied with equality
  for all $t$.  Thus the family $\mu' t$ does not cross the boundary
  of the moment polytope.  It follows that $\dim \hJac(W_{X \qu_t G})$
  is independent of $t$ in a neighborhood of $t_i$, so that
  $\hJac(W_{X \qu_{t_i \pm \eps} G})$ are isomorphic as vector spaces.
  The products $\star_t$ are related by analytic continuation by
  analytic dependence of the Jacobian ideal
  $\lan \partial_\lambda W_{X,G}(y e^\lambda) _{\lambda = 0}\ran$ in
  Definition \ref{jac} on the symplectic class.
\end{proof}

\section{Minimal models and non-displaceable Lagrangians} 
\label{discuss} 

This section is a discussion of how the results here combine with
those of \cite{wo:gdisk}, \cite{ww:qu} on non-displaceable Lagrangian
tori.  In particular we explain that toric orbifolds can have
infinitely many non-displaceable tori because they can have infinitely
many runnings of the toric minimal model program.  Recall the
following from Fukaya-Oh-Ohta-Ono \cite{fooo:toric1}, \cite[Theorem
3.17,Corollary 4.6]{fooo:toric2} and Woodward \cite{wo:gdisk}:

\begin{theorem} \label{wodisk} {\rm (Non-displaceable toric moment
    fibers via critical points of the Givental potential)} Let $X$ be
  a finite dimensional vector space with a linear action of a torus
  $G$ and polarization so that the git quotient $X \qu G$ is a proper
  toric Deligne-Mumford stack with projective coarse moduli space,
  moment map $\Phi: X \qu G \to \t_\R^\dual$ and moment polytope
  $\Delta_{X \qu G} = \Phi(X \qu G)$.  Let $W_{X,G}$ denote the Givental  potential and $\Psi: \Crit_+(W_{X,G}) \to \Delta_{X \qu G}$ the
  tropical moment map. Then for any $y \in \Crit_+(W_{X,G})$, the
  inverse image $\Phinv(\Psi(y)) \subset X \qu G$ is a Hamiltonian
  non-displaceable Lagrangian torus.
\end{theorem}  

The proof in \cite{wo:gdisk} uses that non-displaceability of a
Lagrangian in $X \qu G$ is implied by the $G$-non-displaceability of
its pre-image in $X$, and this non-displaceability is governed by a
suitable $G$-equivariant version of Floer homology.

\subsection{Generic tmmp runnings} 

By the results of the previous sections, we may understand the
critical values of the potential in terms of the corresponding minimal
model program.  Let $Y$ be a smooth proper toric Deligne-Mumford stack
with polarized projective coarse moduli space, and $G$ a torus acting
on a vector space $X$ so that $X \qu G$, equipped with its residual
torus action, is isomorphic to $Y$.

\begin{notation} \label{tmmp}
\begin{enumerate} 
\item {\rm (Toric minimal model program)} The sequence of stacks
\[  X \qu_t G, \ t \in [0,\infty), \quad Y = X \qu G = X \qu_0 G \]  
obtained by varying the equivariant symplectic class $\omega$ in the
direction of $- c_1^G(X)$ will be called a {\em toric minimal model
  program} (tmmp) running for $Y$.  Our terminology differs from the
standard terminology in that we include the path $\omega_t$ of
symplectic classes in the definition of the running.
\item {\rm (Transition times)} The values $t_1,\ldots,t_n$ of $t$ for
  which $X \qu_t G$ is singular (that is, there exist points $x$ in
  the stable locus $X^{\ss}$ with infinite stabilizer subgroups $G_x$
  ) are the {\em transition times} for the tmmp running.
\item {\rm (Dimension jumps)} Let $t_{j,\pm} = t_j \pm \eps$ for
  $\eps$ sufficiently small so that $t_{j-1} + \eps < t_j - \eps, j =
  2,\ldots, n$.  Let
\[ d_j = \dim QH (X \qu_{t_{j,-}} G) - \dim QH(X \qu_{t_{j,+}} G) \]
denote the {\em dimension jump} at $t_j$.  
\item {\rm (Singular moment values)} For simplicity, we assume that
  $X \qu_{t_j} G$ has a connected {\em singular set} with infinite
  stabilizer subgroups $(X \qu_{t_j} G)^{\sing}$ mapping to singular
  moment value
  \[ \Phi( (X \qu_{t_j} G)^{\sing}) \subset \Delta_{ X \qu_{t_j} G }
  \subset \Delta_{ X \qu_0 G } , \quad j =1,\ldots, n.\]
\item {\rm (Fiber of the Fano fibration)} Suppose furthermore that for
  $t$ just before $t_n$, the quotient $X \qu_t G$ is a fibration over
  $Y'' = (X \qu_{t_n} G)''$, with Fano fiber $Y' = (X \qu_{t_n}
  G)'$.
\end{enumerate} 
\end{notation}  

\begin{remark}  \label{induced}
\begin{enumerate}
\item {\rm (Non-uniqueness of tmmp runnings)} Many presentations of
  $X \qu G$ as a git quotient will give the same tmmp runnings.
  However, toric orbifolds can have infinitely many tmmp runnings,
  corresponding to different realizations of $X \qu G$ as git
  quotients.  (Recall we take the family of symplectic class
  $\omega_t$ as part of the definition of the tmmp running.)  The
  ``fake facet equalities'' $ \lan \cdot , \mu_j \ran = \omega_{t,j}$
  (those with empty solution set in $\Delta_{X \qu G})$ can ``catch
  up'' to the ``true facets'' (those with non-empty solution set)
  under the deformation $\omega_t$ at arbitrary times.  For example,
  taking the minimal presentation of $\P(1,3,5)$ as a git quotient
  $\C^3 \qu \C^\times$ yields a trivial toric mmp, but introducing a
  presentation as a quotient $\C^4 \qu (\C^\times)^2$ yields a toric
  mmp with a flip to an ``orbifold Hirzebruch surface'', which is
  similar to the example discussed in \cite{ww:qu}.  See Figure
  \ref{p135mmp}.  Since in this case the time of the transition
  depends on the position of the extra spurious facet (shown as the
  right-most dotted line in Figure \ref{p135mmp}) this give an example
  with infinitely many tmmp runnings.  The computation
  Abreu-Borman-McDuff \cite[Proposition 4.1.4]{abreu:ext} show that in
  the manifold case there is a unique tmmp, since the fake facets
  never ``catch up''.
\item {\rm (Induced tmmp running for fibrations)}
  \label{induceditem}
Any presentation of $Y$ as a git quotient $X \qu G$ induces a
  presentation of the base $(X \qu_{t_n} G)''$ of the final fibration
  as a git quotient, corresponding to the inequalities that are not
  defining inequalities for the Fano fiber $(X \qu_{t_n} G)''$,
  that is, the inequalities which become strict equalities for the
  final polytope $\Delta_{X \qu_{t_n} G}$.  Hence, any presentation of
  $Y$ induces a tmmp running for the base of the final fibration.
\end{enumerate} \end{remark} 

\begin{figure}
\includegraphics[width=4in]{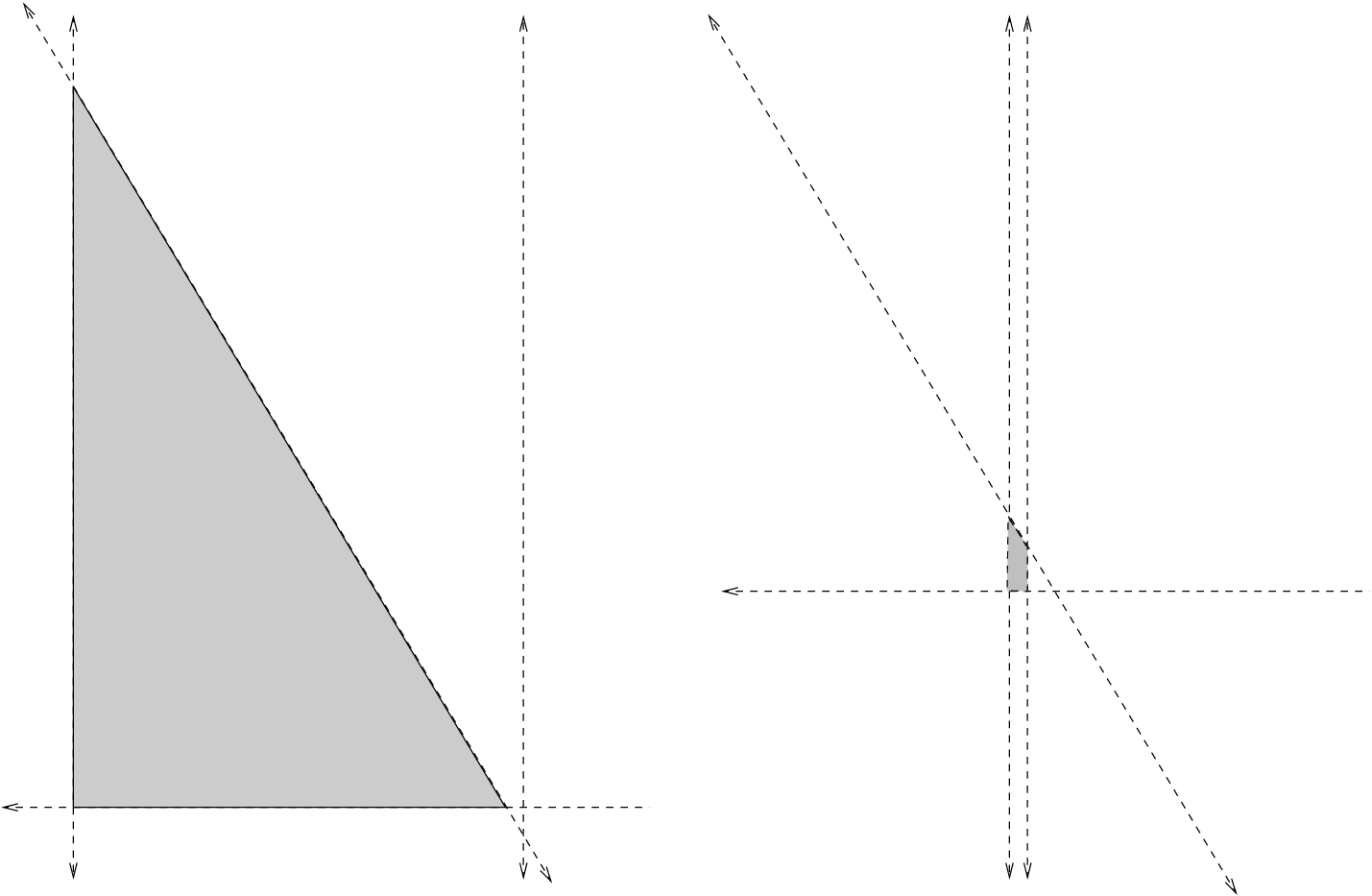}
\caption{Non-trivial toric mmp for $\P(1,3,5)$} 
\label{p135mmp}
\end{figure} 

\begin{notation} \label{espace} {\rm (Eigenspace decomposition for
    quantum multiplication by $c_1$)} Given a presentation
  $\PP = (X,G)$ of a toric stack $X \qu G$ as a git quotient let
  $\omega_\PP = \kappa_X^G(\omega)$ and
\[c_1(Y,\PP) := D_\omega \kappa_X^G c_1^G(X) \in T_{\omega_\PP}
QH(Y).\]
 Consider the decomposition of quantum cohomology into
eigenspace sums for quantum multiplication by $c_1(Y,\PP)$:
\[ T_{\omega_\PP} QH(Y ) \cong \bigoplus_{j=1}^n T_{\omega_\PP}
QH(Y)_{\PP,j} .\]
where each $QH(Y)_{\PP,j}$ is a sum of eigenspaces for eigenvalue
$\lambda$ with the same $q$-valuation $\on{val}_q(\lambda)$.
\end{notation} 

\begin{theorem} \label{jump} 
{\rm (Relationships between tmmps and
    quantum cohomology)} Let $Y$ be a compact toric orbifold and $\PP$
  be a generic toric mmp for $Y$ obtained from a quotient presentation
  $Y = X \qu G$ with transition times $t_j$, dimension jumps $d_j$ and
  singular moment values $\psi_j \in \Psi(\Crit(W_{X,G}))$,
  $j = 1,\ldots, n$.
\begin{enumerate} 
\item \label{aa} The transition times $t_j, j = 1,\ldots, n$ of the tmmp running are the
  $q$-valuations of the eigenvalues $\lambda_j$ of quantum
  multiplication by $c_1(Y,\PP)$.
\item  \label{bb} The dimension jumps $d_j, j = 1,\ldots, n$ of the tmmp running at time
  $t_j, j = 1,\ldots, n$ are the dimensions of the corresponding
  eigenspaces,
\[ \dim T_{\omega_\PP} QH(Y )_{\PP,j} = d_j ;\]
\item  \label{cc}  for each $j = 1,\ldots, n-1$, the inverse image $L_j$ of
  $\psi_j$ in $Y$ is Hamiltonian non-displaceable.  The number of
  local systems making the Lagrangian Floer homology of $L_j$
  non-vanishing, counted with multiplicity, is also equal to $d_j$;
\item \label{dd}  for $j = n$, the factor $T_{\omega_\PP} QH(Y)_{\PP,n}$ further
  splits
\[ T_{\omega_\PP} QH(Y)_{\PP,n} = \bigoplus_{j=1}^{n_1} T_{\omega_\PP}
  QH(Y)_{\PP,n,j} \]
  \[ \dim T_{\omega_\PP} QH(Y)_{\PP,n,j} = \dim(
  T_{\omega_{\PP'}} QH(Y')) d_{j,1} \]
according to a splitting induced from a tmmp running $\PP_1$ for the base $Y''
= (X \qu_{t_n} G)''$ and fiber $Y' = (X \qu_{t_n} G)'$ with dimension
jumps $d_{j,1}$ as in Remark \ref{induced} \eqref{induceditem}, and
for each singular value $\psi_{j,1}$ in such a tmmp running the inverse image
in $Y$ is Hamiltonian non-displaceable.
\end{enumerate} 
\end{theorem} 

\begin{proof} 
  Suppose that $\PP$ is a minimal model program corresponding to a
  presentation of $Y$ as a quotient of $X$ by $G$.  By the main result
  Theorem \ref{main} $QH(Y) \cong \hJac(W_{X,G})$, and the latter
  admits a decomposition into components corresponding to critical
  points with fixed value of the tropical moment map from Definition
  \ref{tropmom}.  Quantum multiplication by
  $c_1(Y,\PP) = D_\omega \kappa_{X,G} ( c_1^G(X) )$ is given by
  multiplication by $W_{X,G}$ itself, hence \eqref{aa} and \eqref{bb}.
  By Lemma \ref{cross}, each summand has dimension that of the
  dimension jump in the given tmmp running.  \eqref{cc} is Theorem
  \ref{wodisk}, with the multiplicity computed using Kouchnirenko's
  theorem.  \eqref{dd} is a consequence of Lemma \ref{fibration}
and Theorem \ref{wodisk}.
\end{proof} 

\subsection{Non-generic tmmp runnings} 

The results above are for generic initial symplectic class only.
Abreu has pointed out to us that there is still a connection between
non-displaceable Lagrangians and minimal model programs, even in the
case that minimal model program involves flips over ``singular'' toric
orbifolds, in the sense that the critical points of the
Landau-Ginzburg potential for the singular toric manifolds ``cause''
non-displaceable moment fibers in the original manifold or orbifold.
More precisely, we suppose that the generic stabilizer is trivial and
we are in the following situation:

\begin{notation}   \label{singnot}
\begin{enumerate} 
\item {\rm (Singular base of a tmmp transition)} Let $Y$ be a compact
  toric orbifold with symplectic class $\omega_Y$.  Consider a toric
  mmp for $Y$ with dimension jumps $d_j$ and singular moment values
  $\psi_j$, $j = 1,\ldots, n$, and the flip/contraction at time $t_j$
  has base a possibly singular toric variety $Z_j$ with polytope
  $\Delta_j$.  Let $\t_j^\dual$ denote the span of $\Delta_j$ and
  $T_j^\dual \subset T^\dual$ the torus with Lie algebra $\t_j^\dual$.
\item {\rm (Normal part of the potential)} Let
\[ W_{X,G,j}: T^\dual(\Lambda_0) \to \Lambda, \quad y \mapsto
\sum_{ \d \nu_j |_{\Delta_j} = 0}  q^{\omega_j} y^{\nu_j} \]
denote the part of the potential $W_{X,G}$ corresponding to the normal
vectors constant on $\Delta_j$.  Thus 
\[ W_{X,G} = W_{X,G,j} + W_{X,G,j}'  \] 
where $W_{X,G,j}'$ is the sum of terms corresponding to vector $\nu_j$
that are {\em non}-constant on $\Delta_j$.
\item {\rm (Normally non-degenerate)} We say $W_{X,G}$ is {\em
    normally non-degenerate} at $\Delta_j$ if each critical point
  $y \in \Crit(W_{X,G,j})$ is non-degenerate.  An example is shown in
  Figure \ref{abreu1}.
\end{enumerate} 
\end{notation}

\begin{theorem} \label{crit} {\rm (Non-displaceable Lagrangians via
    non-generic tmmps)} Suppose that $Y, \PP$ are as above so that
  $W_{X,G}$ is normally non-degenerate at $\Delta_j$, with leading
  order terms of order $q^\alpha$.  Then each critical point
  $(y'',y')$ of $W_{X,G,j} \times W_{X,G,j}' $ with
  $ \on{val}_q(y'') > \alpha$ is equivalent, modulo terms vanishing on
  $T_j$, to a critical point $y$ of $W_{X,G}$, and so $\Psi_j(y)$
  defines a non-displaceable moment fiber in $Y$.
\end{theorem}  

\begin{proof} The proof is an order-by-order correction argument,
  using Proposition \ref{stagewise}, see also the implicit function
  theorem of Fukaya et al \cite[Theorem 10.4]{fooo:toric1}.  Note that
  the tropical moment map
  $\Psi_j: \Crit(W_{X,G,j}) \to \on{int}(\Delta_j)$ maps to $\Delta$
  via the inclusion $\Delta_j \to \Delta$.  Any lift of a critical
  point $y \in \Crit(W_{X,G,j}) \subset T_j^\dual$ has the property
  that $\d W_{X,G}(y)$ descends to $\t^\dual/\t_j^\dual$, since the
  partial derivatives in the direction of $\t_j^\dual$ vanish.  As in
  \ref{stagewise}, \cite[Theorem 10.4]{fooo:toric1}, the point
  $y'' y' \in T$ may be corrected by an element of $T^\dual/T_j^\dual$
  to a critical point $y$ of the full potential $W_{X,G}$.  It follows
  from \cite[Theorems 3.19,Corollary 4.6]{fooo:toric2} that these
  fibers have non-trivial Floer cohomology, and so are
  non-displaceable.
\end{proof}

\begin{figure}[ht]
\includegraphics[width=2.5in]{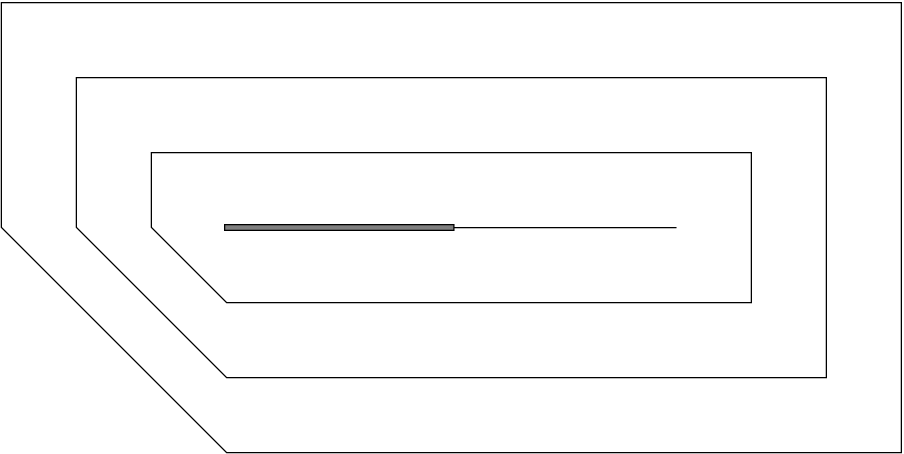}
\caption{A singular minimal model program} 
\label{abreu1}
\end{figure} 

\begin{example} \label{fam} The following example was pointed out to
  us by M. Abreu.  Suppose that $Y$ is the toric manifold whose
  polytope has vertices $(0,1),(0,2),(1,0), (4,0), (4,2)$, which is a
  blow-up of a product $\P^1 \times \P^1$; this was the first example,
  discovered in Fukaya-Oh-Ohta-Ono \cite{fooo:toric1}, of a toric
  manifold with a continuum of non-displaceable toric fibers.  The
  minimal model program for this symplectic class has no flips, and
  the final step is a ``singular fibration'' over a $\P^1$, with a
  single singular fiber consisting of a nodal $\P^1$.  We suppose that
  $X \qu G$ has a non-minimal presentation as a git quotient, so that
  there is a spurious facet of $\Delta_{X \qu G}$ with equation
  $\lambda_1 = -\eps$ for $\eps > 0$.  The potential
  $W_{X,G,1}' |_{\t_1^\dual}$ for this base, allowing bulk
  deformations at the divisors with normal vectors $(1,0)$ and
  $(1,1)$, is of the form
  $W_{X,G,1}' (y_1,y_2) = y_1 + y_1 y_2 + q^4/y_1 + q^{-\eps} y_1$
  while $W_{X,G,1}(y_2) = y_2 + q^2/y_2$.  Therefore
\[ W_{X,G,1}'( \cdot, -1) = q^4/y_1 + q^{-\eps} y_1 \] 
has critical points, as $\eps$ varies, given by $(y_1,-1)$ with
$\Psi(y_1,-1) = (-\infty,2)$.  Those with $\on{val}_q(y_1) > 1$ can be
corrected by Theorem \ref{crit} to honest critical points.  So the
toric moment fibers above the segment between $(1,1)$ and $(2,1)$ are
Hamiltonian non-displaceable. Non-displaceability of the end-points
$(1,1)$ and $(2,1)$ holds by continuity.  This reproduces the
non-displaceability result in this case from \cite{fooo:toric1}.  See
Figure \ref{abreu1}.
\end{example} 

\begin{remark} If $X \qu G$ is non-compact, a minimal model program
  for $X \qu G$ may have no transitions.  For example, suppose that
  $X \qu G$ is the total space of $\mO_{\P^1}(-d)$.  The moment
  polytope of $X \qu G$ is
  $\Delta_{X \qu G} = \{ \mu_2 \ge 0, \mu_1 + \mu_2 \ge -d/2, - \mu_1
  + \mu_2 \ge -d/2 \}.$
  For $d = 2$, the resulting running of the minimal model program has
  $\Delta_{X \qu_t G}$ a family of translations of $\Delta_{X \qu G}$.
  For $d > 2$, the minimal model program running $X \qu_t G$
  corresponds to translation of the polytope $\Delta_{X \qu_t G}$
  together with a dilation $\t_\R^\dual \to \t_\R^\dual$ by a constant
  greater than $1$.
\end{remark}

\def\cprime{$'$} \def\cprime{$'$} \def\cprime{$'$} \def\cprime{$'$}
\def\cprime{$'$} \def\cprime{$'$}
\def\polhk#1{\setbox0=\hbox{#1}{\ooalign{\hidewidth
      \lower1.5ex\hbox{`}\hidewidth\crcr\unhbox0}}} \def\cprime{$'$}
\def\cprime{$'$}

\end{document}